\documentclass[11pt, a4paper, twoside]{scrartcl}

\usepackage[utf8]    {inputenc}
\usepackage[english] {babel}
\usepackage[T1]      {fontenc}

\usepackage{csquotes}
\MakeOuterQuote{"}

\usepackage{amsmath}
\usepackage{amssymb}
\usepackage{amsfonts}
\usepackage{amsxtra}
\usepackage{amsthm}
\usepackage{float}
\usepackage{mathrsfs}
\usepackage{mathtools}
\usepackage{thmtools}
\usepackage{thm-restate}

\usepackage{setspace}
\usepackage[
    automark,
    ]{scrlayer-scrpage}
\KOMAoptions{headsepline=true}
\AddToLayerPageStyleOptions{scrheadings}
    {oninit={
        \ifstr{\headmark}{}{\KOMAoptions{headsepline=false}}{}
    }}
\clearpairofpagestyles
\addtokomafont{pagefoot}{\upshape}
\ohead{\headmark}
\ofoot*{\thepage}

\let\nonClearingSection\section
\def\section{\cleardoublepage\nonClearingSection}

\usepackage{tikz}
\usetikzlibrary{calc}
\usetikzlibrary{matrix,arrows}

\usepackage{hyperref}

\theoremstyle{plain}
\newtheorem{thm}     {Theorem}[subsection]
\newtheorem{cor}[thm]{Corollary}
\newtheorem{pro}[thm]{Proposition}
\newtheorem{lem}[thm]{Lemma}
\theoremstyle{definition}
\newtheorem{dfn}[thm]{Definition}
\newtheorem{ntt}[thm]{Notation}
\newtheorem{exm}[thm]{Example}
\newtheorem{rem}[thm]{Remark}
\theoremstyle{remark}

\newcommand{\BeginSimonFigure}{
    \begin{figure}[H]
    \centering
}
\newcommand{\EndSimonFigure}[2]{
    \addtocounter{thm}{1}
    \captionof{figure}{#2}
    \label{#1}
    \end{figure}
}

\numberwithin{equation}{thm}

\newcommand{\proj}    {\mathnormal{\textup{pr}}}

\newcommand{\id}      {\mathnormal{\textup{id}}}
\newcommand{\incl}    {\mathnormal{\textup{incl}}}

\newcommand{\del}     {\partial}

\newcommand{\loto}[1] {\stackrel{#1}{\longrightarrow}}

\newcommand{\nat}    {\mathbb N}
\newcommand{\zet}    {\mathbb Z}

\newcommand{\rel}    {\mathbb R}
\newcommand{\com}    {\mathbb C}

\newcommand{\restr}{\restriction}

\newcommand{\cont}    {\mathnormal{\mathscr C}}
\newcommand{\conto}   {\mathnormal{\mathscr{C}_0}}

\newcommand{\dualalg} {\mathnormal{\mathfrak D^*}}
\newcommand{\cualalg} {\mathnormal{\mathfrak C^*}}
\newcommand{\qualalg} {\mathnormal{\mathfrak Q^*}}
\newcommand{\fualalg} {\mathnormal{\mathfrak F^*}}

\newcommand{\im}     {\mathop{\textup{im}}}

\newcommand{\supp}   {\mathop{\textup{supp}}}
\newcommand{\colim} {\mathop{\mathrm{colim}}}

\newcommand{\catc}       {\mathnormal{\mathrm{\underline{C}}}}
\newcommand{\catcstar}   {\mathnormal{\mathrm{\underline{C^*\!A}}}}

\newcommand{\starhom}   {$\ast$-ho\-mo\-mor\-phism}
\newcommand{\starhoms}  {$\ast$-ho\-mo\-mor\-phisms}
\newcommand{\norm}[1]   {{\left\|{#1}\right\|}}
\newcommand{\betr}[1]   {{\left|{#1}\right|}}

\newcommand{\varemptyset}{\varnothing}

\newcommand{\subs}{\subseteq}
\newcommand{\set}[2] {{\left\{ #1 : \, #2 \right\}}}
\newcommand{\card}[1] {\left|{#1}\right|}
\newcommand{\setton} {{\left\{0, 1, \ldots, n\right\}}}

\newcommand{\mama}     {\mamaxy{2em}{2.5em}}
\newcommand{\mamax}[1] {\mamaxy{#1}{2.5em}}
\newcommand{\mamaxy}[2]{\matrix (m) [matrix of math nodes, row sep=#2,  column sep=#1,  text height=2ex, text depth=0.6ex]}

\newcommand{\maeq}    {edge [double distance = 2.5pt]}

\newcommand{\masseq}{\matrix (m) [matrix of math nodes,
             nodes in empty cells,
             nodes={minimum width=5ex, minimum height=5ex,
             inner sep=3pt, outer sep=0pt, anchor=north},
             column sep=2ex, row sep=0ex]}

\newcommand{\drawSseqAxes}[2]{
    \draw[->,thick] (m-#2-1.north east) -- (m-1-1.north east) ;
    \draw[->,thick] (m-#2-1.north east) -- (m-#2-#1.north east) ;
    \draw ($(m-#2-#1.north east)+(0,-.3)$) node [above right] {$p.$};
    \draw ($(m-1-1.north east)+(-.25,.02)$) node [above right] {$q$};
}

\newcommand{\eiddot}[3]
{
    \draw[fill=black] (#1) circle (0.15 em) node[#2] {\tiny $#3$};
}

\newcommand{\textSchochetSseq}
{Let $A = \overline{\bigcup_{p\in\nat} I_p}$ be a C*-algebra, where the
    $I_0 \subseteq I_1 \subseteq I_2 \subseteq \dotsb \subseteq I_p \subseteq
    \dotsb$ form a chain of closed
    two-sided ideals. There is a spectral sequence $\{E^r_{p,q}, d^r\}_{r,p,q}$
    with
    \[
        E^1_{p,q} = K_{p+q}(I_p/I_{p-1}).
    \]
    This spectral sequence converges strongly to $K_* A$; i.e.,
    given $s \in \zet$,
    the groups $E^\infty_{p,q}$ along the diagonal $s = p+q$
    pose an extension problem to reconstruct $K_s A$.
}

\newcommand{\equationEOneFinite}
{
    E^1_{p,q} \cong
        \begin{dcases*}
            \bigoplus_{\card J = p+1}
                K_{q}\Big( \bigcap_{j \in J} I_j \Big)
            &
            for $0 \leq p \leq n$,
            \\
            0 & for $p < 0$ or $p > n$.
        \end{dcases*}
}

\newcommand{\equationsumofnfoldsuspensions}
{
    K_{p+q} \left(Q_p/Q_{p-1}\right)
        \cong \bigoplus_{\card J = p+1} K_{q+n}
        \Big( \bigcap_{j \in J} I_j \Big)
}

\newcommand{\equationDefineB}
{
    \begin{align*}
        B &= B(I_0, I_1, \ldots, I_n) \\
        &= \set{f\colon \Delta^n \to A = \sum_{j=0}^n I_j}
        {f \textrm{ continuous, }
        f \restr \partial \Delta^n = 0 \textrm{, } f(\Delta^n_j) \subseteq I_j
            \textrm{ for all } j}.
    \end{align*}
}

\newcommand{\equationDefineBJ}
{
    \[
        B_J = \set{f \in B}
        {\textrm{for each } j' \notin  J \textrm{, } f(\Delta^n_{j'}) = 0}.
    \]
}

\newcommand{\kTheoryOfRelN}
{
    \[
        K_s \cualalg \rel^n =
        \begin{cases}
            \zet & \textrm{for } s - n \textrm{ even,} \\
            0 & \textrm{for } s - n \textrm{ odd.}
      \end{cases}
    \]
}

\newcommand{\letFualalgBeEither}
{Let $\fualalg$ be either the functor
    $\cualalg$ from the coarse category to $\catcstar$ or one
    of the functors $\dualalg$ or $\qualalg$ from the coarse-continuous
    category to $\catcstar$.
}

\newcommand{\thmSseqCoarselyExcisiveCoversForFor}[2]
{
    \letFualalgBeEither
    There is a spectral sequence
    $\{E^r_{p,q}, d^r\}_{r,p,q}$ with
    \[
        E^1_{p,q} \cong
            \begin{dcases*}
                \bigoplus_{\card J = p+1}
                    K_q \fualalg \Big( \bigcap_{j \in J} X_j \Big)
                &
                #1
                \\
                0 &
                #2
            \end{dcases*}
    \]
}

\newcommand{\withRespectToCoarseOrCoarseContinuousMorphisms}
{with respect to morphisms (coarse maps for $\cualalg$,
    or coarse and continuous maps for $\dualalg$ and $\qualalg$)
    to other coarse spaces with compatible coarsely excisive covers
    (Definition \ref{dfn:compatibleCoarseCovers}).
}

\newcommand{\sseqUnendlich}
{
    E^1_{p,q} \cong
        \begin{dcases*}
            \bigoplus_{\card J = p+1} K_q
                \Big( \bigcap_{j \in J} I_j \Big)
            & for $p \geq 0$, \\
            0 & for $p < 0$,
        \end{dcases*}
}

\newcommand{\syntaxHighlightWorkaround}
{
        X_j =
        \begin{cases}
            {\left]-\infty, 0\right]}
                \times \rel^{n-1} & \textrm{for } j = 0, \\
            {\left[0, \infty \right[}^i
                \times {\left]-\infty, 0\right]}
                \times \rel^{n-j-1} & \textrm{for } 0 < j < n, \\
            {\left[0, \infty \right[}^n & \textrm{for } j = n. \\
        \end{cases}
}

\newcommand{\intCO}[2]{{\left[{#1},{#2}\right[}}
\newcommand{\intOC}[2]{{\left]{#1},{#2}\right]}}

\begin{document}
    \pagenumbering{roman}
    
\thispagestyle{empty}

\phantom{a}
\vspace{1 cm}

\begin{center}
    {\fontfamily{lmss}\selectfont\Huge
        A Mayer-Vietoris Spectral Sequence\\
        for C*-Algebras and Coarse Geometry\\}
    \vspace{3 cm}
    {\Large Dissertation\\}
    \vspace{0.5 cm}
    {zur Erlangung des mathematisch-naturwissenschaftlichen Doktorgrades\\}
    {„Doctor rerum naturalium“\\}
    {der Georg-August-Universit\"at G\"ottingen\\}
    \vspace{0.5 cm}
    {im Promotionsprogramm „School of Mathematical Sciences“\\}
    {der Georg-August University School of Science (GAUSS)\\}
    \vspace{5 cm}
    {vorgelegt von\\}
    {Simon Naarmann\\}
    {aus Lippstadt\\}
    \vspace{0.5 cm}
    {Göttingen, 2018\\}
\end{center}

    \newpage
    \thispagestyle{empty}
    
\newcommand{\ugoe}{Georg-August-Universität Göttingen}
\newcommand{\mthI}{Mathematisches Institut}
\newcommand{\namI}{Institut für Numerische und Angewandte Mathematik}
\newcommand{\vspI}{\vspace{0.3 cm}}
\newcommand{\pppI}[1]{
    \paragraph{#1}
    \mbox{}\\

    \vspace{-0.2 cm}
}

\pppI{Betreuungsausschuss}

\begin{tabular}{ll}
Erstbetreuer: & \textbf{Prof.\ Dr.\ Thomas Schick} \\
& \mthI{} \\ \vspI
& \ugoe{} \\
Zweitbetreuer: & \textbf{Prof.\ Dr.\ Ralf Meyer} \\
& \mthI{} \\
& \ugoe{} \\
\end{tabular}

\pppI{Mitglieder der Prüfungskommission}

\begin{tabular}{ll}
Referent: & \textbf{Prof.\ Dr.\ Thomas Schick} \\
& \mthI{} \\ \vspI
& \ugoe{} \\
Koreferent: & \textbf{Prof.\ Dr.\ Ralf Meyer} \\
& \mthI{} \\
& \ugoe{} \\
\end{tabular}

\pppI{Weitere Mitglieder der Prüfungskommission}

\begin{tabular}{ll}
\textbf{Prof.\ Dr.\ Gert Lube} \\
\namI{} \\ \vspI
\ugoe{} \\

\textbf{Prof.\ Dr.\ Thorsten Hohage} \\
\namI{} \\ \vspI
\ugoe{} \\

\textbf{Prof.\ Dr.\ Ingo Witt} \\
\mthI{} \\ \vspI
\ugoe{} \\

\textbf{Prof.\ Dr.\ Chenchang Zhu} \\
\mthI{} \\
\ugoe{} \\
\end{tabular}

\pppI{Tag der mündlichen Prüfung}
\vspace{-0.1 cm}

\begin{tabular}{l}
10. September 2018\\
\end{tabular}

    \newpage
    
\section*{Abstract}

Let $A$ be a C*-algebra that is the norm closure
$A = \overline{\sum_{\beta \in \alpha} I_\beta}$
of an arbitrary sum of C*-ideals $I_\beta \subs A$. We construct a homological
spectral sequence that takes as input the K-theory of
$\bigcap_{j \in J} I_j$ for all finite nonempty index sets $J \subs \alpha$
and converges strongly to the K-theory of $A$.

For a coarse space $X$, the Roe algebra $\cualalg X$ encodes large-scale
properties. Given a coarsely excisive cover
$\{X_\beta\}_{\beta \in \alpha}$ of $X$, we reshape
$\cualalg X_\beta$ as input for the spectral sequence.
From the K-theory of $\cualalg \big( \bigcap_{j \in J} X_j \big)$
for finite nonempty index sets $J \subs \alpha$, we compute
the K-theory of $\cualalg X$ if $\alpha$ is finite, or of a direct limit
C*-ideal of $\cualalg X$ if $\alpha$ is infinite.

Analogous spectral sequences exist for the algebra $\dualalg X$
of pseudocompact finite-propagation operators that contains the Roe algebra
as a C*-ideal, and for $\qualalg X = \dualalg X / \cualalg X$.

    \newpage
    
\section*{Acknowledgements}

I sincerely thank my main advisor Professor Dr.~Thomas Schick
for sparking my interest in coarse geometry, for suggesting
research goals that matched my background in algebraic topology, and
for his excellent supervision throughout my project.
He granted me the liberty to expand my research in
directions that interested me most.
I have felt deep mutual trust both as his doctorate student and as his
teaching assistant. I continue to respect his knowledge and intuition.

I thank my secondary advisor Professor Dr.~Ralf Meyer for his support and
his formidable strategic advice. I have known both Professor Dr.~Ralf Meyer
and Professor Dr.~Thomas Schick
since my Bachelor's and Master's studies; they laid the foundation for my
research.

I thank my parents Elianne and Bernd Naarmann and my brother Julian
for their patience and encouragement throughout this project.
From my childhood on, they have conveyed the importance of good education.

I thank my girlfriend Sylvia for her emotional support and her understanding.

I thank my friends and fellow mathematicians Thorsten and Mehran
for their unconditional enthusiasm and their inspiration.

My work was funded by the German Research Foundation (DFG) via
the Research Training Group 1493,
Mathematical Structures in Modern Quantum Physics,
from June 2013 through September 2013.
Afterwards,
I worked half-time as a scientific assistant
at the Mathematical Institute in Göttingen from October 2013 through
August 2017.
Finally, I received a DFG stipend
from November 2017 through December 2017.

I am grateful for all support and funding.

    \newpage
    \tableofcontents

    \newpage
    \pagenumbering{arabic}
    
\section{Introduction}\thispagestyle{plain}

\subsection{Our result for abstract C*-algebras}

C*-algebras arise in mathematics and theoretical physics alike.
K-theory is a fundamental tool to study and classify these highly-structured
algebras.

As a covariant $\zet$-graded functor of C*-algebras over $\com$,
K-theory is continuous, is additive,
admits suspension isomorphisms $K_s SA \cong K_{s+1} A$, and admits
a cyclic 6-term exact sequence induced by C*-ideal inclusions
$I \subs A$ via boundary maps and Bott periodicity
$\beta\colon K_s A \cong K_{s+2} A$.
Furthermore, for a sum $A = I_0 + I_1$ of two C*-ideals, there is a
Mayer-Vietoris exact sequence,
\[
    \begin{tikzpicture}
        \mama {
            K_0 (I_0 \cap I_1)
            & K_0 I_0 \oplus K_0 I_1
            & K_0 A \\
            K_1 A
            & K_1 I_0 \oplus K_1 I_1
            & K_1 (I_0 \cap I_1).
            \\
        };
        \path[->, font=\scriptsize]
            (m-1-1) edge (m-1-2)
            (m-1-2) edge (m-1-3)
            (m-1-3) edge node[right]{$\del_2^\mathrm{MV} \circ \beta$}
                (m-2-3)
            (m-2-3) edge (m-2-2)
            (m-2-2) edge (m-2-1)
            (m-2-1) edge node[left]{$\del_1^\mathrm{MV}$} (m-1-1)
        ;
    \end{tikzpicture}
\]
In algebraic topology, a similar Mayer-Vietoris sequence computes
the homology or cohomology of a space $X$ from a cover
$X = X_0^\circ \cup X_1^\circ$ by relating the (co)homology of $X_0$, $X_1$,
and $X_0 \cap X_1$.
This topological Mayer-Vietoris sequence generalizes to a spectral
sequence: Given a suitable cover $\{X_\beta\}_{\beta \in \alpha}$
of a topological space, the spectral sequence
takes as input the (co)homology of
$\bigcap_{j \in J} X_j$ for all finite nonempty $J \subs \alpha$
and converges to the (co)homology of $X$, the full space.

It is natural to seek
analogous spectral sequences for the
K-theory Mayer-Vietoris exact sequence.
This is our first main result:

\begin{restatable*}[Spectral sequence for arbitrary sums]{thm}
    {restateAbstractSseqUncountable}
\label{thm:uncountableAbstractSseq}
    Let $\alpha$ be an arbitrary index set: finite, countable, or uncountable.
    Let $A = \overline{\sum_{\beta \in \alpha} I_\beta}$
    be the norm closure of a sum of $\card{\alpha}$-many C*-ideals
    $I_\beta \subs A$.
    There is a spectral sequence
    $\{E^r_{p,q}, d^r\}_{r,p,q}$ with
    \[
        \sseqUnendlich
    \]
    where $J$ ranges over all nonempty finite index subsets $J \subs \alpha$.
    In general, this is a half-page spectral sequence, any term $E^1_{p,q}$
    with $p \geq 0$ may be nonzero.

    This spectral sequence converges strongly to $K_* A$.
    It is functorial with respect to \starhoms{} that preserve
    $\alpha$-indexed ideal decompositions.
\end{restatable*}

To prove this, we begin with the finite case $A = I_0 + I_1 + \dotsb + I_n$
and construct C*-algebras of continuous functions from the standard
simplex $\Delta^n$ to $A$, interlocking allowed ranges in the $I_j \subs A$
on different regions of the simplex.

Sums of these function algebras become a chain of ideals
$Q_0 \subs Q_1 \subs \dotsb \subs Q_n$.
This chain of ideals fits into an already-known spectral sequence that
converges strongly to $K_* Q_n \cong K_*(S^n A)$,
the K-theory of the $n$-fold suspension of $A$.
This spectral sequence for ideal
inclusions has been described by C. Schochet in \cite{schochet-sseq};
we reprove it to highlight its inner workings, its differentials, and
its filtration that guarantees strong convergence.
Compared to that spectral sequence, our
Theorem \ref{thm:uncountableAbstractSseq} relaxes the input conditions:
We do not require that the $I_j$ form a chain of inclusions.

For countable index sets $\alpha$, we link the spectral sequences for
$n$ ideals and $n+1$ ideals
-- this is only possible on the level of K-theory, not on the level of
C*-algebras -- and construct a filtration on the spectral sequence via a
suitable direct limit.
For uncountable sets $\alpha$, we adapt our direct limit construction
to the directed system of finite subsets of $\alpha$.

\subsection{Our application in coarse geometry}

Coarse geometry studies the large-scale structure of metric spaces. If two
spaces differ only within a compact set, coarse invariants will not detect
any difference.

For a coarse space $X$, i.e., a metric space $(X, d)$,
the Roe algebra $\cualalg X$ encodes such large-scale
properties. This C*-algebra and the larger algebra $\dualalg X$ are
introduced, e.g., by N. Higson and J. Roe in \cite{higson2000analytic},
or by J. Roe in \cite{roe-book}. Via these algebras, the
K-homology of $X$ and further invariants of contemporary research are defined,
such as the coarse index when $X$ is a Riemannian manifold. For our work,
it suffices to define $\qualalg X = \dualalg X / \cualalg X$; we will not
formulate our results in the language of K-homology.

For certain sets $X_0$ and $X_1$ with $X_0 \cup X_1 = X$, there is a coarse
Mayer-Vietoris exact sequence: It relates the K-theory of $\cualalg X_0$,
$\cualalg X_1$, and $\cualalg(X_0 \cap X_1)$ to the K-theory of $\cualalg X$.
Its proof, e.g., in \cite{roe-book},
relies on the Mayer-Vietoris sequence for two abstract C*-ideals.

We generalize to arbitrarily many regions.
A decomposition $X = \bigcup_{\beta \in \alpha} X_\beta$
is called \emph{coarsely excisive} if, for all nonempty finite
subcollections $J \subs \alpha$ and $R > 0$, there exists $S > 0$ such
that the intersection of the $R$-neighborhoods is contained in the
$S$-neighborhood of the intersection according to the metric $d$ on $X$:
\[
    \bigcap_{j \in J} N_d(X_j, R) \subseteq
    N_d\Big(\bigcap_{j \in J} X_j, S\Big).
\]
This leads to our second main result:

\begin{restatable*}[Spectral sequence for coarsely excisive covers]
    {thm}{restatSseqCoarseUncountable}
\label{thm:sseqCoarseUncountable}
    Let $(X, d)$ be a coarse space and
    let $\{X_\beta\}_{\beta \in \alpha}$
    be a coarsely excisive cover of $(X, d)$.
    \thmSseqCoarselyExcisiveCoversForFor{
        for $p \geq 0$,}{
        for $p < 0$,
    }
    where $J$ ranges over all nonempty finite subcollections of indices
    in $\alpha$. For finite $\alpha$, this spectral sequence converges
    strongly to $K_* \fualalg X$. In general, the spectral
    sequence converges strongly to the K-theory of
    $\overline{\bigcup_J \sum_{j \in J} \fualalg(X_j \subs X)}$,
    a C*-ideal of $\fualalg X$,
    where $J$ ranges over all finite subcollections of indices in $\alpha$.
    The spectral sequence is functorial
    \withRespectToCoarseOrCoarseContinuousMorphisms
\end{restatable*}

As an example,
we recompute the known K-theory of
$\cualalg \rel^n$.
Also, we find an infinite coarsely excisive cover
of $\zet^\infty$ under many metrics, then show that the K-theory of the
direct limit ideal of Roe algebras for this cover in
$\cualalg \zet^\infty$ vanishes.
Furthermore, we compute
the K-theory of the direct limit of Roe algebras for a countable wedge sum
$\bigvee_{\nat} \intCO{0}{\infty}$ in a single application
of the spectral sequence; this obviates inductive proofs with
the Mayer-Vietoris exact sequence.

\subsection{Relations to other research}

Early motivation for this project was the
Partitioned Manifold Index Theorem
\cite[Proposition 4.9]{siegel-mv}:
Given certain Riemannian manifolds $N \subs M$ with
$G$-equi\-variant covers, the classes of their Dirac operators -- elements in
K-homology -- map to the same element in $K_* C^*_r G$, the K-theory of
the group C*-algebra for $G$, via the coarse index maps. P. Siegel proves
this by induction with the coarse Mayer-Vietoris principle for two regions.
Our idea was to reprove this theorem by a single application
of our Theorem \ref{thm:sseqCoarseUncountable}.

Spectral sequences, however, do not construct specific maps on their targets;
even strong convergence only leads to isomorphism theorems.
Still, if the spectral sequence cannot compute the equality
for the Partitioned Manifold Index Theorem,
it can classify related C*-algebras for coarse spaces in this setting
and decide about the structure of possible morphisms between them.

Sums or inclusion chains of abstract C*-ideals also arise in other settings.
In \cite{meyer-mukherjee}, D. Mukherjee and R. Meyer construct
a gauge-invariant C*-algebra $\mathcal T_0$ of the Toeplitz algebra
$\mathcal T$ for partial product systems.
By \cite[Theorem 4.5]{meyer-mukherjee}, $\mathcal T_0$ is a direct limit
along $N \in \nat$ of images of maps from $\bigoplus_{n < N} K \mathcal{E}_n$
into $\mathcal T_0$ where the
$\mathcal E_n$
are the compact opreators of correspondences.
These images are C*-subalgebras of $\mathcal T_0$. With
our spectral sequences, to compute the K-theory of $\mathcal T_0$,
we may examine the K-theory of quotients or intersections of these
subalgebras.

\subsection{Structure of this thesis}

Section 2, Fundamentals, establishes the notation and gives an overview of
the basic tools. The reader will likely be familiar with several constructions.
For a metric space $X$, we present the coarse algebras
$\cualalg X$, $\dualalg X$, and $\qualalg X = \dualalg X / \cualalg X$.

In Section 3, Ideal inclusions, we show a spectral sequence that takes
a chain of C*-ideal inclusions
$I_0 \subs I_1 \subs I_2 \subs \dotsb \subs I_p \subs \dotsb$
across all $p \in \nat$. Similar to the material in Section 2, this spectral
sequence is known theory. Nonetheless, we reprove it in detail. This
will be the basis for the research work in the following sections.

In Section 4, Finite sums of ideals, we construct our spectral sequence
that takes intersections $\bigcap_{j \in J} I_j$ of C*-ideals $I_j$
from a finite sum $A = I_0 + I_1 + \dotsb + I_n$.

In Section 5, Finite coarse excision, we define coarsely excisive covers
and relative coarse algebras. For finite subcollections of a coarsely excisive
cover, we show that these relative algebras behave well under intersections
and unions. This leads to a version of our spectral sequence for finite
coarsely excisive covers.
As an example, we recompute the K-theory of $\cualalg \rel^n$.

In Sections 6 and 7, Infinite sums of ideals and
Infinite coarse excision,
we relax the condition that $A = I_0 + I_1 + \dotsb + I_n$ needs to
be a finite sum: Now $A = \overline{\sum_{\beta \in \alpha} I_\beta}$
may be a direct limit of
sums of arbitrarily many C*-ideals. The spectral sequences from Sections 4
and 5 first generalize to countable ideal decompositions as input
instead of only finite decompositions, then to uncountable decompositions.
As examples, we compute the K-theory of direct limits of Roe algebras
for $\zet^\infty$ and $\bigvee_{\nat} \intCO{0}{\infty}$.

In Section 8, Generalizations, we list ideas
for real KO-theory
and equivariant spaces.

\section{Fundamentals}\thispagestyle{plain}
\label{sec:foundations}

We will rehearse well-known constructions
to establish notation and conventions, beginning with C*-algebras
and their K-theory. We introduce the basics of coarse geometry and
spectral sequences.

The set $\nat$ of natural numbers includes $0$.

\subsection{C*-algebras}

\begin{dfn}[Complex Banach algebra]
    Let $(A, \left\|-\right\|)$
    be a normed associative algebra over $\com$ that is topologically
    complete according to its norm.
    For all $x$, $y \in A$, the following inequality shall hold:
    $\left\|xy\right\| \leq \left\|x\right\| \left\|y\right\|$.
    Then we call $A$ a \emph{complex Banach algebra}.
\end{dfn}

\begin{dfn}[C*-algebra, C*-ideal]
    Let $A$ be a complex Banach algebra.
    Let $A$ carry an involution $\ast$, i.e., a map $A \to A$, $x \mapsto x^*$,
    satisfying $x^{**} = x$,
    $(\lambda x + y)^* = \overline \lambda x^* + y^*$,
    and $(xy)^* = y^* x^*$ for all
    $x$, $y \in A$ and $\lambda \in \com$.
    Furthermore, $\ast$ shall satisfy the \emph{C*-identity}
    $\left\|x x^*\right\| = \left\|x\right\|^2$.
    Then $(A, *)$ is a \emph{C*-algebra}.

    A closed two-sided ideal in a C*-algebra is called a \emph{C*-ideal}.
\end{dfn}

\begin{rem}
    The zero algebra $0$ is a C*-algebra. The complex numbers $\com$ themselves
    are a C*-algebra with
    $z \mapsto |z|$ as the norm and $z \mapsto \overline z$ as the
    $\ast$-operation.

    For a C*-algebra $A$ with a C*-ideal $I \subseteq A$, the quotient
    $A/I$ is a well-defined C*-algebra.

    Let $\alpha$ be an arbitrary index set and let $A_\beta$
    be a C*-algebra for each $\beta \in \alpha$.
    The norm completion
    $\overline{\bigoplus_{\beta \in \alpha} A_\beta}$
    of the algebraic direct sum
    is again a C*-algebra with the
    norm $\|(a_\beta)_{\beta \in \alpha}\|
        = \sup \set{ \left\|a_{\beta}\right\| }{\beta \in \alpha}$ and
    component-wise addition, multiplication, and involution.
\end{rem}

\begin{dfn}[Category $\catcstar$ of C*-algebras]
    A bounded algebra homomorphism
    $f\colon A \to B$ between two C*-algebras is called
    a \emph{\starhom{}} if it preserves the involution:
    $f(x^*) = f(x)^*$ shall
    hold for all $x \in A$.

    The category $\catcstar$ encompasses all C*-algebras as objects,
    together with all \starhoms{} as arrows.
\end{dfn}

\begin{dfn}[Homotopy of C*-algebras]
    Let $f$, $g\colon A \to B$ be \starhoms{} between C*-algebras.
    A \emph{homotopy} in the category of C*-algebras between $f$ and $g$
    is a map $H\colon (A \times [0,1]) \to B$ that satisfies:
    \begin{itemize}
        \item For all $a \in A$, $H(a, 0) = f(a)$ and $H(a, 1) = g(a)$.
        \item For all $a \in A$, the map $[0,1] \to B$, $t \mapsto H(a, t)$
            is continuous.
        \item For all $t \in [0,1]$, the map $A \to B$, $a \mapsto H(a, t)$
            is a \starhom.
    \end{itemize}
    If such a homotopy exists, then $f$ and $g$ are called \emph{homotopic},
    denoted by $f \sim g$.

    A \starhom{} $f\colon A \to B$ is a
    \emph{homotopy equivalence} if there exists a \starhom{} $g\colon B \to A$
    such that $g \circ f \sim \id(A)$ and $f \circ g \sim \id(B)$.
    Existence of a homotopy equivalence $A \to B$ is denoted by
    $A \simeq B$; then $A$ is called \emph{homotopy equivalent} to $B$.

    A C*-algebra $A$ is called \emph{contractible} if $A \simeq 0$.
\end{dfn}

Homotopy equivalence as C*-algebras is a stronger condition than
topological homotopy equivalence.

\begin{lem}
\label{lem:noContractibleProjections}
    Let $A$ be a contractible C*-algebra and
    $p = p^* = p^2$ a projection in $A$.
    Then $p = 0$.
\end{lem}

\begin{proof}
    Let
    $H\colon A \times [0,1] \to A$ be the contracting homotopy.
    For all $t \in [0,1]$, we have
    $\norm{H(p, t)} = \norm{H(pp^*, t)} = \norm{H(p,t)}^2 \in \rel_{\geq 0}$
    because $x \mapsto H(x, t)$ is a \starhom{}.
    This implies $\norm{H(p,t)} \in \{0,1\}$ and, because of continuity,
    this norm stays constant across all $t \in [0,1]$.
    By construction, $H(p, 1) = 0$, thus $\norm{H(p, 0)} = 0$ and $p = 0$.
\end{proof}

\begin{cor}
    The C*-algebra $\com$ is contractible as a topological space,
    but not contractible as a C*-algebra.
\end{cor}

\begin{proof}
    One possible topological homotopy is $(z, t) \mapsto tz$.
    Lemma \ref{lem:noContractibleProjections}
    precludes a C*-contraction because
    $1 = \overline 1 = 1 \cdot 1$ is a nonzero projection in $\com$.
\end{proof}

\begin{dfn}[Cone]
    Let $A$ be a C*-algebra.
    The \emph{cone of $A$} is the C*-algebra
    \[
        CA =
        \set{f\colon [0,1] \to A}
        {f \textrm{ is continuous, } f(0) = 0}.
    \]
    It carries the uniform norm
    $\left\|f\right\| = \sup \set{\left\|f(x)\right\|}{x \in [0,1]}$;
    this is well-defined because $[0,1]$ is compact.
    Algebra multiplication on $CA$ is given by pointwise multiplication
    of functions. The involution on the cone
    is defined by $f^*(x) = f(x)^*$.

    Let $g\colon A \to B$ be a \starhom{}. The \emph{cone map}
    $Cg\colon CA \to CB$ is given by
    $(Cg)(f) = g \circ f\colon [0,1] \to B$. This construction turns
    $C\colon \catcstar \to \catcstar$ into a covariant functor.
\end{dfn}

This is still basic theory of C*-algebras, but it is reasonable to agree on
whether the functions $f$ in the cone must vanish at 0 or vanish at 1.
In later sections, we will construct a spectral sequence for C*-algebras;
some technical lemmas require cones of algebras.

\begin{pro}
\label{pro:coneiscontractible}
    Let $A$ be a C*-algebra. Then the cone $CA$ is contractible; the
    zero algebra is a strong deformation retract of $CA$.
\end{pro}

\begin{proof}
    Define a homotopy $H\colon CA \times [0,1] \to CA$ by
    $H(f, t)(x) = f(tx)$.
    This is continuous because $f$ is continuous. It
    is the desired deformation retraction because
    $H(f, 0) = 0$ and $H(f, 1) = f$ for all $f \in CA$.
    Since $H(0, t)(x) = 0$ for all $t$ and
    $x \in [0,1]$, it is even a strong deformation retraction.
\end{proof}

\begin{dfn}[Suspension]
    Let $A$ be a C*-algebra and $CA$ its cone. The \emph{suspension of $A$}
    is the subalgebra
    \[
        SA = \set{f \in CA}{f(1) = 0}.
    \]
    The suspension $SA$ inherits its C*-algebra structure from $CA$.
    Likewise, $*$-homo\-mor\-phisms $\varphi\colon A \to B$ induce
    $S\varphi\colon SA \to SB$ via $S\varphi = (C\varphi) \restr SA$,
    making $S\colon \catcstar \to \catcstar$ another covariant functor.
\end{dfn}

    \subsection{C*-algebras for spaces}
\label{ssec:spacealgebras}

\begin{dfn}[$\cont(X,A)$, $\cont X$]
    Let $X$ be a compact Hausdorff space and $A$ a C*-algebra.
    Then $\cont(X, A)$ denotes the C*-algebra
    of $A$-valued continuous functions on $X$ with the sup-norm
    $\|f\|_{\cont(X,A)} = \sup\set{\|f(x)\|_A}{x \in X}$,
    pointwise addition and multiplication,
    and $f^*(x) = {f(x)}^*$.
    If $A$ is unital, $\cont(X, A)$ contains the constant function
    that maps all points in $X$ to $1 \in A$; this function is then a unit.

    Often, $A = \com$;
    we abbreviate by setting $\cont X = \cont(X, \com)$.
\end{dfn}

When $X$ fails to be compact, $\cont X$ is not a normed algebra because some
$A$-valued functions on $X$ are unbounded. More interesting
function algebras impose boundedness:

\begin{dfn}[$\conto(X,A)$, $\conto X$]
    Let $X$ be a locally compact Hausdorff space, $A$ a C*-algebra.
    A continuous function $f \colon X \to A$ \emph{vanishes at infinity}
    if, for all $\varepsilon > 0$, there exists a compact set $K \subseteq X$
    with
    \[
        \set{x \in X}{\left\|f(x)\right\| > \varepsilon} \subseteq K.
    \]
    The set of all such functions $f$ is denoted $\conto(X, A)$.
    In the common case $A = \com$,
    we shall write $\conto X = \conto(X, \com)$.

    Again, $\conto(X, A)$ carries a C*-algebra structure under
    pointwise multiplication and the sup-norm
    $\norm{f} = \sup \set{\norm{f(x)}}{x \in X}$.
\end{dfn}

\begin{rem}
    The sup-norm is well-defined because functions in $\conto(X, A)$
    are necessarily bounded.
    The C*-algebra $\conto(X, A)$ has a unit if and only if $X$ is
    compact and $A$ is unital.
    For compact $X$, the algebra $\conto(X, A)$ coincides with $\cont(X, A)$.

    Equivalent definitions of $\conto(X, A)$
    embed $X$ into an arbitrary compactification $Y$, then define
    $\conto(X, A)$ as the subset of all continuous functions $f\colon Y \to A$
    such that $f \restr (Y - X) = 0$, then restrict these functions to $X$.

    Taking $A$-valued
    functions that vanish at infinity is a contravariant functor from Hausdorff
    spaces with proper continuous maps into $\catcstar$:
    Let $X$ and $Y$ be Hausdorff spaces and let $f\colon X \to Y$ be a proper
    continuous map. Then $(- \circ f)\colon \conto(Y, A) \to \conto(X, A)$
    maps $g\colon Y \to A$ to $g \circ f\colon X \to A$. The composition
    $g \circ f$ vanishes at infinity because $f$ is proper.
\end{rem}

    \subsection{K-theory of C*-algebras}

The exact constructions of the K-theory $K_* A$ for a C*-algebra $A$
are lengthy and shall be omitted; several textbooks,
e.g., \cite{weggeolsen-93} or \cite{rordam}, cover all technical details.
The \emph{zeroth K-theory} group $K_0 A$ is the
Grothendieck group of equivalence classes of projections in a
ring of matrices over $A$ modulo homotopy equivalence.
The {first K-theory} group $K_1 A$ results from a similar construction
with unitary elements of the matrix ring instead of projections.

Both $K_0$ and $K_1$ become continuous covariant functors from
$\catcstar$ to abelian groups: For a morphism $f\colon A \to A'$, the
resulting morphism $K_*f\colon K_*A \to K_*A'$ applies $f$ to all matrix
entries before taking equivalence classes.

\begin{thm}[Suspension isomorphism]
    For a C*-algebra $A$, there is an isomorphism
    $\sigma\colon K_0 SA \to K_1 A$. This allows $\nat$-graded
    K-theory by defining $K_s A$ as $K_{s-1} SA$ inductively; some authors
    even define $K_1 A$ this way instead of via unitary matrix elements.
\end{thm}

\begin{thm}[Bott isomorphism]
    For all $s \in \nat$,
    there are \emph{Bott isomorphisms} $\beta\colon K_s A \to K_{s+2} A$.
    This allows $\zet$-graded K-theory by defining
    $K_s A = K_{s+2} A$ inductively for all $s < 0$.
\end{thm}

\begin{thm}[Six-term exact sequence]
\label{thm:sixTermExactSequence}
    Let $I \subseteq A$ be a C*-ideal. For all $s \in \zet$, K-theory
    admits boundary maps
    $\del_s\colon K_s (A/I) \to K_{s-1} I$
    that make the following six-term
    sequence exact; the horizontal arrows are induced by ideal inclusion and
    projection:
    \[
        \begin{tikzpicture}
            \mama {
                K_0 I
                & K_0 A
                & K_0 (A/I) \\
                K_1 (A/I)
                & K_1 A
                & K_1 I.
                \\
            };
            \path[->, font=\scriptsize]
                (m-1-1) edge (m-1-2)
                (m-1-2) edge (m-1-3)
                (m-1-3) edge node[right]{$\del_2 \circ \beta$} (m-2-3)
                (m-2-3) edge (m-2-2)
                (m-2-2) edge (m-2-1)
                (m-2-1) edge node[left]{$\del_1$} (m-1-1)
            ;
        \end{tikzpicture}
    \]
\end{thm}

\begin{thm}[Abstract Mayer-Vietoris exact sequence]
\label{thm:abstractMayerVietoris}
    Let $A$ be a C*-algebra such that $I_0$, $I_1 \subseteq A$ are
    two C*-ideals with
    $I_0 + I_1 = A$. There is an exact sequence with Mayer-Vietoris
    boundary morphisms:
    \[
        \begin{tikzpicture}
            \mama {
                K_0 (I_0 \cap I_1)
                & K_0 I_0 \oplus K_0 I_1
                & K_0 A \\
                K_1 A
                & K_1 I_0 \oplus K_1 I_1
                & K_1 (I_0 \cap I_1).
                \\
            };
            \path[->, font=\scriptsize]
                (m-1-1) edge (m-1-2)
                (m-1-2) edge (m-1-3)
                (m-1-3) edge node[right]{$\del_2^\mathrm{MV} \circ \beta$}
                    (m-2-3)
                (m-2-3) edge (m-2-2)
                (m-2-2) edge (m-2-1)
                (m-2-1) edge node[left]{$\del_1^\mathrm{MV}$} (m-1-1)
            ;
        \end{tikzpicture}
    \]
\end{thm}

    \subsection{Roe algebras}
\label{subsection:dualalg}

\begin{dfn}[Ample representation]
\label{dfn:ample}
    Let $A$ be a separable C*-algebra.
    Let $\varrho\colon A \to BH$ be a
    representation of C*-algebras on a separable
    Hilbert space $H$, where $BH$ denotes the C*-algebra of all
    bounded linear operators $H \to H$.
    Then $\varrho$ is called \emph{ample} if
    \begin{itemize}
        \item $\varrho$ is nondegenerate, and
        \item $\varrho(0) = 0$ is the only compact operator in $\im(\varrho)
            \subseteq BH$.
    \end{itemize}
\end{dfn}

\begin{dfn}[Very ample representation]
\label{dfn:veryAmple}
    Let $A$ be a separable C*-algebra. A representation
    $\varrho\colon A \to BH$ of C*-algebras is called \emph{very ample}
    if it is a countably infinite sum of ample representations.
\end{dfn}

\begin{rem}
    To admit ample representations, the Hilbert space $H$ must be
    both separable and infinite-dimensional.
    Then suitable ample representations $\varrho$ always exist.
    According to \cite{higson2000analytic}, because $H$ is separable,
    the constructions in Section \ref{subsection:dualalg} do not depend
    on the particular choice of $H$ or $\varrho$ up to isomorphy.

    Every very ample representation is ample.
    Most constructions require ample representations. Some isomorphism
    theorems call for very ample representations,
    but, because $H \cong \bigoplus_\nat H$,
    requiring very ample representations is merely a technical convenience,
    not a fundamental restriction.
\end{rem}

\begin{dfn}[Pseudolocal operator]
    Let $X$ be a locally compact Hausdorff space.
    For the C*-algebra $\conto X$, let $\varrho\colon \conto X \to BH$
    be an ample representation.
    Let $T \in BH$ be an operator such that $\varrho(f)T - T\varrho(f)$
    is a compact operator in $BH$ for all $f \in \conto X$.
    Then $T$ is called \emph{pseudolocal}.
\end{dfn}

\begin{dfn}[Finite propagation]
    Let $(X, d)$ be a locally compact metric space and
    $\varrho\colon \conto X \to BH$ an ample representation.
    An operator $T \in BH$ has \emph{finite propa\-gation} if there exists
    a constant $R > 0$ such that for all $f$, $g \in \conto X$ with
    $d(\supp f, \supp g) \geq R$, the product $\varrho(f)T\varrho(g) \in BH$
    is zero.
\end{dfn}

\begin{dfn}[$\dualalg A$]
    Let $(X, d)$ be a locally compact metric space.
    Fix an ample representation $\varrho\colon \conto X \to BH$.
    The norm closure of the set of all pseudolocal operators in $BH$ with
    finite propagation forms a C*-algebra, denoted by $\dualalg X$.
\end{dfn}

\begin{rem}
    The norm closure turns $\dualalg X$ into a sub-C*-algebra of $BH$.
    Without the norm closure, the operators with finite propagation do not
    form a closed subset. Pseudolocal operators by themselves
    already form a sub-C*-algebra in $BH$ without additional closure.
\end{rem}

\begin{dfn}[Locally compact operator]
    Let $(X, d)$ be a locally compact metric space and
    $\varrho\colon \conto X \to BH$ an ample representation.
    Let $T \in BH$ be an operator such that, for all $f \in \conto X$,
    both $\varrho(f)T$ and $T\varrho(f)$ are compact operators in $BH$.
    Then $T$ is called \emph{locally compact}.
\end{dfn}

\begin{rem}
    Given $\varrho\colon \conto X \to BH$ ample for a locally compact
    metric space $(X, d)$,
    the locally compact operators form a C*-ideal in the algebra of
    pseudolocal operators.
\end{rem}

\begin{dfn}[$\cualalg X$, Roe algebra]
    For a locally compact metric space $(X, d)$ and an ample representation
    $\varrho\colon \conto X \to BH$,
    the \emph{translation algebra} or \emph{Roe algebra} $\cualalg X$
    is the norm closure of the operators $T \in BH$ that are both
    locally compact and have finite propagation.
\end{dfn}

\begin{rem}
    The Roe algebra $\cualalg X$ is a C*-ideal in $\dualalg X$.
\end{rem}

\begin{rem}[K-homology]
\label{rem:kHomology}
    Let $A$ be a C*-algebra and $A^+$ the C*-algebra with a unit
    adjoined. It is possible to define an abstract dual algebra $\dualalg A^+$
    by representing $A$ amply and taking all pseudolocal operators,
    without defining finite propagation.
    For $s \in \zet$, we may define
    the \emph{$s$-th K-homology group of $A$} as
    \[
        K^s A = K_{1-s} \dualalg A^+.
    \]
    K-theory of C*-algebras is a covariant functor;
    K-homology becomes a contravariant functor of C*-algebras.
    Were we concerned only with abstract C*-algebras, we could consider
    \enquote{K-homology} a bad name for a contravariant functor and to rename
    it to \enquote{K-cotheory}. But K-homology becomes a
    covariant functor for topological spaces:

    Let $X$ be a locally compact metric space and $s \in \zet$.
    The abelian group
    \[
        K^s X = K_{-s} \conto X
    \]
    is the \emph{K-homology of the space $X$};
    this defines a covariant functor from
    locally compact metric spaces to
    abelian groups.
    By \cite[Lemma 12.3.2]{higson2000analytic}, there is an isomorphism
    $K^s X = K_{s+1}(\dualalg X / \cualalg X)$.
\end{rem}

We will not need K-homology and will instead formulate all results in
the language of K-theory and Roe algebras. Thus we introduce a notation
similar to \cite{siegel-mv}:

\begin{ntt}[$\qualalg X$]
\label{ntt:qualalg}
    Let $X$ be a locally compact metric space. We write
    \[
        \qualalg X = \dualalg X / \cualalg X.
    \]
\end{ntt}

    \subsection{Coarse spaces}

The most general definition of a coarse space $X$ uses \emph{entourages}
or \emph{controlled sets} -- collections of subsets of $X \times X$
with axioms to capture a notion of closeness.
Following \cite[Chapter 2]{roe-book},
we will instead work with
proper metric spaces, a modest restriction. If our spaces are manifolds,
both methods bring the same results.

\begin{dfn}[Coarse space]
    A \emph{coarse space} $X = (X, d)$ is a proper metric space; i.e.,
    a metric space where closed $d$-bounded sets are compact.
\end{dfn}

\begin{dfn}[Coarse map]
\label{dfn:coarseMap}
    Let $f\colon (X, d_X) \to (Y, d_Y)$ be a map between coarse spaces.
    $f$ is called \emph{coarse} if
    \begin{itemize}
    \item $f$ is uniformly expansive: For $R > 0$, there exists $S > 0$
        such that for all $x$, $x' \in X$ with
        $d_X(x, x') \leq R$, we have $d_Y(fx, fx') \leq S$.
    \item $f$ is proper as a map between the metric spaces $X$ and $Y$:
        For each bounded set $B \subseteq Y$, the preimage $f^{-1}(B)$
        is bounded in $X$.
    \end{itemize}
\end{dfn}

Coarse maps are not required to be continuous.

\begin{rem}
    The identity
    $\id(X)\colon (X, d) \to (X, d)$ is coarse.
    Compositions of coarse maps are coarse.
\end{rem}

\begin{dfn}[Coarse category, coarse-continuous category]
\label{dfn:coarseCategory}
    The \emph{coarse category} has as objects all coarse spaces
    and as morphisms all coarse maps.

    The \emph{coarse-continuous category} is the subcategory of the coarse
    category that still comprises all coarse spaces, but that has as morphisms
    only the coarse maps that are also continuous.
\end{dfn}

\begin{dfn}[Closeness]
    Let $f$, $f'\colon (X, d_X) \to (Y, d_Y)$
    be two maps between coarse spaces.
    We call $f$
    \emph{close to $f'$}, or \emph{coarsely equivalent to $f'$}, if
    there exists $S > 0$ such that for all $x \in X$, we have
    $d_Y(fx, f'x) \leq S$.
\end{dfn}

\begin{dfn}[Coarse equivalence]
\label{dfn:coarseEquivalence}
    Let $X$ and $Y$ be coarse spaces with coarse maps $f\colon X \to Y$
    and $g\colon Y \to X$ such that $g \circ f$ is close to $\id(X)$ and
    $f \circ g$ is close to $\id(Y)$. We call $X$, $Y$ \emph{coarsely
    equivalent} and $f$, $g$ \emph{coarse equivalences}.
\end{dfn}

\begin{exm}
    Fix $n \in \nat$. The lattice $\zet^n$ is coarsely equivalent to
    Euclidean space $\rel^n$ under the metric
    $d_\infty$ with $d_\infty(x, x') = \sup_{j < n} | x_j - x'_j |$
    on both spaces.
    The inclusion $f\colon \zet^n \to \rel^n$ and
    \[
        g\colon \rel^n \to \zet^n, \qquad
        g(x_0, x_1, \ldots, x_{n-1})
        = (\lfloor x_0 \rfloor, \lfloor x_1 \rfloor, \ldots,
        \lfloor x_{n-1} \rfloor)
    \]
    serve as coarse equivalences.
    For all $z \in \zet^n$ and $x \in \rel^n$, the
    distances $d_\infty(z, gfz)$ and $d_\infty(x, fgx)$
    are uniformly bounded by the constant $1$.

    This map $g$ is proper, but it is not continuous.
\end{exm}

Coarse equivalences induce isomorphisms on
the K-theory of Roe algebras:

\begin{lem}[{\cite[Lemma 3.5]{roe-book}}]
    Let $X$, $Y$ be coarse spaces, $f\colon X \to Y$ a coarse map.
    Then $f$ induces a functorial homomorphism $f_*\colon K_*\cualalg X
    \to K_* \cualalg Y$. Coarsely equivalent maps induce the same homomorphism.
\end{lem}

\begin{rem}
    Similarly, the constructions $\dualalg$ and $\qualalg$ are
    functorial, but these functors are merely well-defined on the
    coarse-continuous category.
    Only $\cualalg$ is well-defined for coarse non-continuous maps.

    All three of $\cualalg$, $\dualalg$, and $\qualalg$
    are covariant functors to $\catcstar$: Passing from spaces $X$ and $Y$
    to function algebras $\conto X$ and
    $\conto Y$ is contravariant, and passing from function algebras to
    locally compact or pseudocompact operators with finite propagation
    is again contravariant.
\end{rem}

\begin{cor}
    Let $f\colon X \to Y$ and $g\colon Y \to X$ be coarse equivalences.
    Then $K_p \cualalg X \cong K_p \cualalg Y$ for all $p \in \zet$.
\end{cor}

\begin{proof}
    The compositions $g \circ f$ and $f \circ g$ are close to the identities
    on $X$ and $Y$. They induce identities in K-theory,
    thus both $K_p \cualalg f$ and $K_p \cualalg g$ are isomorphisms.
\end{proof}

\subsection{Coarsely excisive pairs}

We will recapitulate \emph{coarse excision} as defined by J. Roe
in \cite{roe-book}. Later, we will define \emph{coarsely excisive covers}
to generalize this idea.

\begin{dfn}[$R$-neighborhood]
\label{dfn:rneighborhood}
    Let $(X, d)$ be a metric space and $Y \subseteq X$ a subspace.
    For a real number $R > 0$, define the \emph{$R$-neighborhood of $Y$}
    as
    \[
        N_{d}(Y, R) = \set{x \in X}{\inf
            \set{d(x, y)}{y \in Y} \leq R}.
    \]
    When $d$ is a standard metric such as the $1$-metric $d_1$, the Euclidean
    metric $d_2$, or the sup-metric $d_\infty$ on $\rel^n$ or $\zet^n$,
    we will also write $N_1 = N_{d_1}$ or, similarly, $N_2$ or $N_\infty$.
\end{dfn}

\begin{dfn}[Coarsely excisive pair]
\label{dfn:coarselyExcisivePair}
    Let $(X, d)$ be a metric space. Let $U$ and $V$ be subspaces of $X$ with
    $U \cup V = X$.

    The pair $(U, V)$ is called a \emph{coarsely excisive pair for $X$} if,
    for every distance $R > 0$, there exists a distance $S > 0$ such that
    the intersection of the $R$-neighborhoods is
    contained in the $S$-neighborhood of the intersection:
    \[
        N_d(U, R) \cap N_d(V, R)
        \subseteq
        N_d(U \cap V, S).
    \]
\end{dfn}

\begin{exm}
    For the metric space $\rel$ with its standard metric $d$,
    the pair of subspaces $(\rel_{\leq 0}, \rel_{\geq 0})$ is coarsely
    excisive: The $R$-neighborhoods are
    $N_d(\rel_{\leq 0}, R) = \intOC{\infty}{R}$ and
    $N_d(\rel_{\geq 0}, R) = \intCO{-R}{\infty}$. Their intersection is
    $[-R,R]$, which, for $S = R$, is the $S$-neighborhood of
    $\rel_{\leq 0} \cap \rel_{\geq 0} = \{0\}$.

    In the same vain, $\rel^{n+1}$
    admits the coarsely excisive pair
    $(\rel^n \times \rel_{\leq 0}, \rel^n \times \rel_{\geq 0})$
    under $d_1$, $d_2$, or $d_\infty$.
\end{exm}

\begin{exm}
\label{exm:coarselyExcisiveEmpty}
    For all $S > 0$, the $S$-neighborhood of $\varemptyset$ is again
    $\varemptyset$. This imposes restrictions on eligible coarsely excisive
    pairs: In any metric space $(X,d)$, disjoint nonempty sets $U$ and $V$
    cannot form a coarsely excisive pair. Choose $R$ larger than $\inf
    \set{d(x,y)}{x \in U\textrm{, } y \in V}$, then $N(U, R) \cap N(V, R)$
    contains points. This is never a subset of
    $N(U \cap V, S) = N(\varemptyset, S) = \varemptyset$.
\end{exm}

\begin{thm}[{\cite[Section 5]{higsonRoeYu-mayerVietoris}}]
\label{thm:roeMayerVietoris}
    For a coarsely excisive pair $(U, V)$ of $(X, d)$,
    there is an exact Mayer-Vietoris sequence:
    \[
        \begin{tikzpicture}
            \mama {
                K_0 \cualalg(U \cap V)
                & K_0 \cualalg U \oplus K_0 \cualalg V
                & K_0 \cualalg X \\
                K_1 \cualalg X
                & K_1 \cualalg U \oplus K_1 \cualalg V
                & K_1 \cualalg(U \cap V). \\
            };
            \path[->, font=\scriptsize]
                (m-1-1) edge (m-1-2)
                (m-1-2) edge (m-1-3)
                (m-1-3) edge (m-2-3)
                (m-2-3) edge (m-2-2)
                (m-2-2) edge (m-2-1)
                (m-2-1) edge (m-1-1)
            ;
        \end{tikzpicture}
    \]
\end{thm}

Boths proofs in \cite{higsonRoeYu-mayerVietoris} and \cite{roe-book}
reduce this to the Mayer-Vietoris principle
for abstract C*-ideals $I_0$, $I_1 \subseteq A$, Theorem
\ref{thm:abstractMayerVietoris}.

The goal of this thesis is to construct a
spectral sequence extending that abstract Mayer-Vietoris principle
and Theorem \ref{thm:roeMayerVietoris} alongside.

\subsection{Spectral sequences}

A good introduction to spectral sequences is \cite{mccleary-ssq}. We will
give the basic definitions to establish notation.
We require no
cohomological spectral sequences or ring structures on the pages.

\begin{dfn}[Spectral sequence]
    A \emph{spectral sequence} (of homological type, of abelian groups)
    is a system of bigraded differential abelian groups
    $E^r_{p,q}$ for all $0 \neq r \in \nat$
    and $p$, $q \in \zet$ with differentials
    $d^r\colon E^r_{p,q} \to E^r_{p-r,q+r-1}$ for all $r$, $p$, $q$
    such that each $E^{r+1}_{p,q}$ is the homology of $d^r$ at $E^r_{p,q}$.
\end{dfn}

We define convergence of spectral seuqences with notation similar to
\cite[Section 5]{BoardmanConditionallyConvergent}; that exposition does not
assert any common origin of the target group and the $E^r_{*,*}$-terms
of the spectral sequence.
We re-index to match
our spectral sequences of homological type and make explicit the grading
of the $\zet$-graded target group.

\begin{dfn}
\label{dfn:Hausdorff}
    Let $G$ be an abelian group with an increasing filtration
    \[
        \dotsb \subs F^p G \subs F^{p+1} \subs \dotsb \subs G
    \] for $p \in \zet$.
    We call the filtration $\{F^p G\}_{p \in \zet}$
    \begin{itemize}
        \item \emph{Hausdorff} if
            $\bigcap_{p \in \zet} F^p G = 0$,
        \item \emph{exhaustive} if
            $\bigcup_{p \in \zet} F^p G = G$, and
        \item \emph{complete} if the right-derived functor
            of taking the inverse limit yields the zero group for the
            inverse system $F^p G$ for $p \to -\infty$.
    \end{itemize}
\end{dfn}

\begin{dfn}[Strong convergence]
\label{dfn:strongConvergence}
    Let $\{E^r_{p,q}, d^r\}_{r,p,q}$ be a spectral sequence. For $r \geq 1$
    and $p$, $q \in \zet$, write
    \begin{align*}
        Z^r_{p,q} &= {\ker d^r}\colon E^r_{p,q} \to E^r_{p-r,q+r-1},\\
        B^r_{p,q} &= \im d^r\colon E^r_{p+r,q-r+1} \to E^r_{p,q}.
    \end{align*}
    Because $E^r_{*,*}$ for $r \geq 2$ is the homology of
    $E^{r-1}_{*,*}$ under $d^{r-1}$,
    an element in $E^r_{p,q}$
    may be written as $x + B^{r-1}_{p,q}$ with $x \in E^{r-1}_{p,q}$.
    Recursively, this allows us to treat
    $Z^r_{p,q}$ and $B^r_{p,q}$ as subgroups of $E^1_{p,q}$ and define
    \[
        E^\infty_{p,q}
            = \Big( \bigcap_{r \geq 1} Z^r_{p,q} \Big)
            \Bigm/ \Big( \bigcup_{r \geq 1} B^r_{p,q} \Big).
    \]
    Let $G = \bigoplus_{s \in \zet} G_s$
    be a $\zet$-graded abelian group.
    The spectral sequence $\{E^r_{p,q}, d^r\}_{r,p,q}$
    \emph{converges strongly} to $G$ if there exist increasing filtrations
    $\{F^p G_s\}_{p \in \zet}$ of each summand $G_s$ such that these
    filtrations are Hausdorff, exhaustive, complete, and allow isomorphisms
    \[
        E^\infty_{p,q} \cong F^p G_{p+q} / F^{p-1} G_{p+q}.
    \]
\end{dfn}

\begin{dfn}[Morphism of spectral sequences]
\label{dfn:morphismOfSseq}
    Given two spectral sequences $\{E_{p,q}^r, d^r\}_{r,p,q}$ and
    $\{\bar E_{p,q}^r, \bar d^r\}_{r,p,q}$ and $(p', q') \in \zet^2$,
    a \emph{morphism of spectral sequences of
    bidegree $(p', q')$} is
    a system of morphisms of abelian groups,
    \[
        f = {\left\{ f_{p,q}^r\colon E_{p,q}^r
            \to \bar E_{p+p',q+q'}^r\right\}_{r,p,q}},
    \]
    such that
    \begin{itemize}
        \item the group morphisms commute with the differentials; i.e.,
        $f^{r}_{*,*} \circ d^r = \bar d^r \circ f^r_{*,*}$ for all pages
        $r$, and
    \item
        each map $f^{r}_{*,*}$ induces $f^{r+1}_{*,*}$ by passing to
        homology on $\{E^r_{*,*}, d^r\}_{r,p,q}$ and $\{\bar E^r_{*,*},
        \bar d^r\}_{r,p,q}$.
    \end{itemize}
\end{dfn}

\begin{rem}
    Spectral sequences with these morphisms form a category.

    By describing a morphism of spectral sequences on the $R$-th page,
    all subsequent $f^r_{*,*}$ for $r > R$ and $r = \infty$
    are implicitly defined because
    the $E^r_{*,*}$-terms are iterative homologies of the earlier
    $E^R_{*,*}$-term. In our setting, we will construct morphisms of spectral
    sequences only for the $E^1_{*,*}$-terms.

    In particular, if $f^R_{*,*}$ is an isomorphism between the differential
    graded abelian groups $E^R_{*,*}$ and $\bar E^R_{*,*}$,
    then all $f^r_{*,*}$ for $r > R$ and $r = \infty$ become isomorphisms.
\end{rem}

\section{Ideal inclusions}\thispagestyle{plain}
\label{sec:schochetSseq}

\subsection{Main theorem}

\begin{restatable}[Spectral sequence for ideal inclusions]{thm}{schochetSseq}
\label{thm:schochetSseq}
    \textSchochetSseq
\end{restatable}

In \cite{schochet-sseq}, C. Schochet gave a proof of Theorem
\ref{thm:schochetSseq}.

Nonetheless, we will reprove Theorem \ref{thm:schochetSseq}
based on very general theory from \cite{cartan1973homological}.
This extra work reveals the inner mechanisms of the spectral sequence,
shows that the convergenge is strong,
and highlights naturality of all constructions:
The morphisms in K-theory arise from
natural inclusions of C*-ideals, quotients of C*-ideals, and boundary
maps.

This spectral sequence serves as groundwork for the
Mayer-Vietoris results in later sections.

\subsection{Abstract H-systems}

In \cite{cartan1973homological}, H. Cartan and S. Eilenberg
construct an abstract
spectral sequence from a bigraded system of groups,
but they omit some details during their proof of
convergence. Their construction uses cohomological differentials: On the
$E_r^{*,*}$-page, the differential has the degree $(r,1-r)$.
For homological spectral sequences, they suggest the renumbering
$E^r_{p,q} = E_r^{-p,-q}$. We will state the main theorem of
\cite{cartan1973homological} in this renumbered notation, then prove
it with all details.

\begin{dfn}[Ungraded H-system]
\label{dfn:ungradedhsystem}
    Let $H(p,p')$ be abelian groups for $p' \leq p$
    from the range $\zet \cup \{\pm \infty\}$. We introduce the shorthand
    notations
    \begin{align*}
        H(p) &= H(p,-\infty),\\
        H = H(\infty) &= H(\infty, -\infty).
    \end{align*}
    For each $(p, p')$ and $(q, q')$ with $-\infty \leq p \leq q \leq \infty$
    and $p' \leq p \leq \infty$ and $q' \leq q \leq \infty$, let
    there be a morphism
    \[
        i\colon H(p,p') \to H(q,q').
    \]
    For each $-\infty \leq p'' \leq p' \leq p \leq \infty$, let
    \[
        \del\colon H(p,p') \to H(p',p'')
    \]
    be a \emph{connecting homomorphism}.
    We call this collection of groups together with the above morphisms
    an \emph{ungraded H-system} if the following axioms are satisfied:
    \begin{enumerate}
        \item $i\colon H(p,p') \to H(p,p')$ is the identity.
        \item All triangle and square diagrams built with the
            morphisms $i$ commute.
        \item For all $p'' \leq p' \leq p$, there is an exact sequence
            \begin{equation}
            \label{eqn:hsystemexactsequence}
                \dotsb \loto{\del}
                H(p',p'') \loto{i}
                H(p,p'') \loto{i}
                H(p,p') \loto{\del}
                H(p',p'') \to \dotsb.
            \end{equation}
        \item For each index $p' \in \zet \cup \{-\infty\}$,
            the group $H(\infty, p')$ is the direct limit of the
            morphisms
            $i\colon H(p, p') \to H(p+1, p')$ along $p' \leq p$.
    \end{enumerate}
\end{dfn}

In \cite{cartan1973homological}, the indices of these H-systems are
denoted by $(p,q)$
instead of $(p,p')$. To avoid confusion with the bigrading $(p,q)$
of the pages $E^r_{p,q}$ later, we shall use $H(p,p')$.

\begin{dfn}[Graded H-system]
\label{dfn:gradedhsystem}
    Let $\{H(p, p'), i, \del\}_{p,p'}$ be an ungraded H-system as
    in Definition \ref{dfn:ungradedhsystem}.
    We call this a \emph{graded H-system} if
    it satisfies the following extra axioms:
    \begin{enumerate}
    \setcounter{enumi}{4}
        \item All $H(p,p')$ carry a $\zet$-grading:
            $H(p,p') = \bigoplus_{s \in \zet} H_s(p,p')$.
        \item All morphisms $i\colon H(p,p') \to H(q,q')$
            are degree-preserving.
        \item All morphisms $\del\colon H(p,p') \to H(p',p'')$
            have degree $-1$; i.e.,\[
                \im \big(\del \restr H_s(p,p')\big)
                \subseteq H_{s-1}(p',p'').
            \]
    \end{enumerate}
\end{dfn}

\begin{ntt}
\label{ntt:cartanZBE}
    Let $H(p,p')$ for $-\infty \leq p' \leq p \leq \infty$ form a graded
    H-system. For $r \geq 0$ and $q \in \zet$, write
    \begin{align*}
        Z^r_{p,q} &= \im i\colon H_{p+q}(p, p-r-1) \to H_{p+q}(p, p-1), \\
        B^r_{p,q} &= \im \del\colon H_{p+q+1}(p+r, p) \to H_{p+q}(p, p-1), \\
        E^{r+1}_{p,q} &= Z^r_{p,q} / B^r_{p,q}.
    \end{align*}
    In this way, we define $E^r_{p,q}$ only for $r \geq 1$, not for $r \geq 0$.
    Compared to \cite{cartan1973homological}, we have shifted the index
    $r$ in $Z^r_{*,*}$ and $B^r_{*,*}$ by $1$ to match our Definition
    \ref{dfn:strongConvergence} of these groups as closely as possible;
    e.g., we write $Z^0_{p,q}$ for what would be denoted by $Z^1_{p,q}$
    in \cite{cartan1973homological}.
\end{ntt}

\begin{lem}
    We have $B^0_{p,q} = 0$ and
    \[
        E^1_{p,q} \cong Z^0_{p,q} = H_{p+q}(p,p-1).
    \]
\end{lem}

\begin{proof}
    In the long exact sequence
    \[
        \dotsb
        \to H_{p+q+1}(p,p) \loto{\del} H_{p+q}(p,p-1)
        \loto{i} H_{p+q}(p,p-1)
        \loto{i} H_{p+q}(p,p) \to \dotsb,
    \]
    the central
    map $i\colon H(p,p-1) \to H(p,p-1)$ is the identity by Definiton
    \ref{dfn:ungradedhsystem}. Its image is $Z^0_{p,q}$, which
    is all of $H_{p+q}(p,p-1)$. The exactness of the sequence forces
    the preceding map $\del$ to vanish.
    The group $B^0_{p,q}$ is the image of $\del$, therefore it is
    the trivial group.
\end{proof}

The main statement in \cite{cartan1973homological} becomes:
\begin{thm}
\label{thm:hsystemsseq}
    With Notation \ref{ntt:cartanZBE},
    there is a spectral sequence
    $
        \{E^r_{p,q}, d^r_{p,q}\}_{r,p,q}
    $
    of homological type. Its differentials
    $d^r_{p,q}\colon E^r_{p,q} \to E^r_{p-r,q+r-1}$
    are defined as the composition
    \[
        \begin{tikzpicture}
            \mamax{0.5 cm} {
                E^r_{p,q} && E^r_{p-r,q+r-1}
                \\
                Z^{r-1}_{p,q}/B^{r-1}_{p,q}
                &
                Z^{r-1}_{p,q}/Z^r_{p,q}
                \cong
                B^r_{p-r,q+r-1}/B^{r-1}_{p-r,q+r-1}
                & Z^{r-1}_{p-r,q+r-1}/B^{r-1}_{p-r,q+r-1},
                \\
            };
            \path[->, font=\scriptsize]
                (m-1-1) edge node[auto] {$d^r_{p,q}$} (m-1-3)
                (m-2-1) edge (m-2-2)
                (m-2-2) edge (m-2-3)
            ;
            \path[double, font=\scriptsize]
                (m-1-1) \maeq (m-2-1)
                (m-1-3) \maeq (m-2-3)
            ;
        \end{tikzpicture}
    \]
    where the three maps at the bottom are all constructed in
    \cite[Chapter XV, Paragraph 1]{cartan1973homological}: The bottom-left map
    arises from factoring out the larger group $Z^r_{p,q} \supseteq
    B^{r-1}_{p,q}$, the central map is an isomorphism, and the last map arises
    from the inclusion $B^r_{p-r,q+r-1} \to Z^{r-1}_{p-r,q+r-1}$.
    The homology of $E^r_{*,*}$ under $d^r$ at $(p,q)$
    is isomorphic to $E^{r+1}_{p,q}$.
\end{thm}

We will first look at our application to K-theory of C*-algebras, then
return to the full proof of convergence.

\subsection{Application: K-theory}

We construct a graded H-system to compute the K-theory
of a C*-algebra, taking a chain of ideals as data. This is a
$\zet$-graded theory. Bott periodicity forces $K_s A = K_{s+2} A$
for all C*-algebras $A$ over the complex numbers,
leaving only two different K-groups to be computed.

Throughout the remainder of Section \ref{sec:schochetSseq},
let
\[
    0 \subseteq I_0 \subseteq I_{1} \subseteq \dotsb \subseteq I_p
\subseteq \dotsb
\]
be an increasing chain of C*-ideals for $p \in \nat$ and set
\[
    A = \overline{\bigcup_{p \in \nat} I_p}.
\]
For convenience, set $I_p = 0$ for $p < 0$, obtaining a $\zet$-graded
chain of ideals $(I_p)_{p \in \zet}$.

Commonly, the chain of ideals will stabilize after finitely many steps; i.e.,
there exists $n \in \nat$ with $I_n = I_{n+1} = I_{n+2} = \dotsb$. Still,
we develop the spectral sequence
for the general case without assuming stabilization.
The extra work is marginal and the final results of this
thesis will rely on that general case.

\begin{dfn}
\label{dfn:howhbecomesk1}
    For all $\zet$-indices $p' \leq p$ and K-theory degrees
    $s \in \zet$, define
    \begin{align*}
        H_s(p,p') &= K_s(I_p/I_{p'}), \\
        H_s(p) &= K_s I_p, \\
        H_s &= K_s A.
    \end{align*}
    The morphisms $i\colon H_s(p, p') \to H_s(p+1, p')$
    in K-theory are induced by inclusions of ideals.
    The morphisms of the form
    $i\colon H_s(p, p') \to H_s(p, p'+1)$
    are induced by the natural projection $I_p / I_{p'} \to I_p / I_{p'+1}$;
    this projection is well-defined because $I_{p'} \subseteq I_{p' + 1}$.
    All these morphisms commute with each other and preserve the degree
    in K-theory.

    The assignment $H_s(p,p') = K_s(I_p/I_{p'})$
    satisfies the direct limit axiom
    from Definition \ref{dfn:ungradedhsystem} regarding $H(p)$ and $H$:
    \begin{align*}
        K_s(A/I_q) &= \colim_{p \to \infty} H_s(p,p') = H_s(\infty,p'), \\
        K_s A &= \colim_{p \to \infty} H_s(p,-\infty) = H_s.
    \end{align*}
    For $i\colon H_s(p) \to H_s(p,p')$ and $i\colon H_s(p) \to H_s$,
    we use the respective limit maps. Because K-theory is a continuous functor,
    these are induced by inclusions and projections of $A$.
    Again, these limit maps preserve the degree $s$ as desired.
\end{dfn}

\begin{ntt}
    We will deal with two kinds of boundary maps: The K-theoretic boundary map
    and the connecting homomorphism of the resulting H-system. To distinguish
    these, throughout Section \ref{sec:schochetSseq},
    we shall denote the K-theoretic map by
    $\del_K$ and the connecting homomorpism by $\del$.
\end{ntt}

\begin{dfn}
\label{dfn:howhbecomesk2}
    Let $p'' \leq p' \leq p$ be indices in $\zet$.
    We implement the connecting homomorphisms $\del$ in the diagram
    \begin{equation}
    \label{eq:threetermh}
        \dotsb \to
        H_s(p,p'') \loto{i}
        H_s(p,p') \loto{\del}
        H_{s-1}(p',p'') \loto{i}
        H_s(p,p'') \to \dotsb
    \end{equation}
    by a composition of maps in K-theory, making this diagram commutative:
    \[
        \begin{tikzpicture}
            \mama {
                H_s(p,p') & & & H_{s-1}(p',p'')
                \\
                K_s(I_p/I_{p'}) & K_{s-1} I_{p'} & {\phantom{wwwww}}
                    & K_{s-1}(I_{p'}/I_{p''}).
                \\
            };
            \path[->, font=\scriptsize]
                (m-1-1) edge node[above] {$\del$} (m-1-4)
                (m-2-1) edge node[below] {$\del_K$} (m-2-2)
                (m-2-2) edge node[below]
                    {$K_s(\proj\colon I_{p'} \to I_{p'}/I_{p''})$} (m-2-4)
            ;
            \path[double, font=\scriptsize]
                (m-1-1) \maeq (m-2-1)
                (m-1-4) \maeq (m-2-4)
            ;
        \end{tikzpicture}
    \]
\end{dfn}

\begin{lem}
    This choice of connecting homomorphism $\del$ in Definition
    \ref{dfn:howhbecomesk2} makes the sequence \ref{eq:threetermh} exact.
\end{lem}

\begin{proof}
    Let $p'' \leq p' \leq p$ be indices in $\zet$ and
    $I_{p''} \subseteq I_{p'} \subseteq I_p$
    be a chain of ideals in $A$. According to Definition
    \ref{dfn:howhbecomesk1},
    we can rewrite the long exact sequence \ref{eq:threetermh}
    into the top row of the following commutative diagram:
    \[
        \newcommand{\kpq}{K_{s}}
        \newcommand{\frr}[2]{\left(#1/#2\right)}
        \begin{tikzpicture}
            \mama {
                \kpq\frr{I_p}{I_{p''}}
                & \kpq\frr{I_p}{I_{p'}}
                & K_{s-1}\frr{I_{p'}}{I_{p''}}
                & K_{s-1}\frr{I_{p}}{I_{p''}}
                \\
                \kpq I_p
                & \kpq\frr{I_p}{I_{p'}}
                & K_{s-1} I_{p'}
                & K_{s-1} I_{p}.
                \\
            };
            \path[->, font=\scriptsize]
                (m-1-1) edge node[auto] {$i$} (m-1-2)
                (m-1-2) edge node[auto] {$\del$} (m-1-3)
                (m-1-3) edge node[auto] {$i$} (m-1-4)
                (m-2-1) edge node[auto] {$f_*$} (m-1-1)
                (m-2-3) edge node[auto] {$g_*$} (m-1-3)
                (m-2-4) edge node[auto] {$h_*$} (m-1-4)
                (m-2-1) edge node[below] {$\proj_*$} (m-2-2)
                (m-2-2) edge node[below] {$\del_K$} (m-2-3)
                (m-2-3) edge node[below] {$\incl_*$} (m-2-4)
            ;
            \path[double, font=\scriptsize]
                (m-1-2) \maeq (m-2-2)
            ;
        \end{tikzpicture}
    \]
    The vertical arrows $f_*$, $g_*$, $h_*$ arise from natural projections:
    They are induced in K-theory from factoring out
    $I_{p''}$. The identity arrow is also of this type because
    \[
        I_{p''} \subseteq I_{p'}
        \qquad \Longrightarrow \qquad
        \frac{I_p/I_{p''}}{I_{p'}/I_{p''}} \cong I_p/I_{p'}.
    \]
    The top row -- except possibly at $\del$ -- matches the long
    exact sequence in K-theory that corresponds to the short exact sequence
    $0 \to I_{p'}/I_{p''} \to I_p/I_{p''} \to I_p/I_{p'} \to 0$.
    We wish to prove that $\del$ turns the upper row into a long exact
    sequence.

    We recognize $\del_K$ as the boundary map in K-theory for the
    short exact sequence $0 \to I_{p'} \to I_p \to I_p/I_{p'} \to 0$.
    By naturality of the exact sequence with respect to factoring
    out $I_{p''}$, the composition $g_* \circ \del_K$ is the K-theoretic
    connecting homomorphism to make the upper row exact.
    By Definition \ref{dfn:howhbecomesk2},
    we have $\del = g_* \circ \del_K$. Thus $\del$ is the correct arrow
    to construct the H-system.
\end{proof}

\subsection{Filtration}

\begin{dfn}
\label{dfn:filtrationOfKA}
    For $s \in \zet$,
    the chain of ideals leads to a $\zet$-indexed increasing
    filtration $\{F^p K_sA\}_{p \in \zet}$ of $K_s A$,
    \[
        F^p H_s = F^p K_s A = \im(i\colon K_s I_p \to K_s A) \subseteq K_s A.
    \]
\end{dfn}

To discuss this filtration,
we will continue to write $s \in \zet$ for the index in K-theory.
Afterwards, the K-theory groups will be indexed by $(p+q)$ to show
convergence of the spectral sequence.

\begin{pro}
\label{pro:filtrationExhausts}
    For all $s \in \zet$, the filtration
    $\{F^p K_s A\}_{p \in \zet}$
    in Definition \ref{dfn:filtrationOfKA}
    is Hausdorff, exhaustive, and complete according to Definition
    \ref{dfn:Hausdorff}.
\end{pro}

\begin{proof}
    The Hausdorff property is immediate because $I_{p'} = 0$ for $p' < 0$,
    therefore $F^{p'} H_s = \im(i\colon 0 \to K_s A) = 0
    \supseteq \bigcap_{p \in \zet} F^p H_s$.

    For exhaustion, consider the input of the spectral sequence: An
    inclusion chain
    of C*-ideals $I_0 \subseteq I_1 \subseteq I_2 \subseteq \dotsb
    \subseteq I_p \subseteq \dotsb$
    with $\overline{\bigcup_{p \in \nat} I_p} = A$.
    This closure is the direct limit object of the system of ideal inclusions
    $(I_p \to I_{p+1})_{p \in \nat}$ in $\catcstar$. K-theory is a
    continuous functor, rendering the K-theory of the limit object $A$
    isomorphic to the limit of the system of K-theory groups along the
    morphisms $i\colon K_s I_p \to K_s I_{p+1}$. By Definition
    \ref{dfn:howhbecomesk1}, these are exactly the morphisms that appear
    in Definition \ref{dfn:filtrationOfKA}
    of the filtration $\{F^p H_s\}_{p \in \zet}$.
    Finally, universality of the limit object $K_s A$ guarantees
    that the above system of morphisms $i$ exhausts $K_s A$.

    Completeness is trivially satisfied because $F^p H_s = 0$ for all
    $p < 0$. Both the inverse limit and its right derivative
    vanish for this system.
\end{proof}

\begin{rem}
\label{rem:iDiagramCommutes}
    Even when one C*-ideal $I_0 \subseteq I_1$ is included in another,
    the K-theory groups need not
    be connected by a system of injections
    $K_* I_0 \hookrightarrow K_* I_1 \hookrightarrow \dotsb$;
    for example, let $H$ be an infinite-dimensional Hilbert space, then
    $KH$, the compact operators of $H$, have $K_0 KH = \zet$, but
    $KH \subseteq BH$ is an inclusion of ideals and $K_0 BH = 0$.

    Nonetheless, the filtration $\{F^p H_s\}_{p \in \zet} =
    \{F^p K_s A\}_{p \in \zet}$
    satisfies $F^p H_s \subs F^{p+1} H_s$ for all $p \in \zet$:
    The group $\im(i\colon K_s I_p \to K_s A)$ is a subgroup of
    $\im(i\colon K_s I_{p+1} \to K_s A)$ because, by definition of an
    H-system, the morphisms $i$ factor through each other, making this
    diagram commutative:
    \[
        \newcommand{\col}{\colim_p}
        \begin{tikzpicture}
            \mamaxy{1.5cm}{0.3cm}{
                K_s I_p & \\
                & K_s A = \col K_s I_p = K_s \col I_p. \\
                K_s I_{p+1} & \\
            };
            \path[->, font=\scriptsize]
                (m-1-1) edge node[left] {$i$} (m-3-1)
                (m-1-1) edge node[above right] {$i$} (m-2-2.175)
                (m-3-1) edge node[below right] {$i$} (m-2-2.185)
            ;
        \end{tikzpicture}
    \]
\end{rem}

\subsection{Convergence}

We would like to show that this filtration makes the K-theory spectral sequence
converge strongly to $K_*A$. Here we replace the K-theoretic degree
$s$ by $p + q$.

\begin{ntt}[$Z^\infty_{p,q}$, $B^\infty_{p,q}$]
    With $Z^r_{p,q}$ and $B^r_{p,q}$ as in Notation \ref{thm:hsystemsseq},
    write
    \[
        Z^\infty_{p,q} = \bigcap_{r \geq 1} Z^r_{p,q},
        \qquad
        B^\infty_{p,q} = \bigcup_{r \geq 1} B^r_{p,q}.
    \]
\end{ntt}

For each $q \in \zet$,
the filtration $\{F^p K_{p+q} A\}_{p \in \zet}$ from Definition
\ref{dfn:filtrationOfKA} leads to
successive quotients $F^p K_{p+q}A / F^{p-1} K_{p+q}A$ across all $p \in \zet$.
We have to show that this $(p,q)$-indexed collection of quotients
coincides with
\[
    E^\infty_{p,q} = Z^\infty_{p,q}/B^\infty_{p,q} =
        \Big( \bigcap_{r \geq 1} Z^r_{p,q} \Big) \Bigm/
        \Big( \bigcup_{r \geq 1} B^r_{p,q} \Big).
\]
Because the filtration has already been proven
Hausdorff, exhaustive, and complete, the convergence will then be strong.

We have shaped our input
according to the most general axioms in \cite{cartan1973homological}.
Even though convergence for more specialized input is proven in that source,
convergence is merely claimed for our input.
Henceforth, we shall give a full proof.

\begin{lem}
\label{lem:finiteSseqCollapses}
    If there exists $n \in \nat$ with $A = I_n = I_{n+1} = I_{n+2} = \dotsb$,
    then the spectral sequence collapses at page $n+1$:
    We have $E^r_{*,*} = E^{r+1}_{*,*}$ for all
    $r \geq n + 1$, thus $E^\infty_{p,q} = E^{n+1}_{p,q}$.
\end{lem}

\begin{proof}
    Fix $p$ and $q \in \zet$. If $p < 0$, then $B^r_{p,q}$ and $Z^r_{p,q}$
    vanish by definition for all $r \geq 0$ because $I_p = 0$ and we have
    $E^{r+1}_{p,q} = 0$ here.

    Consider the case $p \geq 0$. For all $r \geq n$, we have
    \[
        B^r_{p,q} = \im \del\colon K_{p+q+1}(I_{p+r}/I_p) \to
            K_{p+q}(I_p/I_{p-1}),
    \]
    where $I_{p+r} = I_n = A$ for all $r \geq n+1$. Thus
    all $B^r_{p,q}$
    coincide for such high page numbers $r$. In similar fashion,
    all $Z^r_{p,q}$ become $\im i\colon K_{p+q} I_p \to K_{p+q}(I_p/I_{p-1})$.
    The collapse follows from the definition
    $E^{r+1}_{p,q} = Z^r_{p,q} / B^r_{p,q}$
    for all $r \in \zet$.
\end{proof}

This collapsing lemma reveals some structure of the pages $E^r_{*,*}$ when the
chain of ideals stabilizes.
The following convergence theorem holds with or without stabilization of the
ideals.

\begin{thm}
\label{thm:einftyisfiltered}
    The $E^\infty$-term admits the desired filtration; i.e.,
    \[
        E^\infty_{p,q} \cong F^p H_{p+q} / F^{p-1} H_{p+q}.
    \]
\end{thm}

Substituting the definitions lets us rewrite the claim like this,
denoting the yet-undefined isomorphism in the middle by $f$:
\begin{multline}
\label{eqn:einftyisfiltered}
    E^\infty_{p,q} =
    \frac{Z^\infty_{p,q}}{B^\infty_{p,q}}
    =
    \frac{
        \im i\colon K_{p+q} I_p \to K_{p+q}(I_p/I_{p-1})
    }{
        \im \del\colon K_{p+q+1}(A/I_p) \to K_{p+q}(I_p/I_{p-1})
    } \\
    \stackrel{f}{\cong}
    \frac{
        \im i\colon K_{p+q} I_p \to K_{p+q} A
    }{
        \im i\colon K_{p+q} I_{p-1} \to K_{p+q} A
    }
    =
    \frac{F^p H_{p+q}}{F^{p-1} H_{p+q}}.
\end{multline}
Our strategy is to construct the central
isomorphism $f$ in \ref{eqn:einftyisfiltered} explicitly.
After giving its construction, we show that $f$ is well-defined, injective,
and surjective. That will constitute the proof of
Theorem \ref{thm:einftyisfiltered}.

\begin{dfn}
\label{dfn:einftyisfiltered}
    For $x + B^\infty_{p,q}$, an element in
    $Z^\infty_{p,q}/B^\infty_{p,q}$, we must define $f(x + B^\infty_{p,q})$.
    Find $y \in K_{p+q} I_p$ with $i(y) = x$. Define
    \[
        f(x + B^\infty_{p,q}) = i^p_A(y) + F^{p-1} K_{p+q} A,
    \]
    where $i^p_A\colon K_{p+q} I_p \to K_{p+q} A$
    denotes the standard map from our
    H-system, induced by the inclusion of algebras.
\end{dfn}

\begin{lem}
\label{lem:fiswelldefined}
    The morphism $f$ is well-defined: The construction in Definition
    \ref{dfn:einftyisfiltered}
    is independent of the choice of $y \in K_{p+q} I_p$ with $i(y) = x$.
\end{lem}

\begin{proof}
    Let $y$ and $y' \in K_{p+q} I_p$ with
    $i(y - y') \in B^\infty_{p,q}$. We have to show that $f(y) = f(y')$,
    equivalently, that
    $f(y - y') \in F^{p-1} K_{p+q} A
        = \im i\colon K_{p+q} I_{p-1} \to K_{p+q} A$.

    Because $i(y - y') \in B^\infty_{p,q} =
    \im \del\colon K_{p+q+1}(A/I_p) \to K_{p+q}(I_p/I_{p-1})$,
    we find $z \in K_{p+q+1}(A/I_p)$ with $\del(z) = i(y - y')$.
    This $\del$ belongs to the exact sequence
    \begin{multline*}
        \dotsb \to
        K_{p+q+1}(A/I_{p-1})
        \loto{i'} K_{p+q+1}(A/I_p)\\
        \loto{\del} K_{p+q}(I_p/I_{p-1})
        \loto{i'} K_{p+q}(A/I_{p-1})
        \to \dotsb.
    \end{multline*}
    Onto this exact sequence, we draw a commutative square, and then extend
    the right-hand side of the square to a vertical exact sequence in K-theory.
    \begin{equation}
    \label{eqn:einftyisfiltereddiagram}
        \begin{tikzpicture}
            \mamax{0.6 cm} {
                & & & K_{p+q} I_{p-1} \\
                & & K_{p+q} I_p & K_{p+q} A \\
                \dotsb & K_{p+q+1}(A/I_p) & K_{p+q}(I_p/I_{p-1})
                    & K_{p+q}(A/I_{p-1}) & \dotsb. \\
            };
            \path[->, font=\scriptsize]
                (m-1-4) edge node[right] {$i^{p-1}_A$} (m-2-4)
                (m-2-3) edge node[auto] {$i^p_A$} (m-2-4)
                (m-2-3) edge node[left] {$i$} (m-3-3)
                (m-2-4) edge node[right] {$i$} (m-3-4)
                (m-3-1) edge node[below] {$i'$} (m-3-2)
                (m-3-2) edge node[below] {$\del$} (m-3-3)
                (m-3-3) edge node[below] {$i'$} (m-3-4)
                (m-3-4) edge node[below] {$i'$} (m-3-5)
            ;
        \end{tikzpicture}
    \end{equation}
    Chasing $y - y' \in K_{p+q} I_p$ through this diagram, we obtain
    \[
        (i' \circ i)(y - y') = (i' \circ \del)(z) = 0 \in K_{p+q}(A/I_{p-1})
    \]
    because the bottom row is exact.
    Then $(i \circ i^p_A)(y - y') = 0$
    due to the commutativity of the square.
    With $i^p_A(y - y') \in \ker i\colon K_{p+q} A \to K_{p+q}(A/I_{p-1})$,
    we conclude from the exactness of the vertical sequence that
    $i^p_A(y - y') \in \im i^{p-1}_A = F^{p-1} K_{p+q} A$.

    This shows that
    $f\colon E^\infty_{p,q} \to F^p K_{p+q} A / F^{p-1} K_{p+q} A$
    is well-defined.
\end{proof}

\begin{lem}
    The morphism $f$ is injective.
\end{lem}

\begin{proof}
    Let $x + B^\infty_{p,q}$ be a class in $E^\infty_{p,q}$ that vanishes
    under $f$. We have to show that $x \in B^\infty_{p,q}$.
    This proof looks like the proof of Lemma \ref{lem:fiswelldefined}
    in reverse.

    Select $y \in K_{p+q}(I_p)$ with $i(y)
    = x \in \im i\colon K_{p+q} I_p \to K_{p+q}(I_p/I_{p-1})$.
    From $f\big(i(y) + B^\infty_{p,q}\big) = 0$ and the definition of $f$,
    we know
    $i^p_A(y) \in \im i^{p-1}_A\colon K_{p+q} I_{p-1} \to K_{p+q} A$.
    Consider diagram \ref{eqn:einftyisfiltereddiagram} again:
    By exactness of the vertical sequence, $(i \circ i^p_A)(y) = 0$.
    Commutativity of the square shows that $(i' \circ i)(y) = 0$.
    Exactness of the bottom row gives
    $i(y) \in \im \del\colon K_{p+q+1}(A/I_p) \to K_{p+q}(I_p/I_{p-1})$.

    After substituting $i(y) = x$ and $\im \del = B^\infty_{p,q}$, we have
    shown $x \in B^\infty_{p,q}$ and therefore the injectivity of $f$.
\end{proof}

\begin{lem}
    The morphism $f$ is surjective.
\end{lem}

\begin{proof}
    Let $z + \im i^{p-1}_A$ be in $F^p K_{p+q} A / F^{p-1} K_{p+q} A$.
    For this $z \in F^p K_{p+q} A = \im i^p_A$, we may find a lift
    $y \in K_{p+q} I_p$ with $i^p_A(y) = z$.
    Then $i(y) \in Z^\infty_{p,q}$ already satisfies
    $f\big(i(y) + B^\infty_{p,q}\big) = z + \im i^{p-1}_A$ by definition
    of $f$. Thus $f$ is surjective.
\end{proof}

These lemmas conclude the proof of Theorem \ref{thm:einftyisfiltered}.

\subsection{Summary}

We recall the main theorem of Section \ref{sec:schochetSseq}:

\schochetSseq*

The theorem is plausible: Fix $n \in \nat$, choose $I_p = 0$ for $p < n$
and $I_n = A$. The spectral sequence begins with $E^1_{p,q} = 0$ for $p \neq n$
and $E^1_{n,q} \cong K_{n+q} A$. In the only nonzero column
$E^1_{n,*} \cong E^\infty_{n,*}$, we see the expected
$K_s A \cong E^\infty_{n,q}$ for $n + q = s$.

\begin{proof}[Proof of Theorem \ref{thm:schochetSseq}]
    The computation of $E^1_{p,q}$ is straightforward:
    \[
        E^1_{p,q} =
        \frac{Z^0_{p,q}}{B^0_{p,q}}
        =
        \frac{
            \im \id\colon K_{p+q}(I_p/I_{p-1}) \to K_{p+q}(I_p/I_{p-1})
        }{
            \im \del\colon \underbrace{K_{p+q+1}(I_p/I_p)}_{=0}
                \to K_{p+q}(I_p/I_{p-1})
        }
        \cong K_{p+q}(I_p/I_{p-1}).
    \]
    The differentials $d^r\colon E^r_{*,*} \to E^r_{*,*}$ were defined in
    Theorem \ref{thm:hsystemsseq} and have the correct bidegrees
    $(-r, r-1)$ on page $r$. The strong convergence of this spectral sequence
    follows from Proposition \ref{pro:filtrationExhausts} and
    Theorem \ref{thm:einftyisfiltered}.
\end{proof}

\begin{rem}
    Even though this is a half-plane spectral sequence, its convergence
    is provable like the convergence of a single-quadrant spectral sequence
    because we have
    \emph{exiting differentials}: For each a bidegree $(p, q) \in \zet^2$,
    all except finitely many differentials
    $d^r_{p,q}\colon E^r_{p,q} \to E^r_{p-r, q+r-1}$ exit the half-plane
    $E^r_{p',*}$ for $p' \geq 0$ of nonzero groups.

    There are more intricate results for
    half-plane spectral sequences with \emph{entering differentials} or for
    whole-plane spectral sequences; these will not arise in our setting.
    Besides the classic reference
    \cite{mccleary-ssq}, a good resource for convergence theorems
    is \cite{BoardmanConditionallyConvergent}.
\end{rem}

\section{Finite sums of ideals}\thispagestyle{plain}
\label{sec:finiteSums}

Let $A$ be a C*-algebra.
There is a K-theory spectral sequence for ideals $I_0 \subseteq I_1 \subseteq
I_2 \subseteq \dotsb \subseteq I_n = A$. We will postulate a new spectral
sequence that weakens the \enquote{$\subseteq$} to a mere \enquote{$+$}:
For ideals $I_0$, $I_1$, $I_2$, $\ldots$, $I_n$ with $\sum_{j=0}^n I_j = A$,
there is a spectral sequence that relates the K-theory of their intersections
to the K-theory of $A$.

\subsection{Ideal decompositions}

Even though Section \ref{sec:finiteSums} deals only with the finite case,
we will define \starhoms{} that preserve arbitrarily-sized
ideal decompositions in light of later sections.

\begin{dfn}[Preservation of ideal decompositions]
\label{dfn:preserveIdealDecomps1}
    Let $\alpha$ and $\alpha'$ be arbitrary index sets with
    $\alpha \subs \alpha'$.
    Let $A$
    be the norm closure
    $A = \overline{\sum_{\beta \in \alpha} I_\beta}$
    of the sum of $\card{\alpha}$-many C*-ideals $I_\beta$.
    Let
    $A' = \overline{\sum_{\beta \in \alpha'} I'_\beta}$
    be another
    C*-algebra written as the sum of $\card{\alpha'}$-many C*-ideals
    $I'_\beta$.

    A \starhom{} $f\colon A \to A'$
    \emph{preserves the ideal decomposition} if
    $f(I_\beta) \subseteq I'_\beta$ for every $\beta \in \alpha$.
    Both $f$ and the specific decompositions
    $\{I_\beta\}_{\beta \in \alpha}$ and $\{I'_{\beta}\}_{\beta \in \alpha'}$
    are part of the input data.
\end{dfn}

\begin{rem}[Naturality w.r.t.\ ideal decompositions]
\label{rem:preserveIdealDecomps2}
    It is conceivable to define a category of C*-ideal decompositions
    and decomposition-preserving \starhoms{},
    e.g., with cardinal numbers $\alpha$ as index sets
    to ensure $\alpha \subs \alpha'$
    wherever $\card \alpha \leq \card{\alpha'}$.
    But it will be enough to work in $\catcstar$, the standard category of
    C*-algebras, because all natural constructions here will already be
    natural will w.r.t. ideal decompositions of C*-algebras:

    Let $\catc$ be any category. Let
    $F$, $G\colon \catcstar \to \catc$ be functors of C*-algebras
    and let $\eta\colon F \to G$ be a natural transformation.
    Let $f\colon A \to A'$ be a \starhom{} that preserves an
    $\card{\alpha}$-fold ideal decomposition as in Definition
    \ref{dfn:preserveIdealDecomps1}. Since $f(I_\beta) \subs I'_\beta$
    for all $\beta \in \alpha$ and \starhoms{} are compatible with sums,
    the following diagram commutes in $\catc$:
    \[
        \begin{tikzpicture}
        \mamax{2cm} {
            F(A) = F\big(\sum_{\beta \in \alpha} I_\beta\big)
            &
            G(A) = G\big(\sum_{\beta \in \alpha} I_\beta\big)
            \\
            F(A') = F\big(\sum_{\beta \in \alpha} I'_\beta\big)
            &
            G(A') = G\big(\sum_{\beta \in \alpha} I'_\beta\big).
            \\
        };
        \path[->, font=\scriptsize]
            (m-1-1) edge node[above] {$\eta(A)$} (m-1-2)
            (m-2-1) edge node[below] {$\eta(A')$} (m-2-2)
            (m-1-1) edge node[left] {
                $F(f) = F\big(\sum_{\beta \in \alpha} f \restr I_\beta\big)$
            } (m-2-1)
            (m-1-2) edge node[right] {
                $G(f) = G\big(\sum_{\beta \in \alpha} f \restr I_\beta\big)$
            } (m-2-2)
        ;
        \end{tikzpicture}
    \]
\end{rem}

\subsection{Cake pieces}

To construct the spectral sequence for finite ideal decompositions,
we will define function algebras over
certain subsets of the standard simplex.

\begin{dfn}[Cake piece]
\label{dfn:cakePieces}
    Fix $n \in \nat$. The standard \emph{$n$-simplex} is the
    topological space
    \[
        \Delta^n = \set{(x_0, x_1, \ldots, x_n) \in [0,1]^{n+1}}
        {\sum_{i=0}^n x_i = 1}.
    \]
    Its boundary $\partial \Delta^n$ shall be the subset of points with
    at least one zero entry.
    Let $j \in \nat$ be an index with $j \leq n$. This index defines the
    \emph{$j$-th cake piece}
    \[
        \Delta^n_j = \set{(x_0, x_1, \ldots, x_n) \in \Delta^n}
        {x_j \leq x_{i} \textrm{ for all } i \leq n}.
    \]
    Let $J \subseteq \{0, 1, \ldots, n\}$ be a nonempty subset of the $n+1$
    indices. This determines an intersection of cake pieces:
    \[
        \Delta^n_J = \bigcap_{j \in J} \Delta^n_j.
    \]
    We will see how $\Delta^n_J$ behaves very much like a $j$-th cake piece,
    and therefore also call it a \emph{cake piece}.
\end{dfn}

\BeginSimonFigure
\begin{tikzpicture}
	\coordinate (B1) at (6 cm, 0 cm);
	\coordinate (B2) at (10 cm, 0 cm);
	\coordinate (B3) at (8 cm, 3.464 cm);
    \coordinate (Cb) at (barycentric cs:B1=1,B2=1,B3=1);
    \fill[gray!50] (B2) -- (Cb) -- (B3) -- cycle;
	\draw[thick] (B2) -- (Cb) -- (B3) -- cycle;
    \draw[thick,dashed] (B3) -- (B1) -- (B2);
    \eiddot{B1}{below}{(1, 0, 0)}
    \eiddot{B2}{below}{(0, 1, 0)}
    \eiddot{B3}{above}{(0, 0, 1)}
    \eiddot{Cb}{below}{}
    \node at (barycentric cs:B1=-0.23,Cb=1) {$\Delta^2_{\{0\}}$};
\end{tikzpicture}
\EndSimonFigure{fig:cakeExample}{
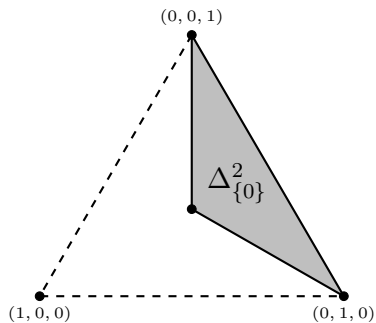
    The simplex $\Delta^2$ with the cake piece
    $\Delta^2_{\{0\}} = \Delta^2_0 \subseteq \Delta^2$ marked
}

Cake pieces are closed subsets of $\Delta^n$.
The central point $\big(\frac1{n+1}, \frac1{n+1}, \ldots, \frac1{n+1}\big)$
of the $n$-simplex
is part of every cake piece; this point is the only element of
$\Delta^n_J$ for the full set $J = \{0, 1, \ldots, n\}$.

Arbitrary
index subsets $J \neq \varemptyset$ make $\Delta^n_J$ look like
$\Delta^{n+1- \card J}$:

\begin{pro}
\label{pro:howdeltanlookslike}
    For a nonempty $J \subseteq \{0, 1, \ldots, n\}$, the subset
    $\Delta^n_J$ is the image of
    $\Delta^{n+1-\card J} \times \{0\}^{\card J -1}$
    under a nondegenerate affine transformation in $\rel^{n+1}$.
\end{pro}

\begin{proof}
    We will analyze several cases explicitly by cardinality of $J$.
    \paragraph{Full set.}
    For $J = \setton$, the full set of $(n+1)$ elements,
    we have already argued how $\Delta^n_J$
    contains only a single point. This is an image of
    $\Delta^0 \times \{0\}^n \subseteq \rel^{n+1}$.

    \paragraph{One element.}
    For $n \geq 2$ and $J = \{j\}$, we have $\Delta^n_J = \Delta^n_j$.
    Without loss of generality, choose $J = \{0\}$.
    The affine
    transformation of $\rel^{n+1}$ to get $\Delta^n_0$
    from the standard $n$-simplex is
    \[
        f = (f_0, f_1, \ldots, f_n)\colon \rel^{n+1} \to \rel^{n+1},
    \]\[
        f_0(x) = \frac{x_0}{n+1},
        \qquad f_i(x) = x_i + \frac{x_0}{n+1}
        \textrm{ for } i \neq 0.
    \]
    The purpose of $f_i(x)$ is to equally distribute among the $n$ other
    coordinates the value $\frac{nx_0}{n+1}$ that has been taken away from
    $x_0$.

    The $f_i$ are nontrivial linear maps, and their direct product
    $f = (f_0, \ldots, f_n)$
    is an affine automorphism of $\rel^{n+1}$. Its inverse is
    \[
        f^{-1} = g = (g_0, g_1, \ldots, g_n)\colon \rel^{n+1} \to \rel^{n+1},
    \]\[
        g_0(x) = (n+1) \cdot x_0,
        \qquad g_i(x) = x_i - x_0
        \textrm{ for } i \neq 0.
    \]
    Even though these maps are defined in $\rel^{n+1}$, they restrict
    well to maps on the simplex: For $x = (x_0, \ldots, x_n) \in \Delta^n$,
    we have $\sum_{i=0}^n x_i = 1 = \sum_{i=0}^n f(x_i)$.
    Furthermore, $f \restr \Delta^n$ maps into $\Delta^n_0$
    because all points $x \in \Delta^n$ satisfy $f_i(x) \geq f_0(x)$.
    The restricted inverse $f^{-1} \restr \Delta^n_0$ maps into $\Delta^n$:
    Positive coordinates stay positive because
    $x_i \geq x_0$ for all $0 \leq i \leq n$.

    \paragraph{Several elements.}
    For $1 < \card J \leq n$, observe how $\card J$ coordinates
    in $\Delta^n_J$ remain equal to each other at all times, and are always
    the smallest. Without loss of generality, let $J$ be
    $\{0, 1, \ldots, \card J - 1\}$, the $\card J$ first coordinates.
    We will construct an affine isomorphism
    \[
        h \colon \Delta^{n+1-\card{J}}_{\{0\}} \to \Delta^n_J
    \]
    by defining $h$ on the $(n + 2 - \card{J})$ corners of
    $\Delta^{n+1-\card{J}}_{\{0\}}$, then extending $h$ to the entire
    cake piece by preserving convex combinations.
    Thereby, $h$ reduces the case of $\Delta^n_J$ to the already-proven
    case $\card J = 1$.

    The central point of $\Delta^{n+1-\card{J}}$ is a corner of
    $\Delta^{n+1-\card{J}}_{\{0\}}$. Have $h$ map this point to the center
    of $\Delta^n$, this is extremal in $\Delta^n_J$.
    Biject the remaining $(n + 1 - \card{J})$
    corners of $\Delta^{n+1-\card{J}}_{\{0\}}$ to the
    corners in $\Delta^n$ that belong to the coordinates in $\setton - J$;
    these points remain extremal in $\Delta^n_J$. This bijection can
    even be chosen to preserve the order of coordinates.
\end{proof}

\begin{cor}
\label{cor:howDeltaNLooksLike}
    Let $J \subseteq \setton$ be nonempty.
    Then $\Delta^{n+1-\card{J}} \cong \Delta^n_J \cong D^{n+1-\card{J}}$,
    where $D^{n+1-\card{J}}$
    denotes the $(n+1-\card{J})$-dimensional unit disk.

    In particular, if $J = \{j\}$, then $\Delta^n_j \cong D^{n}$.
\end{cor}

\begin{proof}
    This follows from $\Delta^n \cong D^n$ and Proposition
    \ref{pro:howdeltanlookslike}.
\end{proof}

These technical constructions relate various subspaces of simplices
to disks. The boundaries of disks are spheres. This will become useful
once we consider C*-algebras of functions
on these subspaces of simplices: When we force functions to vanish on the
boundaries, the C*-algebras can be viewed as suspensions of other algebras.

\subsection{Cake algebras}
\label{sec:cakePieceAlgebras}

\begin{dfn}[Cake algebra]
\label{dfn:cakealgebra}
    Fix $n \in \nat$. Let $A$ be a C*-algebra and $I_j \subseteq A$
    be closed two-sided ideals for $j \in \setton$ with $A = \sum_{j=0}^n I_j$.
    This gives rise to a suspension-like C*-algebra, the
    \emph{cake algebra}
    \equationDefineB
    For $J \subseteq \setton$, define the sub-C*-algebra
    $B_J \subseteq B$, again called a \emph{cake algebra}, by
    \equationDefineBJ
\end{dfn}

\begin{rem}
\label{rem:cakealgebra}
    We observe $B_\setton = B$ and $B_\varemptyset = 0$. Larger index sets
    mean larger function algebras because fewer restrictions apply. Whenever
    $J' \subseteq J$ is a subset, then $B_{J'} \subseteq B_J$ is a
    subalgebra. For $J = \{j\}$, we can characterize $B_{\{j\}}$:
\end{rem}

\begin{pro}
\label{pro:isnfoldsuspension}
    $B_{\{j\}}$ is isomorphic to the $n$-fold C*-algebra suspension of $I_j$.
\end{pro}

\begin{proof}
    By $\partial\Delta^n_j$, we denote the topological
    boundary of $\Delta^n_j$ as a
    subset of $\rel^{n+1}$. A point $x \in \partial\Delta^n_j$
    lies in $\partial \Delta^n$
    if $x_j = 0$. Otherwise, we have $x_j = x_{j'}$ for an index
    $j' \neq j$ and $x$ then lies in $\Delta^n_{j'}$.

    Understanding this, we simplify the above definition for $B_J = B_{\{j\}}$:
    \begin{align*}
        B_{\{j\}}
        &= \set{f\colon \Delta^n \to A}
            {f \restr \partial \Delta^n = 0
            \textrm{, } f(\Delta^n_j) \subseteq I_j
            \textrm{, } f(\Delta^n_{j'}) = 0 \textrm{ for all }
            j' \neq j} \\
        &= \set{f\colon \Delta^n \to A}
            {f \restr \partial \Delta^n = 0
            \textrm{, } f(\Delta^n_j) \subseteq I_j
            \textrm{, } f(\Delta^n - \Delta^n_j) = 0 } \\
        &\cong \set{f\colon \Delta^n_j \to I_j}
            {f \restr \partial \Delta^n_j = 0}.
    \end{align*}
    Because $\Delta^n_{\{j\}} \cong D^n$, the algebra $B_{\{j\}}$
    is isomorphic to the $n$-fold C*-algebra suspension of $I_j$;
    i.e., the $A$-valued functions on $D^n$
    that vanish on the boundary $\partial D^n$.
\end{proof}

\begin{rem}
    As subsets of functions that vanish on a given set, the $B_J$ are
    closed two-sided ideals in $B$. For $J' \subseteq J$, $B_{J'}$ is a
    closed two-sided ideal in $B_J$.
\end{rem}

\begin{lem}
\label{lem:bjcapbjisbjcapj}
    For subsets $J$ and $J'$ of $\setton$, we have
    $B_J \cap B_{J'} = B_{J \cap J'}$.
\end{lem}

\begin{proof}
    This is immediate from the definition of $B_J$:
    \begin{align*}
        B_J \cap B_{J'}
        &=
        \set{f \in B}{f(\Delta^n_{j'})=0 \textrm{ for } j' \notin J}
        \cap
        \set{f \in B}{f(\Delta^n_{j'})=0 \textrm{ for } j' \notin J'}
        \\
        &= \set{f \in B}{f(\Delta^n_{j'})=0
            \textrm{ for } j' \textrm{ with } j' \notin J
            \textrm{ or } j' \notin J'}
        \\
        &= B_{J \cap J'}. \qedhere
    \end{align*}
\end{proof}

\begin{dfn}[Cake sums $Q_p$ for $p \in \zet$]
\label{dfn:cakeSums}
    Let $A$ and $B(I_0, I_1, \ldots, I_n)$ be as in Definition
    \ref{dfn:cakealgebra}. For $p \in \zet$, define the C*-algebra $Q_p$,
    called a \emph{cake sum}, by
    \[
        Q_p = \sum_{\card J \leq p+1} B_J,
    \]
    where $J$ ranges over all subsets of $\setton$ that have cardinality
    $(p+1)$ or less.
\end{dfn}

\begin{rem}[Cake sums for $p < 0$ or $p \geq n$]
\label{rem:cakeSumTrivial}
    For $p < 0$, the sum $Q_p$ is either $B_\varemptyset$ or an empty sum;
    both of these are the zero algebra $Q_p = 0$.

    For $p \geq n$, the sum $Q_p$ is taken over all $B_J$ for all possible
    subsets $J \subseteq \setton$ including $\setton$ itself. By Remark
    \ref{rem:cakealgebra}, for all $J \subseteq \setton$,
    the cake algebra $B_J$ is
    already a subalgebra of $B_{\setton}$.
    Thus $Q_p$ is identical to $B_\setton = B$ for $p \geq n$.
\end{rem}

\begin{rem}[Inclusions $Q_{p-1} \subseteq Q_p$]
\label{rem:cakeSumInclusions}
    We have
    well-defined inclusions $Q_{p-1} \subseteq Q_p$ for all $p \in \zet$
    because $Q_p$ collects at least the cake algebras from $Q_{p-1}$,
    possibly more.

    If $p \leq n$, the
    relation $B_{J'} \subseteq B_J$ for $J' \subseteq J$
    allows another characterization of
    $Q_p$ by summing over fewer sets:
    \begin{equation}
    \label{eqn:cakeSumCharacterization}
        Q_p = \sum_{\card J \leq p+1} B_J = \sum_{\card J = p+1} B_J.
    \end{equation}
    For $p > n$, this characterization would be false because there are
    no subsets of cardinality $(n+2)$ in $\setton$.
\end{rem}

\begin{lem}
\label{lem:cakeSumQuotients}
    For all $p' \leq p \in \zet$, the cake sum $Q_{p'}$ is a C*-ideal
    in $Q_p$. Thus we have well-defined quotients $Q_p/Q_{p'}$, in particular
    $Q_p/Q_{p-1}$.
\end{lem}

\begin{proof}
    For $p' \geq n$, we trivially have $Q_{p'} = Q_p = B$.

    We will prove the case $p' < n$.
    $Q_{p'}$ and $Q_p$ are sub-C*-algebras of the same
    commutative C*-algebra $B$ and we have $Q_{p'} \subseteq Q_p$.
    It remains to show that $Q_{p'}$ is an algebraic ideal.

    For $f \in Q_{p'}$ and $g \in Q_p$, find $J \subseteq \setton$
    such that $f \in B_J$ and $\card{J} = p' + 1$; such a $J$ exists
    according to the characterization \ref{eqn:cakeSumCharacterization}.

    By Definition \ref{dfn:cakealgebra} of $B_J$, the function $f$ must
    vanish on at least $(n-p')$ different cake pieces.
    For any $g \in Q_p$, the pointwise product $fg$ must vanish on the same
    cake pieces, therefore $fg \in Q_{p'}$.
\end{proof}

\begin{rem}
    The algebra $Q_p$ is defined by summing over all
    $B_J$ with $\card J \leq p+1$,
    not merely over those with $\card J \leq p$.
    This index shift is deliberate:
    We are going to define a spectral sequence with the
    K-theory of quotients of
    $Q_p/Q_{p-1}$ as input.
    The index shift
    will affect the layout of the first page $\{E^1_{p,q}\}_{p,q\in \zet}$.

    Consider the trivial input $n = 0$ and $A = I_0$.
    Here $Q_p = B_{\{0\}}$ for $p \geq 0$ and $Q_p = 0$ for $p < 0$.
    This leads to a spectral sequence with $E^1_{0,q} \cong K_q A$ and
    $E^1_{p,q} = 0$ for $p \neq 0$. This is the most desirable layout;
    the K-theory of the lone ideal $A = I_0$ is not shifted in any way:
    \[
    \begin{tikzpicture}
    \masseq{
        \phantom{-}2 & 0 & K_2 A
        & 0
        & 0
        \\
        \phantom{-}1 & 0 & K_1 A
        & 0
        & 0
        \\
        \phantom{-}0 & 0 & K_0 A
        & 0 & 0
        \\
        -1 & 0 & K_{-1} A
        & 0 & 0
        \\
        & -1  &  0  &  1  & 2 \\
        };
    \drawSseqAxes{5}{5}
    \end{tikzpicture}
    \]
\end{rem}

\subsection{K-theory of cake algebras}
\label{subsection:kTheoryCake}

We started with a sum of ideals $I_0$, $I_1$, $\ldots$, $I_n$ and have
developed a chain of ideals
$\dotsb \subseteq Q_p \subseteq Q_{p+1} \subseteq \dotsb$.
The main theorem of Section \ref{subsection:kTheoryCake} relates the K-theory
of this chain to the K-theory of the original ideals:

\begin{thm}
\label{thm:sumofnfoldsuspensions}
    For $A = I_0 + I_1 + \dotsb + I_n$ and the cake sums
    $Q_p$ defined as before, given
    $p \in \setton$ and $q \in \zet$, the K-theory
    of $Q_p/Q_{p-1}$ decomposes as
    \[
        \equationsumofnfoldsuspensions.
    \]
\end{thm}

\begin{rem}
    For $p \notin \setton$, the K-theory $K_*(Q_p/Q_{p-1})$ vanishes because
    $Q_p = Q_{p-1}$.
\end{rem}

\begin{exm}
    Before we prove Theorem \ref{thm:sumofnfoldsuspensions}
    for all $p \in \setton$, we look at the simplest case,
    $p = 0$. One-fold intersections of ideals are merely the ideals themselves.
    The above formula reduces to
    \[
        K_{q} \left(Q_0/Q_{-1}\right)
        \cong \bigoplus_{j \leq n} K_{q+n} I_j.
    \]
    Inserting the definitions $Q_{-1} = 0$ and
    $Q_0 = \sum_{j \leq n} B_{\{j\}}$, we rewrite our claim to
    \[
        K_{q}\Big(\sum_{j \leq n} B_{\{j\}}\Big)
        \cong \bigoplus_{j \leq n} K_{q+n} I_j.
    \]
    Lemma \ref{lem:bjcapbjisbjcapj} implies that
    $B_{\{j\}}$ and $B_{\{j'\}}$ overlap trivially
    as algebras for $j \neq j'$;
    i.e., $B_{\{j\}} \cap B_{\{j'\}}$ contains only the zero function.
    The sum $\sum_{j \leq n} B_{\{j\}}$
    on the left-hand side is therefore isomorphic to a direct sum
    $\bigoplus_{j \leq n} B_{\{j\}}$ of the function spaces $B_{\{j\}}$.

    In Proposition \ref{pro:isnfoldsuspension}, we have shown that $B_{\{j\}}$
    is isomorphic to the $n$-fold suspension of $I_j$, providing the
    desired shift by $n$ degrees,
    $K_{q} B_{\{j\}} \cong K_{q} S^n I_j \cong K_{q+n} I_j$.

    Because taking K-theory commutes with taking direct sums,
    we have shown the theorem for $p = 0$.
\end{exm}

The main ingredient $B_{\{j\}} \cap B_{\{j'\}} = B_\varemptyset = 0$
must now be generalized to prove Theorem \ref{thm:sumofnfoldsuspensions}
for $p > 0$.

\begin{lem}
\label{lem:forSumOfNFoldSuspensions}
    Let $J' \neq J''$ be nonempty $(p+1)$-element subsets of $\setton$
    and let $f \in Q_{p}$ lie in
    the intersection $B_{J'} \cap B_{J''}$. Then $f$ lies already in
    the next-smaller ideal,
    \[
        f \in Q_{p-1} = \sum_{\card{J} = p} B_{J}.
    \]
\end{lem}

\begin{proof}
    By Lemma
    \ref{lem:bjcapbjisbjcapj}, $B_{J'} \cap B_{J''} = B_{J' \cap J''}$.
    The algebra $B_{J' \cap J''}$ is a summand of
    $Q_{\card{J' \cap J''}}$ This algebra is equal to or
    a subset of $Q_{p-1}$ because $\card{J' \cap J''} \leq p$.
\end{proof}

\begin{lem}
\label{lem:vanishingOnNcakePieces}
    Fix an index subset $J \subs \setton$.
    Let $p \in \nat$ be a cardinality.
    Let $f \in B_J$ vanish on all
    $(p + 1)$-fold intersections of cake pieces:
    $f \restr \Delta^n_L = 0$ when
    $\card L = p + 1$.

    Then $f$ is a finite sum of functions $f_L$ with each
    $f_L \in B_L$ for $L \subs J$ and each occurring set
    $L$ has cardinality $\card{L} \leq p$.
\end{lem}

\begin{rem}
    It follows that $f$ is in $Q_{p-1}$, but
    the claim is stronger: Only summands $B_L$ with $L \subs J$ are
    required to construct $f$ in $Q_{p-1}$.
\end{rem}

\begin{proof}[Proof of Lemma \ref{lem:vanishingOnNcakePieces}]
    We prove this by induction along $p$. The base case is $p = 0$:
    One-fold intersections of cake pieces -- where $f$ vanishes by assumption
    -- are the cake pieces themselves, thus $f = 0$, the only function
    in the zero algebra $B_{\varemptyset}$. This concludes the base case.

    For the induction hypothesis, assume that all functions that vanish
    on $p$-fold intersections are sums of functions from $B_L$ with
    $L \subs J$ and $\card{L} \leq p-1$.
    We will show the claim for $p$: Let $f \in B_J$ vanish on $(p + 1)$-fold
    intersections of cake pieces.

    Consider all subspaces $\Delta^n_{L}$ for $L \subs J$
    with $\card L = p$. We may treat each as a
    topological submanifold of $\rel^{n+1}$ on its own and consider its
    boundary $\partial \Delta^n_L$.
    Each point $x \in \partial \Delta^n_L$ lies
    on the boundary $\partial \Delta^n$ of the entire simplex $\Delta^n$
    or in a cake piece $\Delta^n_j$ with $j \neq L$, see Figure
    \ref{fig:twoFoldIntersectionThreeBoundary}.
    If $x \in \partial \Delta^n$, then $f(x) = 0$ by definition of $B$.
    If $x \in \Delta^n_j$ with $j \neq J$, then also $f(x) = 0$ because
    now $x \in \Delta^n_{L \cup \{j\}}$, a $(p + 1)$-fold intersection of
    cake pieces. Thus $f \restr \partial \Delta^n_L = 0$.

    On the interior $\Delta^n_L - \partial \Delta^n_L$,
    $f$ assumes values in $\bigcap_{j \in L} I_j$ by definition of $B_J$.
    From this and the restriction $f \restr \partial \Delta^n_L = 0$,
    we can find a function $g_{L} \in B_{L}$
    such that $f \restr \Delta^n_{L} = g_L \restr \Delta^n_{L}$.
    After defining $g_L$ for each $L \subs J$
    of cardinality $p$, consider the function
    \[
        f' = f - \sum_{\substack{L \subs J \\ \card{L} = p}} g_L.
    \]
    This $f'$ still lies in the C*-algebra $B_J$ because $B_L \subs B_J$
    for each $L$. Furthermore, $f'$
    vanishes on all $p$-fold intersections of cake pieces,
    not merely on the $(p+1)$-fold intersections.

    By our induction hypothesis,
    $f'$ is a finite sum of functions from $B_L$ for $L \subs J$ of cardinality
    $\card{L} \leq p - 1$.
    Each $g_L$ is in $B_L$ with $L \subs J$ and $\card{L} = p$.
    Since $f = f' + \sum_L g_L$, we have shown the induction case
    for cardinality $p$.
\end{proof}

\BeginSimonFigure
\begin{tikzpicture}
	\coordinate (B1) at (6 cm, 0 cm);
	\coordinate (B2) at (10 cm, 0 cm);
	\coordinate (B3) at (8 cm, 3.464 cm);
    \coordinate (Cb) at (barycentric cs:B1=1,B2=1,B3=1);
	\draw[thick] (Cb) -- (B2);
    \draw[thick,dashed] (B3) -- (B1) -- (B2) -- (B3);
    \eiddot{B1}{below}{(1, 0, 0)}
    \eiddot{B2}{below}{(0, 1, 0)}
    \eiddot{B3}{above}{(0, 0, 1)}
    \eiddot{Cb}{below}{}
    \node at (8.9 cm, 1.05 cm) {$\Delta^2_{\{0,2\}}$};
\end{tikzpicture}
\EndSimonFigure{fig:twoFoldIntersectionThreeBoundary}{
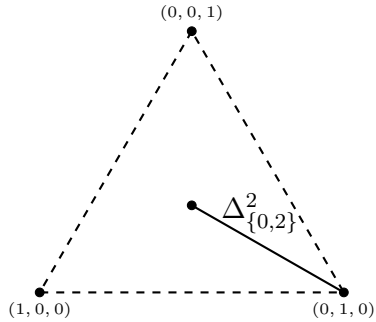
    The two-fold intersection $\Delta^2_{\{0,2\}}$
    and its two-point boundary:
    one point in $\partial \Delta^2$,
    one in the three-fold intersection $\Delta^2_{\{0,1,2\}}$
}

\newcommand{\smallJCup}{\sum_{L \subsetneqq J} B_L}

\begin{lem}
\label{lem:smallerIdealsVanish}
    Let $J \subseteq \setton$ be a nonempty index set.
    \begin{itemize}
    \item
        Let $f$ be a function in $\smallJCup$.
        Then $f \restr \Delta^n_J = 0$.
    \item
        Conversely, let $g$ be a function in $B_J$
        with $g \restr \Delta^n_J = 0$.
        Then $g \in \smallJCup$.
    \end{itemize}
\end{lem}

\begin{proof}
    We have $f \in \smallJCup$.
    Each $L \subsetneqq J$ lacks at least one index $j \in J$, therefore
    $f \restr \Delta^n_j = 0$ by definition of $B_{L}$.
    Since $\Delta^n_J$ is contained in the boundary $\partial\Delta^n_j$,
    we conclude $f \restr \Delta^n_J = 0$.

    Conversely, let $g \in B_J$ vanish on $\Delta^n_J$.
    All functions in $B_J$ vanish on $(\card J + 1)$-fold
    intersections of cake pieces. Furthermore, $\Delta^n_J$
    is the only $\card{J}$-fold intersection
    that touches the interior of the support of $g$.
    Thus $g$ vanishes on all $\card{J}$-fold intersections.
    By Lemma \ref{lem:vanishingOnNcakePieces},
    $g$ is a sum of functions $g_L$ from $B_L$ with
    $\card L < \card J$ and $L \subs J$.
\end{proof}

In the following technical proposition, $A$ and $B$ are general C*-algebras;
they
need not coincide with $B(I_0, I_1, \ldots, I_n)$ that we defined before.
Nonetheless, we choose the names $A$ and $B$ here because we will later apply
this result to the $B_J$ from Theorem \ref{thm:sumofnfoldsuspensions}.

\newcommand{\vanD}{\mathrm{Van}_D}

\begin{pro}
\label{pro:werner}
    Let $X \subseteq \rel^n$ be a compact set and
    let $D \subseteq X$ be a compact subspace of $X$.
    Let $A$ be a C*-algebra, $B \subseteq \cont(X, A)$ a C*-ideal
    of functions from $X$ to $A$, and
    $\vanD \subseteq B$ the vanishing ideal of $D$; i.e., the ideal
    of functions $f \in B$ with $f \restr D = 0$.

    Then $B/\vanD$ is isomorphic as a C*-algebra to
    $B' = \set{f \restr D}{f \in B}$.
\end{pro}

This is plausible: When we enlarge $D$, then functions in $\vanD$
are allowed less variation, thus $\vanD$ becomes smaller and
the quotient space $B/\vanD$ becomes larger.

\begin{proof}
    Define the operator $T\colon B/\vanD \to B'$ by
    $T[f] = f \restr D$. This is a well-defined linear map because
    for $f$ and $f' \in [f] \in B/\vanD$, we have
    $f - f' \in \vanD$, therefore $f \restr D = f' \restr D$.

    $T$ is a continuous operator with norm
    \[
        \norm{T}
        = \sup \set{ \norm{f \restr D} }
            { f \in B \textrm{ with } \norm{[f]} \leq 1},
    \]
    where $\norm{[f]} = \inf \set{\norm{f - g}}{g \in \vanD}$.
    From $\norm{f \restr D} \leq \norm{[f]} \leq \norm f$, we see that $T$ is
    continuous with norm $\norm T \leq 1$.

    $T$ is bijective: If $T[f] = 0$, then $f \restr D = 0$, therefore
    $f \in \vanD$ and $[f] = 0 \in B/\vanD$.
    On the other hand, given $f \restr D \in B'$ with $f \in B$,
    surely $[f]$ is a preimage of $f \restr D$ under $T$.

    As a restriction of functions, $T$ preserves products
    and the C*-involution. Together with bijectivity of $T$,
    we conclude that $\norm T = 1$ and that $T$ is an isometric
    *-isomorphism by \cite[Theorem I.5.5]{davidson1996}.
\end{proof}

We could have obtained $\norm T = 1$ from an analytical argument, too:
Given $f \restr D$,
force $f\colon X \to A$ to decay rapidly outside $D \subseteq X$
by multiplying with bump functions. $B$ is an ideal in $\cont(X, A)$,
and $A$ admits an approximate unit.

We shall now return to the setting where $A = I_0 + I_1 + \dotsb + I_n$
is a sum of ideals and $B = B(I_0, I_1, \ldots, I_n)$ is the function algebra
constructed over the $n$-simplex.

\begin{lem}
\label{lem:smallerIdealsKtheory}
    Let $J \subseteq \setton$ be a nonempty
    index set of cardinality $\card J = p+1$.
    Then
    \[
        B_J / (B_J \cap Q_{p-1})
        \cong S^{n-p}\Big( \bigcap_{j \in J} I_j \Big),
    \]
    where $S^{n-p}$ denotes the $(n-p)$-fold suspension of C*-algebras.
\end{lem}

Recall that $Q_p$ was a sum over all $B_{L}$ with index sets $L$
of cardinality $\card{L} = p+1$,
thus $B_J \subseteq Q_p$ and dividing by
the intersection $(B_J \cap Q_{p-1}) \subseteq Q_p$ is meaningful.

\begin{proof}
    All functions in $B_J$ are supported in $\bigcup_{j \in J} \Delta^n_j$.
    Write $D = \Delta^n_J = \bigcap_{j \in J} \Delta^n_j$
    for the subset of $\bigcup_{j \in J} \Delta^n_j$
    where functions in $B_J$ may take
    nonzero values in all $I_j$ for $j \in J$ simultaneously.
    Now
    \[
        B_J \cap Q_{p-1} = \sum_{L \subsetneqq J} B_{L},
    \]
    and, by Lemma \ref{lem:smallerIdealsVanish}, functions in
    $B_J \cap Q_{p-1}$ are exactly those functions in $B_J$ that vanish on $D$.
    We can apply Proposition \ref{pro:werner} to the function algebra $B_J$
    on the base space $X = \bigcup_{j \in J} \Delta^n_j$ and its subset $D$
    to get
    \begin{equation}
    \label{eqn:useWernerOnBJ}
        B_J / (B_J \cap Q_{p-1}) \cong
        \set{f \restr D}{f \in B_J}.
    \end{equation}
    Considering $D = \Delta^n_J$ an
    $(n+1-\card{J})$-dimensional topological manifold on its own,
    $\Delta^n_J$ itself has a boundary $\del \Delta^n_J$ and
    a nontrivial interior $(\Delta^n_J - \del \Delta^n_J)$.
    In the edge case where
    $J = \setton$ is the full set, $\Delta^n_{\setton}$
    is a single point, which is a zero-dimensional manifold with empty
    boundary $\del \Delta^n_{\setton} = \varemptyset$.

    The boundary $\del \Delta^n_J$ is contained in the boundary
    of the original domain $\bigcup_{j \in J} \Delta^n_j$. Therefore functions
    in $B_J$, even when restricted to $D$ as in
    \ref{eqn:useWernerOnBJ},
    must still vanish on this new boundary $\del \Delta^n_J$.

    On the interior of $D = \Delta^n_J$,
    functions in $B_J$ must take values in $\bigcap_{j \in J} I_j$ by
    definition of $B_J$, but no further restrictions apply.
    We can rewrite \ref{eqn:useWernerOnBJ} as
    \[
        B_J / ( B_J \cap Q_{p-1} )
        \cong \set{ f\colon \Delta^n_J \to \bigcap_{j \in J} I_j }
            { f \textrm{ is continuous and } f \restr \del \Delta^n_J = 0 }.
    \]
    Finally, by Lemma \ref{cor:howDeltaNLooksLike},
    $\Delta^n_J \cong \Delta^{n+1-\card{J}} = \Delta^{n-p}$
    is homeomorphic to the $(n-p)$-dimensional
    unit disk. This allows us to further rewrite the algebra
    $B_J / ( B_J \cap Q_{p-1} )$
    as the $(n-p)$-fold suspension in the claim.
\end{proof}

We are now ready to prove the main theorem about the K-theory
of the chain of ideals $Q_p$.

\begin{proof}[Proof of Theorem \ref{thm:sumofnfoldsuspensions}]
    Fix $p \in \setton$ and $q \in \zet$.
    With $Q_p = \sum_{\card J = {p+1}} B_J$
    according to the characterization \ref{eqn:cakeSumCharacterization},
    we have to show:
    \[
        \equationsumofnfoldsuspensions.
    \]
    First, we will show that the quotient $Q_p/Q_{p-1}$ decomposes as
    a direct sum. Let
    $f \in Q_p$ lie in the images of different inclusions $B_J \to Q_p$ and
    $B_{J'} \to Q_p$ for $\card J = \card{J'} = p+1$. By Lemma
    \ref{lem:forSumOfNFoldSuspensions}, we have $f \in Q_{p-1}$ and therefore
    $[f] = 0 \in Q_p/Q_{p-1}$. This shows that $Q_p/Q_{p-1}$ is a direct
    sum. Each summand corresponds to one $B_J$ with $\card J = p+1$:
    \[
        Q_p/Q_{p-1} = \bigoplus_{\card{J} = p+1} B_J / (B_J \cap Q_{p-1}).
    \]
    For each $J$, we computed
    $B_J / (B_J \cap Q_{p-1}) \cong S^{n-p} \big( \bigcap_{j \in J} \big)$
    in Lemma \ref{lem:smallerIdealsKtheory}.
    Passing to K-theory, we can replace the $(n-p)$-fold
    suspension with a degree shift by $(n-p)$:
    \[
        K_{p+q}\big( B_J / (B_J \cap Q_{p-1}) \big)
        \cong
        K_{p+q} \bigg( S^{n-p} \Big( \bigcap_{j \in J} I_j \Big)\bigg)
        \cong
        K_{q+n} \Big( \bigcap_{j \in J} I_j \Big).
    \]
    The claim follows because taking K-theory commutes
    with taking direct sums.
\end{proof}

\begin{thm}
\label{thm:inclusioninducesiso}
    The inclusion of algebras $B \to
    \set{f\colon \Delta^n \to A}{f \restr \partial \Delta^n = 0}$
    induces an isomorphism in K-theory.
\end{thm}

The following Lemmas \ref{lem:inclusioninducesiso:first}
to \ref{lem:inclusioninducesiso:last} will prove this theorem.
Define the following intermediate algebras:
\begin{align*}
    R_0 &= \set{f\colon \Delta^n \to A}{f \restr \partial \Delta^n = 0}, \\
    R_1 &= R_0 \cap \set{f}{f(\Delta^n_0) \subseteq I_0}, \\
    R_2 &= R_1 \cap \set{f}{f(\Delta^n_1) \subseteq I_1}, \\
    & \ \, \vdots \\
    B = R_{n+1} &= R_n \cap \set{f}{f(\Delta^n_n) \subseteq I_n}.
\end{align*}
To show that $B \to R_0$ is an isomorphism in
K-theory, we show that each inclusion
\[
    \incl\colon R_k \to R_{k-1}
\]
induces an isomorphism for $k \in \{n+1, n, \ldots, 2, 1\}$.

\begin{lem}
\label{lem:inclusioninducesiso:first}
    For $k > k'$, the algebra $R_k$ is a C*-ideal in $R_{k'}$.
\end{lem}

\begin{proof}
    The additional restrictions to the set of functions in $R_k$ over
    $R_{k'}$ forces the functions to map points into the given
    C*-ideals of $A$
    instead of anywhere in $A$. Because all $I_0, \ldots, I_n$ are C*-ideals
    and the multiplication of functions happens pointwise, $R_k$ becomes
    a C*-ideal in $R_{k'}$.
\end{proof}

\begin{lem}
\label{lem:pairishomeomorphic}
    The pair of topological spaces
    $(\Delta^n_k, \partial \Delta^n \cap \Delta^n_k)$
    is homeomorphic to
    $(D^{n-1} \times [0,1], D^{n-1} \times \{1\})$.
\end{lem}

\begin{proof}
    $\partial \Delta^n \cap \Delta^n_k$ is exactly the $k$-th face of the
    $n$-simplex.
    In Proposition \ref{pro:howdeltanlookslike}, we have seen how
    $\Delta^n_n \cong \Delta^n$. In particular, for one-element
    sets $J = \{n\}$, the proof shows how $\Delta^n_J$ and $\Delta^n$ are
    diffeomorphic via a stretch by the factor $(n+1)$. This stretch has
    $\partial \Delta^n \cap \Delta^n_n$ as a set of fixed points.
    $\Delta^n_n$ and $\Delta^n_k$ are certainly homeomorphic.

    The cake piece $\Delta^n_k$ is a compact, convex $n$-dimensional manifold
    within $\rel^{n+1}$ and $\partial \Delta^n \cap \Delta^n_k$
    is a convex $(n-1)$-dimensional hypersurface within $\partial \Delta^n_k$.
    Corollary \ref{cor:howDeltaNLooksLike} relates the simplices to
    the desired disks.
\end{proof}

\begin{lem}
\label{lem:quotientkvanishes}
    For all $k \in \{1, 2, \ldots, n+1\}$,
    the quotient $R_{k-1} / R_k$ has trivial K-theory.
\end{lem}

\begin{proof}
    The subset $\Delta^n - \Delta^n_{k-1}$ is open in the
    entire space $\Delta^n$.
    With the convention that $\{0, 1, \ldots, -1\}$ denotes
    the empty set and that $\{0, 1, \ldots, 0\} = \{0\}$, we compute:
    \begin{align*}
        R_{k-1} / R_k
        &=
        \frac{\set{f\colon \Delta^n \to A}{f(\partial \Delta^n) = 0
        \textrm{ and } f(\Delta^n_j) \subseteq I_j
        \textrm{ for } j \in \{0, 1, \ldots, k-2\}}}
        {\set{f\colon \Delta^n \to A}{f(\partial \Delta^n) = 0
        \textrm{ and } f(\Delta^n_j) \subseteq I_j
        \textrm{ for } j \in \{0, 1, \ldots, k-1\}}}
        \\
        &\cong
        \frac{
            \set{f\colon \Delta^n_{k-1} \to A}
            {f(\partial \Delta^n \cap \Delta^n_{k-1}) = 0}
            \phantom{\textrm{ and } f(\Delta^n_{k-1}) \subseteq I_{k-1}}
        }{
            \set{f\colon \Delta^n_{k-1} \to A}
            {f(\partial \Delta^n \cap \Delta^n_{k-1}) = 0
            \textrm{ and } f(\Delta^n_{k-1}) \subseteq I_{k-1}}
        }
        \\
        &\cong
        \set{f\colon \Delta^n_{k-1} \to A/I_{k-1}}
        {f(\partial \Delta^n \cap \Delta^n_{k-1}) = 0}.
    \end{align*}
    Because of Lemma \ref{lem:pairishomeomorphic}, the quotient
    $R_{k-1} / R_k$
    is isomorphic to the algebra
    \[
        R' = \set{f\colon D^{n-1} \times [0,1] \to A/I_{k-1}}
                 {f\left(D^{n-1} \times \{1\}\right) = 0}.
    \]
    This is a contractible algebra: The homotopy $h\colon R' \times I \to R'$,
    \[
        h(f, t)(x, t') = f(x, t' \cdot t),
    \]
    defines a \starhom{} for each fixed $t$. This construction is
    analogous to the proof of Proposition \ref{pro:coneiscontractible}
    for the contractibility of cone algebras.
    Since K-theory is homotopy invariant and $h(f, 0) = 0$ for all $f$,
    we conclude that $K_*(R_{k-1}/R_k) = K_* R' = 0$.
\end{proof}

\begin{lem}
\label{lem:inclusioninducesiso:last}
    $\incl\colon R_k \to R_{k-1}$ induces an isomorphism in K-theory.
\end{lem}

\begin{proof}
    We examine the six-term exact sequence associated to the inclusion
    of the ideal.
    \[
        \begin{tikzpicture}
            \mama {
                K_0 R_k &
                K_0 R_{k-1} &
                K_0(R_{k-1} / R_k)
            \\
                K_1(R_{k-1} / R_k) &
                K_1 R_{k-1} &
                K_1 R_k.
            \\};
            \path[->, font=\scriptsize]
                (m-1-1) edge node[auto] {$\incl_0$} (m-1-2)
                (m-1-2) edge node[auto] {$\proj_0$} (m-1-3)
                (m-1-3) edge node[auto] {$\partial_2 \circ \beta$} (m-2-3)
                (m-2-3) edge node[auto] {$\incl_1$} (m-2-2)
                (m-2-2) edge node[auto] {$\proj_1$} (m-2-1)
                (m-2-1) edge node[auto] {$\partial_1$} (m-1-1)
            ;
        \end{tikzpicture}
    \]
    Since $K_p(R_{k-1} / R_k)$ vanishes for both even and odd $p$
    as shown in Lemma \ref{lem:quotientkvanishes},
    $\incl_p$ is an isomorphism for all $p$.
    This also concludes the proof of Theorem \ref{thm:inclusioninducesiso},
    which is an $n$-fold application of these lemmas.
\end{proof}

\subsection{Review of developed theory}
\label{subsection:reviewofcake}

Let $A$ be a C*-algebra that can be written as a finite sum
$I_0 + I_1 + \dotsb + I_n = A$ of
closed two-sided ideals $I_j \subs A$.
For the \emph{cake pieces} $\Delta^n_j \subseteq \Delta^n$,
we have constructed in Definition \ref{dfn:cakealgebra} a new C*-algebra $B$
of functions into $A$,
\equationDefineB
We worked with arbitrary index subsets $J \subseteq \setton$. For such a $J$,
we have defined
\equationDefineBJ
These \emph{cake algebras} $B_J$ are closed two-sided ideals in
$B = B_{\setton}$.
For $p \in \zet$, we defined the \emph{cake sums}
\[
    Q_p = \sum_{\card J \leq p+1} B_J.
\]
There are inclusions $Q_{p-1} \to Q_p$ and quotients of
C*-ideals, $Q_p / Q_{p-1}$.
This inclusion chain of C*-ideals is the decisive structure:
We can later feed these algebras into the spectral sequence
for ideal inclusions.

For $p \in \setton$, Theorem \ref{thm:sumofnfoldsuspensions} computes
\[
    \equationsumofnfoldsuspensions{}.
\]
This expression continues to hold for $p > n$ where $K_{p+q}(Q_p/Q_{p-1}) = 0$:
There are no subsets $J$ with $\card J = n+2$. But the expression fails for
$p = -1$: To avoid inter\-sections over the empty set, we must
explicitly mention $K_{p+q}(Q_p/Q_{p-1}) = 0$ for $p = -1$ whenever we
extend Theorem \ref{thm:sumofnfoldsuspensions} to all $p \in \zet$.

Theorem \ref{thm:inclusioninducesiso} shows that the inclusion
$B \to \set{f\colon \Delta^n \to A}{f \restr \partial \Delta^n = 0}$
induces an isomorphism in K-theory.
The C*-algebra on the right-hand side is isomorphic to the $n$-fold suspension
$S^nA$, therefore $K_p B \cong K_{p+n} A$.
With this, we can go back
to $A$, the algebra of original interest,
even though the quotients $Q_p$ encode information about $B$.

\subsection{Main theorem}

\begin{thm}[Spectral sequence for finite sums of C*-ideals]
\label{thm:sseqforsums}
    Let $A$ be a C*-algebra and $I_0$, $I_1$, $\ldots$, $I_n$ be
    $(n+1)$ C*-ideals in $A$ with
    $I_0 + I_1 + \dotsb + I_n = A$.
    There is a spectral sequence
    $\{E^r_{p,q}, d^r\}_{r,p,q}$ with
    \begin{equation}
    \label{equation:newsseqindices}
        \equationEOneFinite
    \end{equation}
    This spectral sequence converges strongly to $K_* A$.
\end{thm}

This spectral sequence is functorial for \starhoms{} that preserve ideal
decompositions; this will be the next theorem.

\begin{proof}[Proof of Theorem \ref{thm:sseqforsums}]
    For the C*-ideals $I_0$, $I_1$, $\ldots$, $I_n$,
    define cake algebras $B = B(I_0, I_1, \ldots, I_n)$
    and cake sums $Q_p \subseteq B$ for $p \in \zet$ as reviewed in
    Section \ref{subsection:reviewofcake}.
    We have $Q_p = B$ for $p \geq n$ by Remark \ref{rem:cakeSumTrivial}.
    For the series of inclusions
    \[
        \dotsb = 0 = 0 \subseteq Q_0 \subseteq Q_1 \subseteq
        \dotsb \subseteq Q_n = B = Q_{n+1} = Q_{n+2} = \dotsb,
    \]
    Theorem \ref{thm:schochetSseq} gives a spectral sequence
    $\{\bar E^r_{p,q}, \bar d^r\}_{r,p,q}$ with
    \begin{equation}
    \label{equation:oldsseqindices}
        \bar E^1_{p,q} \cong K_{p+q}(Q_p/Q_{p-1}),
    \end{equation}
    converging to $K_*(B)$.
    By Theorem \ref{thm:sumofnfoldsuspensions}, we can replace the K-theory
    of these quotients by the K-theory of a more immediate intersection,
    \begin{equation}
    \label{eqn:newsseqindices}
        \bar E^1_{p,q} \cong K_{p+q}(Q_p/Q_{p-1}) \cong
            \begin{dcases*}
                \bigoplus_{\card J = p+1}
                K_{q+n} \Big( \bigcap_{j \in J} I_j \Big)
                &
                for $0 \leq p \leq n$,
                \\
                0
                &
                for $p < 0$ or $p > n$.
            \end{dcases*}
    \end{equation}
    This spectral sequence converges strongly
    to $K_*\big( \overline{\bigcup_{p \in \zet} Q_p} \big)
    = K_* Q_n = K_* B$. By Theorem \ref{thm:inclusioninducesiso},
    $K_q B \cong K_{q+n} A$. To simplify, we will shift down by $n$ all
    degrees in K-theory, both in \ref{eqn:newsseqindices} and in the
    expression for the convergence. As a result, we obtain
    a new spectral sequence
    $\{E^r_{p,q}, d^r\}_{r,p,q}$ with
    \[
        \equationEOneFinite
    \]
    This spectral sequence converges strongly to $K_* A$.
\end{proof}

\begin{thm}[Functoriality of the spectral sequence from Theorem
    \ref{thm:sseqforsums}]
\label{thm:sseqForSumsIsFunctorial}
    The spectral sequence $\{E^r_{p,q}, d^r\}_{r,p,q}$
    from Theorem \ref{thm:sseqforsums} is functorial with
    respect to \starhoms{} that preserve $(n+1)$-fold ideal decompositions
    (Definition \ref{dfn:preserveIdealDecomps1}):

    For $n' \geq n$,
    let $A' = I'_0 + I'_1 + \dotsb + I'_{n'}$ be a C*-algebra and
    let $f\colon A \to A'$ be a \starhom{} such that $f(I_j) \subseteq I'_j$
    for all $j \leq n$. Let $\{E^r_{p,q}, d^r\}_{r,p,q}$ be the spectral
    sequence from Theorem \ref{thm:sseqforsums} that converges to
    $A$ and let $\{\bar E^r_{p,q}, \bar d^r\}_{r,p,q}$ be the
    spectral sequence that converges to $A'$.

    Then $f$ induces a morphism $\{f^r_{p,q}\}_{r,p,q}$ of spectral sequences
    (Definition \ref{dfn:morphismOfSseq}) of bidegree $(0, 0)$ with
    \[
        f^r_{p,q}\colon E^r_{p,q} \to \bar E^r_{p,q}
    \]
    for all $r \geq 1$ and $p$, $q \in \zet$ that
    commutes with the differentials and, in turn, induces
    $K_*f\colon K_* A \to K_* A'$ on the convergence targets.
\end{thm}

\begin{proof}
    All constructions since Section \ref{sec:cakePieceAlgebras}
    on the level of C*-algebras have been functorial with respect to
    ideal decompositions:
    Cake algebras, sums of cake algebras,
    cones, suspensions.

    Likewise, taking K-theory and constructing
    direct sums of K-theory groups for each nonempty $J \subseteq \setton$
    are functorial
    in the same way. Thus $\{f^r_{p,q}\}_{r,p,q}$ exists and induces
    the correct morphism $\{f^\infty_{p,q}\}_{p,q}\colon
    \{E^\infty_{p,q}\}
    \to \{\bar E^\infty_{p,q}\}$, which induces
    the desired $K_*f\colon K_* A \to K_* A'$ on the convergence targets.
\end{proof}

\section{Finite coarse excision}\thispagestyle{plain}

We generalize coarsely excisive pairs from
Definition \ref{dfn:coarselyExcisivePair} to coarsely excisive covers
of arbitrary cardinality.
Later in this section,
we will apply the spectral sequence from
Theorem \ref{thm:sseqforsums} for finitely many C*-ideals
to C*-algebras obtained from finite coarsely excisive covers.

\subsection{Coarsely excisive covers}

\begin{dfn}[Coarsely excisive cover]
\label{dfn:coarselyExcisiveCover}
    Let $(X, d)$ be a coarse space. Let
    $\{X_\beta\}_{\beta \in \alpha}$ be a cover of $X$ of arbitrary
    cardinality $\card \alpha$
    such that each $X_\beta$ is a closed subset in $X$.

    The cover $\{X_\beta\}_{\beta \in \alpha}$
    is called \emph{coarsely excisive}
    if, for all nonempty finite sets $J \subs \alpha$
    and for all distances $R > 0$, there exists a distance $S > 0$
    such that the intersection of the
    $\card{J}$-many $R$-neighborhoods
    lies in the $S$-neighborhood of the $\card{J}$-fold intersection:
    \begin{equation}
    \label{eqn:coarselyExcisiveCover}
        \bigcap_{j \in J} N_d(X_j, R) \subseteq
        N_d\Big(\bigcap_{j \in J} X_j, S\Big).
    \end{equation}
\end{dfn}

\begin{rem}
\label{rem:coarselyExcisiveCover}
    The distance $S$ may be chosen depending both on $R$ and the
    particular finite subcollection $\{X_j\}_{j \in J}$ at hand.
    It is not required that,
    given $R$, a single $S > 0$ satisfies \ref{eqn:coarselyExcisiveCover}
    uniformly for all
    subcollections of the cover, or even only for all subcollections of
    a given cardinality $\card J$.
\end{rem}

\begin{rem}
    This is a straightforward generalization of coarsely excisive pairs.
    These covers will yield C*-ideals in the Roe algebra of $X$ suitable for
    our spectral sequence. They behave as expected:
    \begin{itemize}
    \item
        A coarsely excisive pair of closed sets
        is a two-set coarsely excisive cover.
    \item
        The extra requirement that each $X_\beta$ be closed in $X$ does not
        affect any coarse properties. An arbitrary subset $Y \subs X$
        and its closure $\overline Y$ have the same neighborhoods
        $N_d(Y, R) = N_d(\overline Y, R) \subs X$ for any given
        $R > 0$ because $d$ is a metric.
    \item
        Provided $\varemptyset$ is not a member of a coarsely excisive cover,
        all intersections of finitely many sets in the cover must contain
        at least one point. Assume $\{X_j\}_{j \in J}$ are $\card J$ sets in
        the cover with empty intersection. Then $\{X_j\}_{j \in J}$
        does not satisfy
        \ref{eqn:coarselyExcisiveCover}. This can be seen in a similar way
        as Example
        \ref{exm:coarselyExcisiveEmpty}: In a metric space, any two
        nonempty sets
        have finite distance from each other, and thus $R$ can
        be chosen as the maximum of the pairwise distances among the $X_j$,
        producing a nonempty $\bigcap_{j \in J} N_d(X_j, R)$ that is
        not a subset of $N_d(\varemptyset, S) = \varemptyset$.
    \end{itemize}
\end{rem}

\begin{dfn}[Compatible coarsely excisive covers]
\label{dfn:compatibleCoarseCovers}
    Let $(X, d)$ and $(X', d')$ be coarse spaces and $f\colon X \to X'$
    a coarse map.
    Let $\{X_\beta\}_{\beta \in \alpha}$ be a finite or infinite coarsely
    excisive
    cover of $X$ and $\{X'_\beta\}_{\beta \in \alpha'}$ a coarsely
    excisive cover of
    $X'$ with $\alpha \subs \alpha'$ such that
    $f(X_\beta) \subseteq X'_\beta$ for every $\beta \in \alpha$.

    Then the two covers are called \emph{compatible with $f$}, or,
    when $f$ is clear from the context, simply \emph{compatible}.
\end{dfn}

    \subsection{Relative Roe algebras}
\label{sec:relativeRoeAlgebras}

Before we can use our spectral sequence for abstract C*-ideals
on coarsely excisive covers, we have
to shape our data accordingly -- we don't have sums and intersections of
C*-ideals, but rather unions and intersections of subspaces $X_j$.
We connect the coarse world to the world of abstract
C*-ideal intersections via relative Roe algebras and relative
$\dualalg$-algebras.

\begin{ntt}
\label{ntt:relativeRoeAlgebras}
    In Sections \ref{sec:relativeRoeAlgebras} through
    \ref{sec:sumsRelativeAlgebras}, let
    $(X, d)$ be a metric space.
    Fix a very ample representation $\varrho\colon \conto X \to BH$
    for a separable
    Hilbert space $H$ to define $\dualalg X$ and $\cualalg X$.

    Let $\{X_\beta\}_{\beta \in \alpha}$
    be a
    coarsely excisive cover of $(X, d)$ for an arbitrary
    index set $\alpha$.
    Let $J \subseteq \alpha$ be a nonempty finite subset of indices.
\end{ntt}

\begin{dfn}[Support near a subset]
    Let $Y \subseteq X$ be a subspace.
    An operator $T \in \cualalg X$ is \emph{supported near $Y$} if there
    exists a constant $R > 0$ (that may depend on $T$)
    such that all $f \in \conto X$ with $d(\supp f, Y) > R$ satisfy
    $\varrho(f) T = T \varrho(f) = 0 \in BH$.
\end{dfn}

\begin{dfn}[Relative Roe algebra]
\label{dfn:relativeRoeAlgebra}
    For $Y \subseteq X$ closed, the
    \emph{relative Roe algebra of $Y$ in $X$}, denoted
    $\cualalg(Y \subseteq X)$, is defined as the norm closure of
    all operators in $\cualalg X$ that are supported near $Y$.
\end{dfn}

\begin{dfn}[Relative $\dualalg$ algebra]
\label{dfn:relativeD}
    Let $Y \subseteq X$ be a subspace. The C*-algebra
    $\conto X$ contains $\conto(X - Y)$ as a C*-ideal: The inclusion
    morphism $\conto(X - Y) \to \conto X$ extends functions by zero on $Y$.

    Let $T \in \dualalg X$ be an operator such that,
    for all $f \in \conto(X - Y) \subs \conto X$,
    both $\varrho(f) T$ and $T \varrho(f)$ are compact operators.
    Then $T$ is called \emph{locally compact outside $Y$}.

    For $Y \subs X$ closed,
    the \emph{relative $\dualalg$ algebra} $\dualalg(Y \subseteq X)$
    is the norm closure of all operators in $\dualalg X$ that are supported
    near $Y$ and locally compact outside $Y$.
\end{dfn}

\begin{rem}
    The algebra $\cualalg(Y \subseteq X)$ is a C*-ideal in
    $\cualalg X$; the algebra $\dualalg(Y \subseteq X)$ is a C*-ideal
    in $\dualalg X$. As subalgebras of $\cualalg X$ and $\dualalg X$,
    all operators in these ideals
    are locally compact or pseudocompact, respectively.
    Each operator has finite propagation
    or is a norm limit of operators with finite propagation.

    For $Y = X$, we have $\cualalg(X \subseteq X) = \cualalg X$ and
    $\dualalg(X \subseteq X) = \dualalg X$.
    Operators in $\dualalg(Y \subseteq X)$
    act with three strengths on represented functions:
    pseudocompactly on $Y$, in a locally compact way near
    $Y$, and trivially far away from $Y$.
\end{rem}

Our guideline is \cite[Theorem 9.2]{roe-book}:
If the coarsely excisive cover $\{X_0, X_1\}$ has only two sets,
then there are isomorphisms
$\cualalg(X_0 \subs X) + \cualalg(X_1 \subs X) \cong \cualalg X$
and $\cualalg(X_0 \subs X) \cap \cualalg(X_1 \subs X)
\cong \cualalg(X_0 \cap X_1)$.

Besides for $\cualalg$, we prove a similar
result for $\dualalg$ and $\qualalg = \dualalg / \cualalg$; especially
$\dualalg$ requires extra work over the proofs for $\cualalg$.
Also, finite subsets $\{X_j\}_{j \in J}$ of arbitary coarsely excisive covers
$\{X_\beta\}_{\beta \in \alpha}$ require more care than an inductive
application of the result for two regions.
Even when $\bigcup_{j \in J} X_j$ happens to be a small subset of $X$,
our C*-algebras still arise from representations
of $\conto X$ in its entirety.
Our proofs must safely ignore regions of $X$ far away from
$\bigcup_{j \in J} X_j$.

\begin{thm}
\label{thm:isoOfRelativeRoeA}
    For all closed subsets $Y \subseteq X$ and K-theory
    degrees $s \in \zet$, there are natural isomorphisms of K-theory groups:
    \begin{align*}
        K_s\cualalg(Y \subseteq X) &\cong K_s\cualalg Y,\\
        K_s\dualalg(Y \subseteq X) &\cong K_s\dualalg Y.
    \end{align*}
\end{thm}

\begin{rem}
\label{rem:isoOfRelativeRoeB}
    The proof for $\cualalg$
    in \cite[Theorem 9.2]{roe-book} passes from $Y \subseteq X$
    to the coarsely equivalent $N_d(Y, n) \subseteq X$ for $n \in \nat$
    and takes the direct limit in K-theory along $n \to \infty$.
    The isomorphism is induced by the inclusion $Y \to X$.
    The construction for $\cualalg$ is natural with respect to coarse maps.
    The construction for $\dualalg$ is natural with respect to
    maps that are both coarse and continuous.

    A proof for $\dualalg$ is in \cite[Proposition 3.8]{siegel-mv}.
    This construction
    calls for a very ample representation $\varrho\colon \conto X \to BH$
    as we have required in Notation \ref{ntt:relativeRoeAlgebras},
    not merely for an ample representation.
\end{rem}

\begin{lem}
\label{lem:extendToBorel}
    The representation $\varrho\colon \conto X \to BH$ can be extended
    to all Borel functions on $X$.
\end{lem}

\begin{proof}
    This follows from \cite[Theorem II.1.1]{davidson1996} and
    \cite[Proposition II.1.2]{davidson1996}:
    As a nondegenerate representation of $\conto X$, the given $\varrho$
    is equivalent to a direct sum
    $\bigoplus_{\gamma \in \Gamma} \varrho_\gamma$
    of cyclic representations $\varrho_\gamma$,
    each unitarily equivalent
    to pointwise
    multiplication with a continuous function $f_\gamma$ that depends
    only on the cyclic representation $\varrho_\gamma$,
    on the Hilbert space $L^2(X)$ of $H$-valued functions
    using a regular Borel probability measure. Now extend by pointwise
    multiplication.

    The topological space in this construction is compact, but continuous
    functions on a compact space differ, as an algebra, from $\conto X$
    of a noncompact space $X$
    merely by the value at the extra point of the one-point
    compactification of $X$.
\end{proof}

\begin{lem}
\label{lem:partitionStaysInCualalg}
    Let $\psi\colon X \to [0,1]$ be a Borel function.
    Extend $\varrho$ to
    all Borel functions as in Lemma \ref{lem:extendToBorel} such that
    $\varrho(\psi)$ makes sense. Let $Y \subseteq X$ be a closed subset.
    Let $T$ be an operator in $\dualalg(Y \subseteq X)$. Then the following
    statments hold.
    \begin{itemize}
    \item We have $\varrho(\psi) T \in \dualalg(Y \subseteq X)$.
    \item Let $R_\mathrm{supp}$ be a distance constant for the support
        of $T$ near $Y$; i.e., for functions $f \in \conto X$
        with $d(\supp f, Y) > R_\mathrm{supp}$,
        the operators $\varrho(f) T$ and
        $T \varrho(f)$ vanish. Then $\varrho(\psi) T$ is supported
        near $Y$ with the same distance constant $R_\mathrm{supp}$.
    \item If $T$ has finite propagation with distance constant
        $R_\mathrm{prop}$, then $\varrho(\psi) T$ has finite
        propagation with distance constant $R_\mathrm{prop}$.
        (If $T$ does not admit such an $R_\mathrm{prop}$, then $T$
        is in the norm completion of operators that do.)
    \item If $T \in \cualalg(Y \subseteq X)$,
        then also $\varrho(\psi) T \in \cualalg(Y \subseteq X)$.
    \end{itemize}
\end{lem}

The main idea is that the support of a pointwise
product of functions $f \psi$ for $f \in \conto X$
must be a subset of $\supp f$ within $X$. Later in this section,
$\psi$ will be a function from a Borel partition of unity and
$\supp(f \psi)$ may be much smaller than $\supp f$.

\begin{proof}[Proof of Lemma \ref{lem:partitionStaysInCualalg}]
    We show the claim for all finite-propagation operators.
    The absolute algebras $\dualalg X$ and $\cualalg Y$ are norm completions
    of such operators; the general claim follows because
    the relative algebras $\dualalg(Y \subs X)$ and $\cualalg(Y \subs X)$
    carry the same norm as subalgebras of the absolute algebras.

    Let $R_\mathrm{prop} > 0$
    be a constant of finite propagation for $T$;
    i.e., for
    $f$, $g \in \conto X$ with
    $d(\supp f, \supp g) \geq R_\mathrm{prop}$,
    the product $\varrho(f)T\varrho(g) \in BH$ is zero.
    For such $f$, $g \in \conto X$
    with $d(\supp f, \supp g) \geq R_\mathrm{prop}$,
    we have
    $\supp(f \psi) \subseteq \supp f$, therefore
    \[
        d\big(\supp(f \psi), \supp g\big)
        \geq d(\supp f, \supp g) \geq R_\mathrm{prop}
    \]
    and $\varrho(f)\varrho(\psi)T\varrho(g) = \varrho(f\psi) T \varrho(g) = 0$.
    Thus $\varrho(\psi) T$ has finite propagation with the same
    constant $R_\mathrm{prop}$.

    Fix $R_\mathrm{supp} > 0$ such that
    all $f \in \conto X$ with $d(\supp f, Y) > R$ satisfy
    $\varrho(f) T = T \varrho(f) = 0 \in BH$; such an $R_\mathrm{supp}$
    exists because $T$ is supported near $Y$.
    Given $f \in \conto X$ with $d(\supp f, Y) > R$,
    we have
    $\supp(f \psi) \subseteq \supp f$, thus
    $\varrho(f) \varrho(\psi) T = \varrho(f\psi) T = 0$.
    Furthermore, $\varrho(\psi) T \varrho(f) = 0$ because $T \varrho(f) = 0$.
    Thus $\varrho(\psi) T$ is supported near $Y$ with the same
    distance constant $R_\textrm{supp}$.

    For pseudocompactness, given $f \in \conto X$,
    we must show that the following operator is compact:
    \begin{align*}
        \varrho(f)\varrho(\psi)T - \varrho(\psi)T\varrho(f)
        &= \varrho(f \psi)T - \varrho(\psi)T\varrho(f) \\
        &= \varrho(\psi f)T - \varrho(\psi)T\varrho(f) \\
        &= \varrho(\psi)
            \underbrace{\big( \varrho(f) T - T \varrho(f)\big)}_{
            \makebox[0pt]{\textrm{{\scriptsize compact since
                $T \in \dualalg X$}}}},
    \end{align*}
    which is compact as a product with a compact operator.
    Thus $\varrho(\psi) T$ is pseudolocal.

    For local compactness of $\varrho(\psi) T$
    outside $Y$, let $f$ be a function in $\conto(X - Y)$
    Because $T$ is already locally compact outside $Y$,
    $\varrho(f) T$ and $T \varrho(f)$ are compact operators. Then
    $\varrho(f) \varrho(\psi) T = \varrho(\psi) \varrho(f) T$
    and $\varrho(\psi) T \varrho(f)$ are also compact.
    Thus $\varrho(\psi) T$ is locally compact outside $Y$.

    Additionally, if $T \in \cualalg(Y \subseteq X)$, the same argument
    applied to arbitrary $f \in \conto X$ shows that
    $\varrho(f)\varrho(\psi)T$ and $\varrho(\psi)T\varrho(f)$ are
    compact for all $f \in \conto X$. Thus $\varrho(\psi) T$ is locally
    compact if $T$ is.
\end{proof}

\subsection{Intersections of relative algebras}
\label{sec:intersectionsRelativeAlgebras}

\begin{ntt}
    Throughout Section \ref{sec:intersectionsRelativeAlgebras},
    in addition to Notation \ref{ntt:relativeRoeAlgebras} that defines
    the finite nonempty subset
    $\{X_j\}_{j \in J}$ of the coarsely excisive cover
    $\{X_\beta\}_{\beta \in \alpha}$, write
    \[
        Z = \bigcap_{j \in J} X_j.
    \]
\end{ntt}

\begin{lem}
\label{lem:relCapA}
    Let $\fualalg$ denote either the functor $\cualalg$ or $\dualalg$.
    Then
    \[
        \fualalg(Z \subseteq X) \subs
        \bigcap_{j \in J} \fualalg(X_j \subseteq X).
    \]
\end{lem}

\begin{proof}
    For $T \in \fualalg(Z \subs X)$, there exists $R > 0$ such that
    $\varrho(f)T = T \varrho(f) = 0$ for all $f \in \conto X$ with
    $d(\supp f, Z) > R$ by definition of $T$ being supported near $Z$.

    Given $j \in J$, let
    $f_j \in \conto X$ be a function with
    $d(\supp f_j, X_j) > R$. Then
    $d(\supp f_j, Z) > R$ because $Z \subseteq X_j$, therefore
    $T$ is supported near $X_j$. This holds for all $j \in J$.
    For $\fualalg = \cualalg$, this finishes the proof: $T$ is in
    $\cualalg(X_j \subs X)$ for all $j \in J$.

    For $\fualalg = \dualalg$, we must show, in addition, that $T$ is
    locally compact outside $X_j$ for the given $j \in J$;
    this holds because $Z \subs X_j$ and $T \in \dualalg(Z \subs X)$
    is locally compact outside $Z$.
\end{proof}

\begin{ntt}
    For a subset $Y \subs X$, let
    \[
        \chi(Y)\colon X \to \{0,1\}
    \]
    denote the charateristic function of $Y$ on $X$; i.e., $\chi(Y)(x) = 1$
    if and only if $x \in Y$.
\end{ntt}

\begin{lem}
\label{lem:relCapB}
    For $\fualalg = \cualalg$ or $\fualalg = \dualalg$, we have
    \[
        \fualalg(Z \subseteq X) \supseteq
        \bigcap_{j \in J} \fualalg(X_j \subseteq X);
    \]
    i.e., the inclusion from Lemma \ref{lem:relCapA} is an equality of sets.
\end{lem}

\begin{proof}
    Fix $T \in \bigcap_{j \in J} \fualalg (X_j \subs X)$. We will show
    that $T \in \fualalg(Z \subs X)$.

\paragraph{Support of $T$ near $Z$.}
    Since $T$ is supported near $X_j$ for all $j \in J$,
    there are constants $R_j > 0$ such that
    whenever $f \in \conto X$ satisfies $d(\supp f, X_j) > R_j$ for at
    least one $j \in J$, then $\varrho(f)T = T \varrho(f) = 0$.
    (It is not necessary that $d(\supp f, X_j) > R_j$ holds for all
    $j \in J$. It is enough if this holds for one $j \in J$ because
    \enquote{support near a subset} states what
    happens on the complement of the subset, not on the subset itself.)

    The cover $\{X_\beta\}_{\beta \in \alpha}$ is coarsely excisive.
    In particular,
    for the constant $R = (\max_{j \in J} R_j) + 1$ and
    for the chosen finite index set $J$, there exists $S > 0$ with
    $
        \bigcap_{j \in J} N_d(X_j, R) \subseteq N_d(Z, S)
    $
    or, reformulating with complement sets,
    \[
        \big( X - N_d(Z, S) \big)
        \subseteq \bigcup_{j \in J} \big( X - N_d(X_j, R) \big).
    \]
    This constant $S$ depends on $T$ and the $X_j$, but not on any function
    in $\conto X$.
    To finish the proof,
    choose $f \in \conto X$ with $d(\supp f, Z) > S+1$. We must show
    that $\varrho(f)T = T \varrho(f) = 0$.

    The support of this $f$ lies within $X - N_d(Z, S)$, thus there exists
    a $j \in J$ with $\supp f \subseteq \big(X - N_d(X_j, R)\big)$ and
    therefore $d(\supp f, X_j) \geq R > R_j$. Because $T$ is supported
    near $X_j$, we conclude that
    $\varrho(f)T = T \varrho(f) = 0$ as desired. Thus $T$ is supported
    near $Z$.

\paragraph{Local compactness of $T$ outside $Z$.}
    If $\fualalg = \cualalg$, the proof is finished
    because $T$ is locally compact everywhere in $X$.

    If $\fualalg = \dualalg$,
    we must show that $T$ is locally compact outside $Z$.
    Fix a function $g \in \conto(X - Z) \subs \conto X$.
    We must show that $\varrho(g)T$ and $T\varrho(g)$ are compact.

    The support of $g$ may still overlap each $X_j$. To remedy this,
    decompose $X$ into $2^{\card J}$ regions: Given $L \subs J$,
    write
    \begin{align*}
        Y_L
        &= \Big(\bigcap_{j \in L} X_j \Big)
            - \bigcup_{j \notin L} X_j\\
        &= \set{ x \in X}{ x \in X_j \textrm{ if and only if } j \in L}.
    \end{align*}
    For all $L \subs J$, the set $Y_L$
    is a Borel set as a finite union, intersection, and set difference
    of closed sets $X_j$.
    For $L \neq L' \subs J$, the regions $Y_L$ and $Y_{L'}$ are
    disjoint. Furthermore,
    \[
        X = \bigcup_{L \subs J} Y_L, \qquad
        Z = Y_J = \bigcap_{j \in J} X_j.
    \]
    Decompose $g$ into $2^{\card J}$ Borel functions on $X$
    by multiplying with indicator functions:
    \[
        g = \sum_{L \subs J} \chi(Y_L)g.
    \]
    Extend $\varrho$ from $\conto X$ to all Borel functions on $X$.
    Lemma \ref{lem:partitionStaysInCualalg} shows that
    the operators
     $\varrho\big(\chi(Y_L)\big)T$ and
    $T\varrho\big(\chi(Y_L)\big)$ remain
    in $\bigcap_{j \in J} \dualalg(X_j \subs X)$ because the $Y_L$ are Borel
    sets.

    For each $L \subs J$, examine $\chi(Y_L)g$:
    For $L = J$, we have $\chi(Y_J)g = \chi(Z)g = 0$
    since $g \in \conto(X - Z)$.
    Both $\varrho\big(\chi(Y_J)g\big)T$
    and $T\varrho\big(\chi(Y_J)g\big)$ are
    the zero operator, thus compact.
    For $L \neq J$, fix an index $j \in J - L$.
    Then $\chi(Y_L)g$ vanishes on $X_j$ because $Y_L$ contains
    no points from $X_j$ by definition.
    We have $T \in \dualalg(X_j \subs X)$, therefore $T$ is locally compact
    outside $X_j$, making both $\varrho\big(\chi(Y_L)g\big)T$
    and $T\varrho\big(\chi(Y_L)g\big)$ compact.

    This shows that the decomposition
    $g = \sum_{L \subs J} \chi(Y_L)g$
    allows $\varrho(g)T$ and $T\varrho(g)$ to be written
    as finite sums of $2^{\card J}$ compact operators each. Such sums are compact. Thus
    $T$ is locally compact outside $Z$.
\end{proof}

\begin{pro}
\label{pro:connectTwoWorldsCapCD}
    Let $s \in \zet$ be a degree for K-theory.
    Then for the nonempty finite index set $J \subseteq \alpha$,
    there are natural isomorphisms
    \begin{align*}
        K_s \cualalg \Big(\bigcap_{j \in J} X_j\Big)
            &\cong K_s \Big( \bigcap_{j \in J} \cualalg(X_j \subs X) \Big), \\
        K_s \dualalg \Big(\bigcap_{j \in J} X_j\Big)
            &\cong K_s \Big( \bigcap_{j \in J} \dualalg(X_j \subs X) \Big).
    \end{align*}
\end{pro}

\begin{proof}
    Combine Lemma \ref{lem:relCapA} with Lemma \ref{lem:relCapB}
    for $Z = \bigcap_{j \in J} X_j$:
    \begin{align*}
        \cualalg(Z \subseteq X) &= \bigcap_{j \in J} \cualalg(X_j \subseteq X),
        \\
        \dualalg(Z \subseteq X) &= \bigcap_{j \in J} \dualalg(X_j \subseteq X).
    \end{align*}
    Theorem \ref{thm:isoOfRelativeRoeA} relates the K-theory of relative
    algebras with the K-theory of absolute algebras.
    This yields the claimed isomorphisms.
\end{proof}

\begin{pro}
\label{pro:connectTwoWorldsCapQ}
    Write $\qualalg X = \dualalg X / \cualalg X$ as in
    Notation \ref{ntt:qualalg}.
    Let $s \in \zet$ be a degree in K-theory. Then
    there is a natural isomorphism
    \[
        K_s \qualalg \Big(\bigcap_{j \in J} X_j\Big)
            \cong K_s \Big( \bigcap_{j \in J} \qualalg(X_j \subs X) \Big), \\
    \]
\end{pro}

\begin{proof}
    For closed $Y \subseteq X$,
    consider the following commutative
    diagram.
    Its rows are long exact sequences in K-theory.
    Its vertical isomorphisms are induced by the
    inclusion of metric spaces $Y \subseteq X$ as in
    Remark \ref{rem:isoOfRelativeRoeB}.
    The third vertical morphism is well-defined
    as follows because the rows are exact:
    For an operator
    $T \in \dualalg(Y \subs X)$ whose K-theory class $[T]$ maps
    to $[T'] \in \dualalg Y$ under the second vertical morphism,
    the third morphism maps $[T] + K_s \cualalg(Y \subs X)$ to
    $K_s\big([T'] + K_s \cualalg Y\big)$.
    \[
        \begin{tikzpicture}
            \mamax{0.3cm} {
                \dotsb
                & K_s \cualalg(Y \subseteq X)
                & K_s \dualalg(Y \subseteq X)
                & \frac{K_s \dualalg(Y \subseteq X)}
                       {K_s \cualalg(Y \subseteq X)}
                & K_{s-1} \cualalg(Y \subseteq X)
                & \dotsb
                \\
                \dotsb
                    & K_s \cualalg Y
                    & K_s \dualalg Y
                    & K_s \qualalg Y
                    & K_{s-1} \cualalg Y
                    & \dotsb.
                \\
            };
            \path[->, font=\scriptsize]
                (m-1-1) edge (m-1-2)
                (m-1-2) edge (m-1-3)
                (m-1-3) edge (m-1-4)
                (m-1-4) edge (m-1-5)
                (m-1-5) edge (m-1-6)
                (m-2-1) edge (m-2-2)
                (m-2-2) edge (m-2-3)
                (m-2-3) edge (m-2-4)
                (m-2-4) edge (m-2-5)
                (m-2-5) edge (m-2-6)
                (m-1-2) edge node[left] {$\cong$} (m-2-2)
                (m-1-3) edge node[left] {$\cong$} (m-2-3)
                (m-1-4) edge (m-2-4)
                (m-1-5) edge node[left] {$\cong$} (m-2-5)
            ;
        \end{tikzpicture}
    \]
    All squares commute because $\cualalg Y \to \dualalg Y$ is a C*-ideal
    inclusion and because the vertical isomorphisms arose from taking
    natural direct limits.

    By the five lemma, the vertical arrow to $K_s \qualalg Y$ must also
    be an isomorphism. It is natural by construction.
    With the isomorphisms from Propositions
    \ref{pro:connectTwoWorldsCapCD}, we have
    \begin{align*}
        K_s \qualalg \Big(\bigcap_{j \in J} X_j\Big)
            \cong
            \frac{K_s \dualalg \big(\bigcap_{j \in J} X_j\big)}
                 {K_s \cualalg \big(\bigcap_{j \in J} X_j\big)}
            &\cong
            \frac{K_s \big( \bigcap_{j \in J} \dualalg(X_j \subs X) \big)}
                 {K_s \big( \bigcap_{j \in J} \cualalg(X_j \subs X) \big)}
            \\
            &\cong
            K_s \Big( \bigcap_{j \in J} \qualalg(X_j \subs X) \Big).
        \qedhere
    \end{align*}
\end{proof}

\subsection{Sums of relative algebras}
\label{sec:sumsRelativeAlgebras}

\begin{ntt}
    Throughout Section \ref{sec:sumsRelativeAlgebras},
    in addition to Notation \ref{ntt:relativeRoeAlgebras}, write
    \[
        Z = \bigcup_{j \in J} X_j.
    \]
    For a subset $Y \subs X$, again, denote the indicator function by
    $\chi(Y)\colon X \to \{0,1\}$.
\end{ntt}

\begin{lem}
\label{lem:relSumA}
    Let $\fualalg$ be either the functor $\cualalg$ or $\dualalg$.
    Then
    \[
        \fualalg(Z \subseteq X) \subs
        \sum_{j \in J} \fualalg(X_j \subseteq X).
    \]
\end{lem}

\begin{proof}
    By definition, $\fualalg(Y \subs X)$ and $\fualalg X$ for closed
    subsets $Y \subs X$
    are norm completions of finite-propagation operator algebras. It suffices
    to check the inclusion for finite-propagation operators
    in $\fualalg(Z \subs X)$; passing to norm completions will then
    prove the claim.

    In this light, let $T \in \cualalg(Z \subseteq X)$ be an operator
    with finite propagation.
    We will construct operators $T_j$ for $j \in J$
    such that $T_j \in \fualalg(X_j \subs X)$ and $\sum_{j \in J} T_j = T$.

    Because $T$ has finite propagation,
    find $R_\textrm{prop} > 0$ such that $\varrho(f) T \varrho(g) = 0 \in BH$
    whenever $f$, $g \in \conto X$ satisfy
    $d(\supp f, \supp g) \geq R_\textrm{prop}$.
    Find $R_\textrm{supp} > 0$ such that
    $T$ is supported in the $R$-neighbor\-hood of $Z$.

    Cover $Z = \bigcup_{j \in J} X_j$ by the following sets $Y_j$ for
    $j \in J$ and their union $Y$:
    \begin{align*}
        Y_j &= N_d(X_j, R_\textrm{prop} + R_\textrm{supp} + 1)
            = \set{x \in X}
                {d(x, X_j) \leq R_\textrm{prop} + R_\textrm{supp} + 1},\\
        Y &= \bigcup_{j \in J} Y_j.
    \end{align*}
    Each $Y_j$ is closed in $X$. Certainly,
    $T$ is supported in $Y$, again a closed set.

    Choose a linear order $\prec$ on the finite set $J$.
    Define a partition $\{P_j\}_{j \in J}$ of $Y$ via
    \begin{align*}
        P_j = &\,\set{x \in Z}{x \in X_j
                \textrm{ and there are no } j' \prec j
                \textrm{ with } x \in X_{j'}}\\
            \cup
            &\,\set{x \in Y - Z}{x \in Y_j
                \textrm{ and there are no } j' \prec j
                \textrm{ with } x \in Y_{j'}};
    \end{align*}
    this is a partition of $Y$ because, given either $x \in Z$ or
    $x \in Y - Z$, exactly one $P_j$ is eligible
    to contain $x$. Furthermore, each $P_j$ is a Borel set because it may
    be written as a finite union, intersection, and difference of Borel sets
    $X_{j'}$ and $Y_{j'}$.

    Extend the representation $\varrho\colon \conto X \to BH$
    to the Borel functions of $X$
    according to Lemma \ref{lem:extendToBorel}.
    For all $j \in J$, define operators $T_j \in BH$ by
    \[
        \widetilde T = \frac{\varrho\big(\chi(X - Y)\big) T}{\card J},
        \qquad
        T_j = \varrho\big(\chi(P_j)\big) T + \widetilde T.
    \]
    By Lemma \ref{lem:partitionStaysInCualalg},
    $\widetilde T$ and all $T_j$ are in $\fualalg(Z \subs X)$. The $T_j$ sum to
    \begin{equation}
    \label{eqn:connectTwoWorldsSumA}
        \sum_{j \in J} T_j
            = \sum_{j \in J}
                \varrho\big(\chi(P_j)\big) T
                + \sum_{j \in J}
                \frac{\varrho\big(\chi(X - Y)\big) T}{\card J}
            = \varrho \big( \underbrace{\chi(Y) + \chi(X - Y)
                }_{\textrm{= } 1 \textrm{ on } X} \big) T
            = T.
    \end{equation}
    The summand $\widetilde T$ of $T_j$
    merely clarifies $\sum_{j \in J} T_j = T$; it has no deeper meaning.
    For all functions $g \in \conto X$, the products
    $\varrho(g) \widetilde T$ and
    $\widetilde T \varrho(g)$ vanish because $\chi(X - Y)$ is supported
    further than $R_\textrm{prop} + R_\textrm{supp}$ away from $Z$, whereas
    $\widetilde T \in \fualalg(Z \subs X)$ has the same propagation constant
    $R_\textrm{prop}$ and support distance constant $R_\textrm{supp}$ as $T$
    by Lemma \ref{lem:partitionStaysInCualalg}.

\paragraph{Support of $T_j$ near $X_j$.}
    For a given $j \in J$,
    let $f \in \conto X$ have support far enough away from $X_j$:
    \[
        d(\supp f, X_j) > R_\textrm{prop} + R_\textrm{supp} + 1.
    \]
    We will show $\varrho(f) T_j = T_j \varrho(f) = 0$.
    For $\varrho(f) T_j = 0$, we have
    \begin{equation}
    \label{eqn:connectTwoWorldsSumB}
        \varrho(f) T_j = \varrho \big( f \chi(P_j)\big) T
            + \underbrace{\varrho(f) \widetilde T}_{\textrm{= } 0} = 0
    \end{equation}
    because the pointwise product $f \chi(P_j)$ is zero in
    $\conto X$: The function $f$ is supported more than
    $R_\textrm{prop} + R_\textrm{supp} + 1$
    away from $X_j$, but $P_j \subseteq Y_j
        = N_d(X_j, R_\mathrm{prop} + R_\textrm{supp} + 1)$.
    For $T_j \varrho(f) = 0$, we similarly show that
    \begin{equation}
    \label{eqn:connectTwoWorldsSumC}
        T_j \varrho(f)
            = \varrho\big( \chi(P_j)\big) T \varrho(f)
            + \underbrace{\widetilde T \varrho(f)}_{\textrm{= } 0} = 0;
    \end{equation}
    to see this, observe that $d(X_j, X - P_j) \leq R_\textrm{supp} + 1$
    by construction of $P_j$ and $d(X_j, \supp f) >
    R_\textrm{prop} + R_\textrm{supp} + 1$ by the choice of $f$. The
    difference between these two values
    is more than $R_\textrm{prop}$, the propagation constant
    of $T$, thus $\varrho\big( \chi(P_j)\big) T \varrho(f) = 0$.

\paragraph{Local compactness of $T_j$ outside $X_j$.}
    Let $f$ be a function in $\conto(X - X_j) \subs \conto X$.
    We will show
    that $\varrho(f) T_j$ and $T_j \varrho(f)$ are compact.

    Recall that
    $\varrho(f) T_j = \varrho \big( f \chi(P_j)\big) T + 0$,
    thus it suffices to examine
    the pointwise product $f \chi(P_j)$: It may assume nonzero values only in
    $(Y_j - Z) \cup X_j$ by definition of $P_j$. Because $X_j \subs Z$,
    each point from $P_j$ falls either into $X_j$ or into $Y - Z$.
    We may decompose $f \chi(P_j)$ as
    \[
        f \chi(P_j) = f \chi(P_j \cap X_j) + f \chi(P_j - Z).
    \]
    The left summand is the zero function because $f$ vanishes on $X_j$.
    The right summand may be nonzero, but vanishes on $Z$.
    Outside $Z$, the original $T$ is locally compact, therefore
    $\varrho\big(f \chi(P_j - Z)\big) T$ is compact.
    Since $f \chi(P_j - Z) = f \chi(P_j)$ and furthermore
    $\varrho \big( f \chi(P_j)\big) T = \varrho(f) T_j$,
    the desired operator $\varrho(f) T_j$ is compact.

    The difference $\varrho(f) T_j - T_j \varrho(f)$ is compact because
    $T_j \in \dualalg(Z \subs X)$ is pseudocompact.
    With $\varrho(f) T_j$ already proven
    compact, $T_j \varrho(f)$ must be compact, too.

\paragraph{Summary.}
    The claim follows from \ref{eqn:connectTwoWorldsSumA},
    \ref{eqn:connectTwoWorldsSumB}, \ref{eqn:connectTwoWorldsSumC},
    and from the local compactness of $T_j$ outside $Z$:
    We have decomposed $T$ into a sum
    $\sum_{j \in J} T_j$ with each $T_j$ in $\fualalg(X_j \subs X)$.
\end{proof}

\begin{lem}
\label{lem:relSumB}
    For $\fualalg = \cualalg$ or $\fualalg = \dualalg$, we have
    \[
        \fualalg(Z \subseteq X) \supseteq
        \sum_{j \in J} \fualalg(X_j \subseteq X);
    \]
    i.e., the inclusion from Lemma \ref{lem:relSumA} is an equality of sets.
\end{lem}

\begin{proof}
    For each $j \in J$, let $T_j$ be an operator in
    $\fualalg(X_j \subseteq X)$ such that $T_j$ is supported in an
    $R_j$-neighborhood of $X_j$. Define $T = \sum_{j \in J} T_j$.
    This $T$ is supported
    in the $(\max_{j \in J} R_j)$-neighborhood of $Z = \bigcup_{j \in J} X_j$.

    For $\fualalg = \dualalg$, given $f \in \conto(X - Z) \subs \conto X$
    and $j \in J$, we know that
    $f$ is also in $\conto(X - X_j)$. The operators
    $\varrho(f)T_j$ and $T_j\varrho(f)$ are compact
    since $T_j \in \fualalg(X_j \subs X)$. The finite sums
    $\varrho(f)T$ and $T\varrho(f)$ of compact operators are again compact.
\end{proof}

\begin{pro}
\label{pro:connectTwoWorldsSumCD}
    For all $s \in \zet$, there are natural isomorphisms
    \begin{align*}
        K_s \cualalg \Big(\bigcup_{j \in J} X_j\Big)
            &\cong K_s \Big( \sum_{j \in J} \cualalg(X_j \subs X) \Big), \\
        K_s \dualalg \Big(\bigcup_{j \in J} X_j\Big)
            &\cong K_s \Big( \sum_{j \in J} \dualalg(X_j \subs X) \Big).
    \end{align*}
\end{pro}

\begin{proof}
    Combine Lemma \ref{lem:relSumA} with Lemma \ref{lem:relSumB}
    for $Z = \bigcup_{j \in J} X_j$:
    \begin{align*}
        \cualalg(Z \subseteq X) &= \sum_{j \in J} \cualalg(X_j \subseteq X),\\
        \dualalg(Z \subseteq X) &= \sum_{j \in J} \dualalg(X_j \subseteq X).
    \end{align*}
    Theorem \ref{thm:isoOfRelativeRoeA} yields the claimed isomorphisms.
\end{proof}

\begin{pro}
\label{pro:connectTwoWorldsSumQ}
    Let $s \in \zet$ be a degree in K-theory. There
    is a natural isomorphism
    \[
        K_s \qualalg \Big(\bigcup_{j \in J} X_j\Big)
            \cong K_s \Big( \sum_{j \in J} \qualalg(X_j \subs X) \Big).
    \]
\end{pro}

\begin{proof}
    In the proof of Proposition \ref{pro:connectTwoWorldsCapQ},
    we constructed a natural isomorphism for closed subsets $Y \subseteq X$,
    \[
        \frac{K_s \dualalg(Y \subseteq X)}
             {K_s \cualalg(Y \subseteq X)}
        \loto{\cong} \qualalg Y.
    \]
    Combine this isomorphism with the natural isomorphisms
    from Proposition \ref{pro:connectTwoWorldsSumCD}:
    \begin{align*}
        K_s \qualalg \Big(\bigcup_{j \in J} X_j\Big)
            \cong
            \frac{K_s \dualalg \big(\bigcup_{j \in J} X_j\big)}
                 {K_s \cualalg \big(\bigcup_{j \in J} X_j\big)}
            &\cong
            \frac{K_s \big( \sum_{j \in J} \dualalg(X_j \subs X) \big)}
                 {K_s \big( \sum_{j \in J} \cualalg(X_j \subs X) \big)}
            \\
            &\cong
            K_s \Big( \sum_{j \in J} \qualalg(X_j \subs X) \Big).
        \qedhere
    \end{align*}
\end{proof}

\subsection{Main theorem}

We may summarize Propositions
\ref{pro:connectTwoWorldsCapCD},
\ref{pro:connectTwoWorldsCapQ},
\ref{pro:connectTwoWorldsSumCD},
and
\ref{pro:connectTwoWorldsSumQ}
in a single theorem.

\begin{thm}
\label{thm:connectTwoWorlds}
    Let $(X, d)$ be a metric space. Let $\{X_\beta\}_{\beta \in \alpha}$
    be a finite or infinite coarsely excisive cover of $X$ and let
    $J \subseteq \alpha$ be a finite nonempty subset.

    \letFualalgBeEither
    Let $s$ be a degree in K-theory. Then
    \begin{align*}
        K_s \fualalg \Big(\bigcap_{j \in J} X_j\Big)
            &\cong K_s \Big( \bigcap_{j \in J} \fualalg(X_j \subs X) \Big), \\
        K_s \fualalg \Big(\bigcup_{j \in J} X_j\Big)
            &\cong K_s \Big( \sum_{j \in J} \fualalg(X_j \subs X) \Big).
    \end{align*}
    These isomorphisms are natural
    \withRespectToCoarseOrCoarseContinuousMorphisms{}
\end{thm}

When the cover $\{X_\beta\}_{\beta \in \alpha}$ has a finite index set
$\alpha$, the algebras
become suitable for our spectral sequence for finite ideal inclusions.

\begin{thm}
\label{thm:sseqForFiniteCoarselyExcisiveCovers}
    Let $(X, d)$ be a metric space with a finite coarsely excisive
    cover $\{X_\beta\}_{\beta \in \alpha}$.
    \thmSseqCoarselyExcisiveCoversForFor{
        for $0 \leq p < \card{\alpha}$,}{for $p < 0$ or $p \geq \card{\alpha}$,
    }
    where $J$ ranges over all nonempty subsets of $\alpha$.
    This spectral sequence converges strongly to $K_* \fualalg X$
    and is functorial
    \withRespectToCoarseOrCoarseContinuousMorphisms
\end{thm}

\begin{proof}
    Apply the spectral sequence from Theorem \ref{thm:sseqforsums}
    about finite sums of abstract C*-algebras to the algebras from
    Theorem \ref{thm:connectTwoWorlds} for coarse spaces.

    Both properties from Theorem \ref{thm:connectTwoWorlds} are required here:
    The intersection property
    guarantees that the $E^1$-term looks as stated.
    The sum property of the relative C*-algebras guarantees that
    the spectral sequence converges to the non-relative C*-algebra of the
    entire space.

    Functoriality of the spectral sequence follows from functoriality of
    the spectral sequence from Theorem \ref{thm:sseqforsums}
    for finite sums of abstract C*-ideals and from
    the naturality of the isomorphisms in Theorem \ref{thm:connectTwoWorlds}.
\end{proof}

\subsection{Application:
    \texorpdfstring{$K_* \cualalg \rel^n$}{K(C(Rn))}}
\label{sec:applicationKCRn}

Let $d_1$ and $d_\infty$ denote the usual 1-metric and $\sup$-metric on
$\rel^n$: For $x$, $y \in \rel^n$, we have
\[
    d_1(x, y) = \sum_{j < n} \betr{x_j - y_j},
    \qquad
    d_\infty(x, y) = \max_{j < n} \betr{x_j - y_j}.
\]

\begin{thm}
\label{thm:kHomologyOfEuclideanSpace}
    For the $n$-dimensional Euclidean space $\rel^n$, metrized either with the
    1-metric $d_1$ or the $\sup$-metric $d_\infty$,
    the Roe algebra has the following K-theory:
    \kTheoryOfRelN
\end{thm}

This is known, but we reprove this with our K-theory
spectral sequence for finite coarsely excisive covers.

\begin{rem}
    Towards the end of Section \ref{sec:applicationKCRn},
    some claims and proofs might look like straightforward geometry of
    $\rel^n$. In particular, since $d_1$ and $d_\infty$ are equivalent
    metrics on $\rel^n$, they must induce equivalent coarse structures;
    it would suffice to look at only one of them.

    Nonetheless, we will conduct these proofs in detail
    for both $d_1$ and $d_\infty$ because these results
    will serve as lemmas for Section \ref{subsection:zinfty}
    to compute the K-theory of a C*-ideal of
    $\cualalg \zet^\infty$ under different metrics.
\end{rem}

\begin{dfn}[Flasque]
\label{dfn:flasque}
    Let $(X, d)$ be a metric space. $X$ is \emph{flasque} if there is a
    coarse map $f\colon X \to X$ satisfying the following conditions:
    \begin{itemize}
        \item The map $f$ is coarsely equivalent to $\id(X)$.
        \item For all $K \subseteq X$, there exists $N \in \nat$ such that
            for all $n \geq N$, $f^n(X) \cap K = \varemptyset$.
        \item The powers of $f$ are uniformly coarse: For all $R > 0$,
            there exists
            $S > 0$ that, for all $n \in \nat$ at once and $x$, $y \in X$
            with $d(x, y) \leq R$, we have $d(f^n x, f^n y) \leq S$.
    \end{itemize}
\end{dfn}

\begin{lem}
    For a metric space $X$, the product $X \times \rel_{\geq 0}$ is
    flasque.
\end{lem}

\begin{proof}
    Consider the self-map $f$ on $X \times \rel_{\geq 0}$ with
    $f(x, t) = (x, t+1)$. Shifting points by a constant distance, $f$ is
    coarsely equivalent to the identity, yet the powers $f^n$ eventually
    shift points out of any given bounded set. As an isometry, $f$ is
    uniformly coarse: Choose $S = R$ for the third condition
    in Definition \ref{dfn:flasque}.
\end{proof}

\begin{pro}
\label{pro:flasqueiszero}
    Let $X$ be a flasque space. Then $K_*\cualalg X = 0$.
\end{pro}

Proposition \ref{pro:flasqueiszero} is proven in
\cite[Proposition 9.4]{roe-pmit}. To prove Theorem
\ref{thm:kHomologyOfEuclideanSpace} about $K_*\cualalg \rel^n$
with a coarsely excisive cover of $\rel^n$, it is helpful to
have many flasque intersections.

\begin{dfn}[Block decomposition of $\rel^n$]
\label{dfn:blockDecomposition}
    Cover $\rel^n$ with $n+1$
    overlapping subsets, or \emph{blocks},
    $X_0$, $X_1$, $\ldots$, $X_n$, where
    \[
        \syntaxHighlightWorkaround
    \]
\end{dfn}

\begin{exm}
    In the simplest case, $\rel^1$ is covered with two overlapping rays,
    one extending into either direction. $\rel^2$ is covered with three pieces,
    a left-hand half-space $X_0$, a bottom-right-hand quadrant $X_1$, and
    a top-right-hand quadrant $X_2$, as in Figure \ref{fig:reldflasque}:

    \BeginSimonFigure
    \begin{tikzpicture}
        \node at (5 cm, 4 cm) {$\phantom{unusedOnlyForSpacing}$};
        \draw[thick] (4 cm, 2 cm) -- (2 cm, 2 cm);
        \draw[thick] (2 cm, 0 cm) -- (2 cm, 4 cm);
        \draw[thick, dashed, gray!50] (0 cm, 0 cm) -- (0 cm, 4 cm)
            -- (4 cm, 4 cm) -- (4 cm, 0 cm) -- (0 cm, 0 cm);
        \node at (1 cm, 2 cm) {$X_0$};
        \node at (3 cm, 1 cm) {$X_1$};
        \node at (3 cm, 3 cm) {$X_2$};

        \newcommand{\drb}{\draw[thick]}
        \newcommand{\drg}{\draw[thick, dashed, gray!50]}
        \drg (6 cm, 0 cm) -- (9 cm, 0 cm) -- (10 cm, 1 cm)
            -- (10 cm, 4 cm) -- (7 cm, 4 cm) -- (6 cm, 3 cm) -- cycle;
        \drg (9 cm, 0 cm) -- (9 cm, 3 cm) -- (6 cm, 3 cm);
        \drg (9 cm, 3 cm) -- (10 cm, 4 cm);

        \drb (7.5 cm, 0 cm) -- (7.5 cm, 3 cm) -- (8.5 cm, 4 cm);
        \drb (8 cm, 3.5 cm) -- (9.5 cm, 3.5 cm) -- (9.5 cm, 0.5 cm);
        \drb (9.5 cm, 2 cm) -- (10 cm, 2.5 cm);

        \node at (6.75 cm, 1.5 cm) {$X_0$};
        \node at (8.25 cm, 1.5 cm) {$X_1$};
        \node at (9.75 cm, 1.5 cm) {$X_2$};
        \node at (9    cm, 3.75 cm) {$X_3$};
    \end{tikzpicture}
    \EndSimonFigure{fig:reldflasque}{
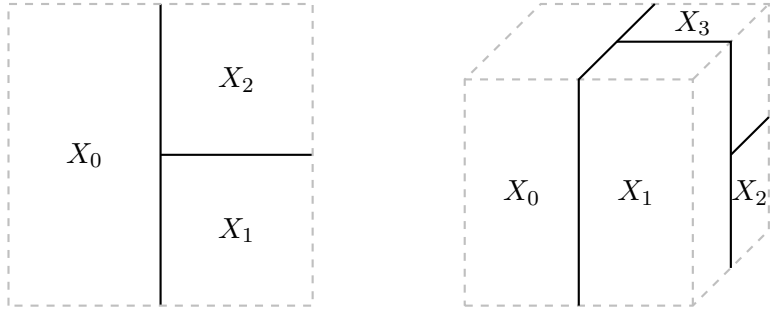
        Decomposition of $\rel^2$ into 3 blocks and of $\rel^3$
        into 4 blocks
    }
\end{exm}

\begin{rem}
    In the block decomposition $\rel^n = X_0 \cup X_1 \cup \dotsb \cup X_n$,
    each $X_j$ contains at least one flasque factor, therefore
    $K_*\cualalg X_j = 0$. Likewise, intersecting fewer than all $n+1$
    segments produces a flasque intersection with trivial K-theory of the
    Roe algebra.
    Only the $(n+1)$-fold intersection
    is not flasque; it is the compact one-point set. Its Roe algebra has
    K-theory $\zet$ in even degrees and zero in odd degrees.
\end{rem}

\begin{dfn}[Blocky subset of $\rel^n$]
\label{dfn:blocky}
    Let $X$ be a subset of $\rel^n$. We call $X$ \emph{blocky} if
    both of these conditions hold:
    \begin{itemize}
    \item
        The origin $0 \in \rel^n$ is part of $X$.
    \item
        For all $x \in X$, all $n$ coordinates $x_j$ of $x$, and all
        $\lambda \geq 0$, varying $x_j$ by $\lambda$ doesn't leave $X$; i.e.,
        this point is a part of $X$:
        \[
            (x_0, \ x_1, \ \ldots, \ x_{j-1}, \ \lambda x_j, \ x_{j+1}, \
                \ldots, \ x_{n-1})
        \]
    \end{itemize}
\end{dfn}

\begin{exm}
    Blocky subsets of $\rel^n$ are conical, but not all conical subsets
    are blocky.
    Consider the upper-right quadrant in $\rel^2$: This is blocky.
    It remains conical, but not blocky, after rotating around the origin
    by an eighth-turn.
\end{exm}

\begin{rem}
    For a nonempty collection of blocky subsets
    $\{X_\beta\}_{\beta \in \alpha}$ of $\rel^n$,
    the intersection
    $\bigcap_{\beta \in \alpha} X_\beta$ is again blocky.
\end{rem}

All sets of the block decomposition $\{X_j\}_{j \leq n}$ of $\rel^n$
and all their intersections
$\bigcap_{j \in J} X_j$ for nonempty $J \subs \setton$
satisfy this definition of \emph{blocky}. This is the
motivation behind the following Lemma \ref{lem:blockyIntersection}.

\begin{lem}
\label{lem:blockyIntersection}
    Let $X_\beta$, $X_\gamma \subseteq \rel^n$ be blocky sets. Choose $R > 0$.
    Then
    \begin{equation}
    \label{eqn:blockDecompInfty}
        N_\infty(X_\beta, R)
        \cap N_\infty(X_{\gamma}, R)
        = N_\infty(X_\beta \cap X_{\gamma}, R),
    \end{equation}
    where $N_\infty$ denotes the $R$-neighborhood
    under the $\sup$-metric $d_\infty$.
\end{lem}

\begin{proof}
    The direction \enquote{$\supseteq$} is immediate: If a point $y \in \rel^n$
    is at most $R$ away from $X_\beta \cap X_\gamma$, then it is at most
    $R$ away from both $X_\beta$ and $X_\gamma$ independently.

    To show \enquote{$\subseteq$}, fix
    $y \in N_\infty(X_\beta, R) \cap N_\infty(X_\gamma, R)$.
    We will show that this $y$ is already in
    $N_\infty(X_\beta \cap X_\gamma, R)$.
    For each coordinate
    $y_\delta$, by definition of $d_\infty$,
    we have $\inf_{x \in X} \betr{y_\delta - x_\delta} \leq R$
    for both $X = X_\beta$ and $X = X_\gamma$. In particular, for
    both $X = X_\beta$ and $X = X_\gamma$,
    \[
        \inf_{x \in X} \betr{y_\delta - x_\delta}
        = \begin{cases}
            0 & \textrm{if } y_\delta > 0 \textrm{ and }
                \{0\}^\delta \times \intCO{0}{+\infty}
                \times \{0\}^{n-\delta-1} \subseteq X, \\
            0 & \textrm{if } y_\delta < 0 \textrm{ and }
                \{0\}^\delta \times \intOC{-\infty}{0}
                \times \{0\}^{n-\delta-1} \subseteq X, \\
            \betr{y_\delta} \leq R & \textrm{otherwise}.
        \end{cases}
    \]
    Certainly, $X_\beta \cap X_\gamma$ is nonempty; at least the origin is part
    of this intersection.
    To finish the proof, assume that $d_\infty(y, X_\beta \cap X_\gamma) > R$.
    Then there exists a $\delta$-th coordinate such that
    \begin{equation}
    \label{eqn:coordinateDeltaTooBig}
        \inf \set{\betr{y_\delta - x_\delta}}{x \in X_\beta \cap X_\gamma} > R.
    \end{equation}
    Then $\betr{y_\delta} > R$ because
    $X_\beta \cap X_\gamma$ is blocky.
    This forbids the \enquote{otherwise}-case for both $X = X_\beta$ and
    $X = X_\gamma$. Both of the remaining two cases force the
    ray from the origin in the $\delta$-th dimension that contains
    $(0, \ldots, 0, y_\delta, 0, \ldots, 0)$ to be a subset of both
    $X_\beta$ and $X_\gamma$. But now, setting $y_\delta$ to zero will not
    affect the distance to the blocky set $X_\beta \cap X_\gamma$:
    \begin{align*}
        d_\infty \big((y_0, \ldots, y_{\delta-1}, y_\delta, y_{\delta+1},
            \ldots, y_{n-1}), X_\beta \cap X_\gamma\big)& \\
        = d_\infty \big((y_0, \ldots, y_{\delta-1}, 0, y_{\delta+1},
            \ldots, y_{n-1}), X_\beta \cap X_\gamma\big)&.
    \end{align*}
    After setting $y_\delta$ to $0$,
    if there are still coordinates remaining that satisfy
    \ref{eqn:coordinateDeltaTooBig}, repeat this argument and set those
    coordinates to $0$, too, without altering the distance
    to $X_\beta \cap X_\gamma$.
    Eventually, we find that no coordinates satisfy
    \ref{eqn:coordinateDeltaTooBig} anymore. Therefore the assumption
    $d_\infty(y, X_\beta \cap X_\gamma) > R$ is false and $y \in
    N_\infty(X_\beta \cap X_\gamma, R)$.
\end{proof}

\begin{cor}
\label{cor:blockyIsExcisive}
    Let $\{X_\beta\}_{\beta \in \alpha}$
    be a collection of blocky subsets of $\rel^n$ such that
    $\bigcup_{\beta \in \alpha} X_\beta = \rel^n$.
    Then $\{X_\beta\}_{\beta \in \alpha}$ is coarsely excisive
    with respect to the $\sup$-metric $d_\infty$.
    In particular,
    for a nonempty finite index set $J \subseteq \alpha$ and $R > 0$, we have
    \[
        \bigcap_{j \in J} N_\infty(X_j, R)
        = N_\infty\Big( \bigcap_{j \in J} X_j, R \Big).
    \]
\end{cor}

\begin{proof}
    For $\card J = 1$, the claim is trivial.
    For a $\card J$-fold intersection with $\card J > 1$, use induction
    along the cardinality of $J$: Choose $\beta \in J$ and set
    $J' = J - \{\beta\}$ for which the claim already holds. Then
    \begin{align*}
        \bigcap_{j \in J} N_\infty(X_j, R)
        &= N_\infty(X_\beta, R) \cap \bigcap_{j \in J'} N_\infty(X_j, R) \\
        &= N_\infty(X_\beta, R)
            \cap N_\infty \Big( \bigcap_{j \in J'} X_j, R \Big) \\
        &= N_\infty \Big( \bigcap_{j \in J} X_j, R \Big),
    \end{align*}
    applying Lemma \ref{lem:blockyIntersection} at the end
    because $\bigcap_{j \in {J'}} X_j$ is blocky.
\end{proof}

\begin{pro}
\label{pro:blockDecompSupIsExcisive}
    The block decomposition $\{X_j\}_{j \leq n}$ of $\rel^n$ from Definiton
    \ref{dfn:blockDecomposition}
    is coarsely excisive under the $\sup$-metric $d_\infty$:
    For a nonempty $J \subseteq \setton$ and $R > 0$, we have
    \[
        \bigcap_{j \in J} N_\infty(X_j, R)
        = N_\infty\Big( \bigcap_{j \in J} X_j, R \Big).
    \]
\end{pro}

\begin{proof}
    Each $X_j$ is blocky as a cartesian product of copies of
    $\intOC{-\infty}{0}$, $\intCO{0}{\infty}$, and $\rel$.
    The result follows from Corollary \ref{cor:blockyIsExcisive}.
\end{proof}

\begin{lem}[Relating 1-metric and $\sup$-metric]
\label{lem:block1sup}
    Let $X \subseteq \rel^n$ be arbitrary and choose $R > 0$. Denote by
    $N_1(X, R)$ the neighborhood of $X$ under the 1-metric $d_1$
    and by $N_\infty(X, R)$ its neighboorhood under the
    $\sup$-metric $d_\infty$. Then
    \begin{equation}
    \label{eqn:block1sup}
        N_1(X, R)
        \subseteq N_\infty(X, R)
        \subseteq N_1(X, nR).
    \end{equation}
\end{lem}

\begin{proof}
    For $x$, $y \in \rel^n$ arbitrary, we always have
    \[
        d_1(x,y) = \sum_{j < n} \betr{x_j - y_j}
        \geq \sup_{j < n} \betr{x_j - y_j}
        = d_\infty(x,y)
    \]
    and
    \[
        nd_\infty(x,y)
        = \sum_{j < n} \sup_{{j'} < n} \betr{x_{j'} - y_{j'}}
        \geq \sum_{j < n} \betr{x_j - y_j}
        = d_1(x,y).
    \]
    These estimates continue to hold
    when we take $\inf_{x \in X}$ for each term instead of a
    single point $x$, yielding the distance to $X$.
    Passing to neighborhoods, since larger metrics mean smaller
    neighborhoods, we get for
    the inclusion on the right-hand side of \ref{eqn:block1sup}:
    \begin{align*}
        N_\infty(X, R)
        &= \set{y \in \rel^n}{d_\infty(X,y) \leq R} \\
        &= \set{y \in \rel^n}{nd_\infty(X,y) \leq nR} \\
        &\subseteq \set{y \in \rel^n}{d_1(X,y) \leq nR} \\
        &= N_1(X, nR).
    \end{align*}
    A similar estimate holds for the inclusion on the left-hand side
    of \ref{eqn:block1sup}.
\end{proof}

\begin{pro}
\label{pro:blockDecomp1isExcisive}
    The block decomposition $\{X_j\}_{j \leq n}$ is coarsely
    excisive under the 1-metric $d_1$:
    Given a nonempty $J \subseteq \setton$ and $R > 0$, we can set
    $S = nR$ to ensure
    \[
        \bigcap_{j \in J} N_1(X_j, R)
        \subseteq
        N_1\Big( \bigcap_{j \in J} X_j, nR \Big).
    \]
\end{pro}

\begin{proof}
    Combining Proposition \ref{pro:blockDecompSupIsExcisive} with Lemma
    \ref{lem:block1sup}, we get
    \begin{align*}
        \bigcap_{j \in J} N_1(X_j, R)
        &\subseteq \bigcap_{j \in J} N_\infty(X_j, R) \\
        &= N_\infty \Big( \bigcap_{j \in J} X_j, R \Big) \\
        &\subseteq N_1 \Big( \bigcap_{j \in J} X_j, nR \Big).
        \qedhere
    \end{align*}
\end{proof}

This finishes the preparations for our proof of Theorem
\ref{thm:kHomologyOfEuclideanSpace}: We would like to show
\kTheoryOfRelN

\begin{proof}[Proof of Theorem \ref{thm:kHomologyOfEuclideanSpace}]
    The block decomposition $\{X_j\}_{j \leq n}$ of $\rel^n$ into $n+1$ pieces
    from Definition \ref{dfn:blockDecomposition} is coarsely excisive
    under either the $\sup$-metric by Proposition
    \ref{pro:blockDecompSupIsExcisive}
    or the 1-metric by Proposition \ref{pro:blockDecomp1isExcisive}.

    Intersecting fewer than $n+1$ blocks $X_j$ yields a flasque space $Y$
    with trivial $K_* \cualalg Y = 0$. Intersecting
    all $n+1$ points gives the compact one-point set $\{0\}$ with
    $K_s \cualalg \{0\} = \zet$ for $s$ even and $K_s \cualalg \{0\} = 0$
    for $s$ odd.

    These results fit into our spectral sequence from
    Theorem \ref{thm:sseqForFiniteCoarselyExcisiveCovers}
    for coarsely excisive covers, letting $J$ range over all nonempty subsets
    of $\setton$: The first page is
    \[
        E^1_{p,q} \cong
            \begin{dcases*}
                \bigoplus_{\card J = p+1}
                K_q \cualalg \Big( \bigcap_{j \in J} X_j \Big)
                &
                for $0 \leq p \leq n$,\\
                0 & for $p < 0$ or $p > n$,
            \end{dcases*}
    \]
    and the spectral sequence converges to $K_* \cualalg \rel^n$.
    The $E^1$-term has only one nonzero column $E^1_{n,*}$ from intersecting
    all $n+1$ pieces:
    \[
        \begin{tikzpicture}
            \masseq{
                \phantom{-}2 & 0 & \dotsb & 0 & \zet & 0 & 0 \\
                \phantom{-}1 & 0 & \dotsb & 0 & 0    & 0 & 0 \\
                \phantom{-}0 & 0 & \dotsb & 0 & \zet & 0 & 0 \\
                -1 & 0 & \dotsb & 0 & 0    & 0 & 0 \\
                & 0 & \dotsb & n-1 & \phantom{-.}n\phantom{-.}  & n+1 & n+2 \\
                };
            \drawSseqAxes{7}{5}
        \end{tikzpicture}
    \]
    This spectral sequence collapses on the first page. There is no extension
    problem to solve. We may read $K_s \cualalg \rel^n$ directly
    from the $s$-th diagonal $p + q = s$ of $E^1_{*,*}$:
    If $s - n$ is even, this K-theory is $\zet$; otherwise, it vanishes.
\end{proof}

\section{Infinite sums of ideals}\thispagestyle{plain}
\label{sec:infiniteIdeals}

For $A = I_0 + I_1 + \dotsb + I_n$, we have a spectral sequence. What happens
when $A$ is the norm closure of a sum over countably many C*-ideals
instead of over finitely many? Now $A$ is a direct limit C*-algebra,
\[
    A = \overline{\sum_{j \in \nat} I_j}
    = \overline{\bigcup_{n \in \nat} \Big( \sum_{j < n} I_j \Big)}
    = \overline{\bigcup_{
            \substack{J \subs \nat \\ \card J \in \nat}
        } \Big( \sum_{j \in J} I_j \Big)}.
\]
Can we replace $\Delta^n$ by an infinite simplex in the underlying
construction?
Or can we take the existing spectral sequence for a subalgebra
$A(n) = \sum_{j<n} I_j$, then use the inclusion of algebras
$A(n) \to A(n+1)$ to link the spectral sequences together? Does the direct
limit of these spectral sequences converge to $K_*A$, or can we at least
salvage some information about $K_*A$ from this construction?

After developing a spectral sequence for
$A = \overline{\sum_{j \in \nat} I_j}$,
Section \ref{sec:uncountableSums} will generalize the result to direct limits
$A = \overline{\sum_{\beta \in \alpha} I_\beta}$ of sums of uncountably many C*-ideals $I_\beta$ for arbitrary index sets $\alpha$.

\subsection{Na\"{i}ve approaches}

Let $A = \overline{\sum_{j \in \nat} I_j}$
be a C*-algebra with the $I_j \subseteq A$
closed two-sided ideals. For $n \in \nat$ and $j \leq n$, we have defined
the \emph{cake piece}
$\Delta^n_j$ in Definition \ref{dfn:cakePieces} as a subset
of the standard simplex $\Delta^n$.
We form C*-algebras $B = B(I_0, \ldots, I_n)$
based on the first $n + 1$ ideals like in Definition \ref{dfn:cakealgebra}:
\equationDefineB
This leads to algebras $B(I_0, \ldots, I_n)_J$ for
$J \subseteq \setton$. But everything depends on our initial
choice of the simplex $\Delta^n$.

One immediate idea is to generalize the underlying function algebras $B_J$:

\begin{dfn}
    We define the \emph{infinite simplex}
    \[
        \Delta^\nat = \set{(x_j)_j \in [0,1]^\nat}
            { \sum_{j=0}^\infty x_j \leq 1 }.
    \]
    This becomes a topological space with the usual product topology.
\end{dfn}

Let $A = \overline{\sum_{j \in \nat} I_j}$ be a C*-algebra.
Even if we succeed in defining a function space
$\Delta^\nat_J$ for $J \subseteq \nat$
and function algebras $B_J \subseteq B$ of certain functions
$\Delta^\nat \to A$,
it will be hard interpret the function algebras appropriately.
In the finite case with $n+1$ ideals, the algebra
\[
    \set{f\colon \Delta^n \to A}{f \restr \partial \Delta^n = 0}
\]
is isomorphic to the $n$-fold suspension of $A$. The suspension isomorphism
allowed us to relate the K-theory of ideals to the K-theory of $A$.
When we replace $\Delta^n$ with $\Delta^\nat$, we lose the suspension
isomorphism and cannot give a convergence theorem for a spectral sequence.

For another approach,
recall the spectral sequence for ideal inclusions,
constructed both in Section \ref{sec:schochetSseq}
and by C. Schochet in \cite{schochet-sseq}:

\schochetSseq*

In our setting, we do not have $A = \overline{\bigcup_{p \in \nat} I_p}$
but merely $A = \overline{\sum_{j=0}^\infty I_j}$.
A plausible adaption to our setting might be:
\begin{itemize}
    \item Compute $K_*\big(\sum_{j=0}^n I_j\big)$
        from $I_0$, $I_1$, $I_2$, $\ldots$, $I_n$ using our spectral sequence
        that takes finite sums of ideals.
    \item For each $n$, compute
        $K_*\big(\sum_{j=0}^n I_j / \sum_{j=0}^{n-1} I_j\big)$
        with the six-term exact sequence.
    \item Feed these results at once into the spectral sequence
        from Theorem \ref{thm:schochetSseq}
        to compute $K_* A$.
\end{itemize}
The downside is the multilayered computation: We build many spectral
sequences, solve an extension problem for every single one, and then fit the
results into yet another spectral sequence. This is unlikely to work except
in trivial cases where $K_*\big(\overline{\sum_{j=0}^\infty I_j}\big)$
would have been straightforward to compute by other means, or when
the K-theory would
already equal $K_*\big(\sum_{j=0}^n I_j\big)$ for an $n \in \nat$.
Instead, we would
like a more robust approach featuring a single spectral sequence
$\{E^r_{p,q}, d^r\}_{r,p,q}$
with terms $E^1_{p,q}$ that are easier to describe and compute.

\subsection{Linking two chains of ideals}

\begin{ntt}
    Fix a C*-algebra $A = \overline{\sum_{j=0}^\infty I_j}$
    where the $I_j$ are closed two-sided
    ideals of $A$. We denote by
    $E(n)^r_{p,q}$ the $(p,q)$-th module in the
    $r$-th page of the spectral sequence for the sum of the first
    $n$ ideals $I_0 + I_1 + \dotsb + I_{n-1}$:
    Each of these spectral sequences
    $\{E(n)^r_{p,q}, d(n)^r\}_{r,p,q}$ is constructed according to
    our main Theorem \ref{thm:sseqforsums} about finite sums of
    C*-ideals.
\end{ntt}

We will construct a morphism of spectral sequences from $E(n)$ to $E(n+1)$.
This requires several technical propositions. Each spectral sequence
arises from a chain of ideals
$Q_0 \subseteq Q_1 \subseteq Q_2 \subseteq \dotsb$ that strongly depends
on a fixed number of ideals $I_0$, $I_1$, $\ldots$, $I_j$, $\ldots$
chosen in the beginning of the construction -- see Section
\ref{sec:cakePieceAlgebras}.
To relate $E(n)$ with $E(n+1)$, we first construct
morphisms between the two chains of ideals that lead to
these two spectral sequences.

Along this way, we diligently track naturality
with respect to \starhoms{} that preserve ideal decompositions
(Definition \ref{dfn:preserveIdealDecomps1} and Remark
\ref{rem:preserveIdealDecomps2}).

\begin{lem}
\label{lem:addzeroideal}
    Let $I_0, \ldots, I_n$ be C*-ideals. Fix $k \in \setton$. Define
    $I'_k = 0$, moreover $I'_j = I_j$ for $j \neq k$.
    Fix $J \subseteq \setton - \{k\}$. Then
    \[
        B(I_0, \ldots, I_n)_J =
        B(I'_0, \ldots, I'_n)_J =
        B(I'_0, \ldots, I'_n)_{J \cup \{k\}}.
    \]
\end{lem}

\begin{proof}
    The first equality holds because functions in
    $B(I_0, \ldots, I_n)_J$ and $B(I'_0, \ldots, I'_n)_J$ are never allowed
    to take nonzero values in $I_k$ or $I'_k$, the only ideals that
    differ in the construction.

    The second equality holds because $B_J$ and
    $B_{J \cup \{k\}}$ differ at most by conditions enforced on the subspace
    $\Delta^n_k \subseteq \Delta^n$, but on $\Delta^n_k$,
    functions map to $I'_k = 0$ anyway.
\end{proof}

\begin{pro}
\label{pro:addzeroideal}
    Let $I_0$, $\ldots$, $I_n$ be C*-ideals that sum to $A$.
    Construct the cake algebra $B(I_0, \ldots, I_n)$
    as in Definition \ref{dfn:cakealgebra}.
    Now include the zero algebra $0 = I_{n+1}$
    as an extra ideal, leading to a
    different cake algebra $B(I_0, \ldots, I_n, 0)$:
    \begin{align*}
        B(I_0, \ldots, I_n)
        &\subseteq \cont(\Delta^n, A),\\
        B(I_0, \ldots, I_n, 0)
        &\subseteq \cont(\Delta^{n+1}, A).
    \end{align*}
    Let $J \subseteq \setton$ be an index set.
    There is a suspension isomorphism
    \[
        S\big(B(I_0, \ldots, I_n)_J\big) \cong
        B(I_0, \ldots, I_n, 0)_{J \cup \{n+1\}}.
    \]
\end{pro}

Figure \ref{fig:addzeroideal}
gives the geometric idea. Functions on the line vanish on
its two end points. Functions on the grey four-sided shape vanish on the
continuously drawn boundary lines but not on the dashed line.

We claim that the suspension algebra of the functions on the line is isomorphic
to the function algebra on the grey shape.

\BeginSimonFigure
\begin{tikzpicture}
	\coordinate (A1) at (0 cm, 0 cm);
	\coordinate (A2) at (4 cm, 0 cm);
    \coordinate (Ce) at (barycentric cs:A1=1,A2=1);

	\draw[thick] (A1) -- (A2);
    \eiddot{A1}{below}{(1, 0)}
    \eiddot{A2}{below}{(0, 1)}
    \eiddot{Ce}{below}{(\frac12, \frac12)}
    \node[above] at (barycentric cs:A1=1,Ce=1) {$I_1$};
    \node[above] at (barycentric cs:A2=1,Ce=1) {$I_0$};

	\coordinate (B1) at (6 cm, 0 cm);
	\coordinate (B2) at (10 cm, 0 cm);
	\coordinate (B3) at (8 cm, 3.464 cm);
    \coordinate (Cb) at (barycentric cs:B1=1,B2=1,B3=1);

    \fill[gray!50] (B1) -- (Cb) -- (B2) -- (B3) -- cycle;
	\draw[thick] (B1) -- (B2) -- (B3) -- (B1) -- (Cb) -- (B2);
    \draw[thick, dashed] (Cb) -- (B3);
    \eiddot{B1}{below}{(1, 0, 0)}
    \eiddot{B2}{below}{(0, 1, 0)}
    \eiddot{B3}{above}{(0, 0, 1)}
    \eiddot{Cb}{below}{}
    \node at (barycentric cs:B1=1,Cb=1,B3=1) {$I_1$};
    \node at (barycentric cs:B2=1,Cb=1,B3=1) {$I_0$};
    \node at (barycentric cs:B1=1,Cb=1,B2=1) {$I_2 = 0$};
\end{tikzpicture}
\EndSimonFigure{fig:addzeroideal}{
    Geometric idea of Proposition \ref{pro:addzeroideal}
}

\begin{proof}[Proof of Proposition \ref{pro:addzeroideal}]
    By Lemma \ref{lem:addzeroideal},
    we can reduce the case of arbitrary $J$ to $J = \setton$ by replacing
    $I_j$ with $0$ for all $j \notin J$, completing the construction in this
    proof, then re-inserting the original ideals $I_j$.

    Define
    \[
        X = \bigcup_{j = 0}^n \Delta^{n+1}_j, \qquad
        Y = X \cap \partial \Delta^{n+1}, \qquad
        Z = X \cap \Delta^{n+1}_{n+1}.
    \]
    The set $X$
    corresponds to the grey area in the example figure above, $Y$ to the
    grey area's ceiling, $Z$ to its floor.
    Functions in $B(I_0, \ldots, I_n, 0)$ vanish
    on $Y$ and $Z$, but they may assume nontrivial
    values in $I_0$, $\ldots$, $I_n$ on the interior $X - Y - Z$.

    The points $x$ in $X$ have $n+2$ barycentric coordinates
    $(x_0, \ldots x_{n+1})$. For functions in $B(I_0, \ldots, I_n, 0)$,
    the relative values of the first $n+1$ of these coordinates select ideals
    among $I_0$, $\ldots$, $I_n$ as the function's range.
    The final coordinate $x_{n+2}$
    does not affect that choice, but $x_{n+2}$ is not well-suited to see that
    $B(I_0, \ldots, I_n, 0)$ is isomorphic to a suspension algebra.
    Instead, we show this with a
    reparametrization $\varphi$ of $X$. Set
    \newcommand{\ymin}{\mathnormal{y_\mathrm{min}}}
    \newcommand{\interv}{[\frac1{n+2}, 1]}
    \[
        \ymin = \min \set{y_j}
        {y_j \textrm{ is a barycentric coordinate of }
        y = (y_0, \ldots, y_n)},
    \]
    \[
        \varphi\colon
        \Delta^n \times \left[\frac{1}{n+2}, 1\right]
        \to \Delta^{n+1},
    \]
    \[
        \varphi(y_0, \ldots, y_n, t) =
        (y_0 - {t \ymin}, \ldots,
         y_n - {t \ymin}, (n+1) t \ymin).
    \]
    This function $\varphi$ maps continuously to $X$:
    $\varphi$ distributes $(n+1)$ portions
    of $t \ymin$ from the first $(n+1)$ coordinates to the last
    coordinate. Thus
    the sum of all coordinates remains 1. Furthermore,
    $\varphi$ cannot map to the
    interior of $\Delta^{n+1}_{n+1}$ because the last coordinate cannot be
    the uniquely smallest: With $t \geq \frac1{n+2}$ by definition
    of $\varphi$, we have
    \[
        (n+1) t \ymin \geq \ymin - t \ymin.
    \]
    Equality holds exactly for $t = \frac1{n+2}$.

    The map $\varphi$ is surjective onto $X$: Construct a preimage of
    $(y_0, \ldots, y_n, y_{n+1})$ by distributing $\frac{y_{n+1}}{n+1}$
    onto each of the first $n+1$ coordinates, then choose $t$.
    Furthermore, $\varphi$ is injective on the
    interior of $\Delta^n \times \interv$.

    On the simplex boundaries, we need not check injectivity.
    Because $B(I_0, \ldots, I_n)$ or
    $B(I_0, \ldots, I_n, 0)$ must vanish on the simplex boundaries,
    it suffices to verify that all points of
    $\del \Delta^n \times \interv$ map to $\del \Delta^{n+1}$. This holds
    because $\ymin = 0$, thus $\ymin - t\ymin = 0$ regardless of the value
    $t \in \interv$.

    On the interior of $X$, that is, on the interior of
    $\varphi\big((\Delta^n)^\circ \times
    \mathnormal{]}\frac{1}{n+2}, 1\mathnormal{[}
    \big)$,
    functions in $B(I_0, \ldots, I_n, 0)$ are subject
    to the restrictions from the relations of the first $n+1$ barycentric
    coordinates, but not to any restriction from the final zero ideal or
    from the coordinate $t \in \interv$.
    On $\varphi\big((\Delta^n)^\circ \times \big\{\frac{1}{n+2}, 1\big\}\big)$,
    the functions must be
    zero. Thus via $\varphi$, we see that $B(I_0, \ldots, I_n, 0)$
    is isomorphic to the suspension of $B(I_0, \ldots, I_n)$.

    It suffices to parametrize $X$ instead of the entire simplex
    $\Delta^{n+1}$ because every considered C*-function living on $X$ has only
    one possible extension -- by zero --
    to a function in $B(I_0, \ldots, I_n, 0)$.
\end{proof}

\begin{rem}
    The isomorphism from Proposition \ref{pro:addzeroideal}
    is natural with respect to
    finite ideal decompositions: A collection of \starhoms{}
    $i_j\colon I_j \to I'_j$ for $j \in \setton$
    on the input C*-ideals, together with the constructed isomorphisms on
    either side, leads to a commutative diagram
    \[
        \begin{tikzpicture}
        \mama {
            S\big(B(I_0, \ldots, I_n)_J\big)
            & B(I_0, \ldots, I_n, 0)_{J \cup \{n+1\}}
            \\
            S\big(B(I'_0, \ldots, I'_n)_J\big)
            & B(I'_0, \ldots, I'_n, 0)_{J \cup \{n+1\}}.
            \\
        };
        \path[->, font=\scriptsize]
            (m-1-1) edge node[above] {$\cong$} (m-1-2)
            (m-2-1) edge node[below] {$\cong$} (m-2-2)
            (m-1-1) edge node[left] {$S\big(B(i_0, \ldots, i_n)\big)$} (m-2-1)
            (m-1-2) edge node[right] {$B(i_0, \ldots, i_n, 0)$} (m-2-2)
        ;
        \end{tikzpicture}
    \]
\end{rem}

\begin{pro}
\label{pro:conncetQpWithQpPlusOne}
    Let $A = I_0 + I_1 + \dotsb + I_n$ be a sum of C*-ideals.
    The function algebras $B_J = B(I_0, \ldots, I_n)_J$
    for $J \subseteq \setton$
    give rise to cake sums $Q_p = \sum_{\card J \leq p+1} B_J$ by Definition
    \ref{dfn:cakeSums}.

    Let
    $I_{n+1} = 0$ be an extra zero ideal. The function algebras $\widetilde B_J
    = B(I_0, \ldots, I_n, 0)_J$ for this larger set of ideals
    and $J \subseteq \{0, 1, \ldots, n, n+1\}$
    give rise to $\widetilde Q_p = \sum_{\card J \leq p+1} \widetilde B_J$.

    Then $\widetilde Q_p \cong SQ_p$ for $p \in \zet$
    and $\widetilde Q_{n+1} = \widetilde Q_{n}$.
    In other words, the chains of ideal inclusions $Q_p \to Q_{p+1}$
    and $\widetilde Q_p \to \widetilde Q_{p+1}$
    fit into this commutative diagram:
    \[
        \begin{tikzpicture}
        \mama {
            SQ_0 & SQ_1 & \cdots & SQ_n &
            \\
            \widetilde Q_0 & \widetilde Q_1 & \cdots
            & \widetilde Q_n & \widetilde Q_{n+1}.
            \\
        };
        \path[->, font=\scriptsize]
            (m-1-1) edge (m-1-2)
            (m-1-2) edge (m-1-3)
            (m-1-3) edge (m-1-4)
            (m-2-1) edge (m-2-2)
            (m-2-2) edge (m-2-3)
            (m-2-3) edge (m-2-4)
            (m-1-1) edge node[left]{$\cong$} (m-2-1)
            (m-1-2) edge node[left]{$\cong$} (m-2-2)
            (m-1-4) edge node[left]{$\cong$} (m-2-4)
            (m-1-4) edge node[above]{$\cong$} (m-2-5)
        ;
        \path[-]
            (m-2-4) \maeq (m-2-5)
        ;
        \end{tikzpicture}
    \]
\end{pro}

\begin{proof}
    Our index sets will sometimes contain indices in $\setton$, sometimes
    in $\{0, 1, \ldots, n, n+1\}$.
    For clarity, given $p \in \zet$, we define these collections of index sets:
    \begin{align*}
        \mathscr A(p) &= \set{J \subseteq \setton}{\card J \leq p}, \\
        \mathscr B(p) &= \set{J \subseteq \{0, 1, \ldots, n+1\}}
                            {\card J \leq p}.
    \end{align*}
    For $p < 0$, $\mathscr A(p)$ and $\mathscr B(p)$ are empty and
    $Q_p$ is the zero algebra.

    The following argument will work without
    modification for all $p \in \zet$.
    For a given $p \in \zet$,
    we may write the collection
    $\mathscr B(p+1)$ as the disjoint union of
    $\mathscr A(p+1)$ and $\set{J \cup \{n+1\}}{J \in \mathscr A(p)}$.
    Applying this to the definition of $\widetilde Q_p$, we obtain
    \newcommand{\clapsum}[2]{\sum_{{\mathscr{#1}({#2})}}}
    \[
        \widetilde Q_p = \sum_{\mathscr B(p+1)} \widetilde B_J
        = \clapsum A {p+1} \widetilde B_J
        + \clapsum A{p} \widetilde B_{J \cup \{n+1\}}.
    \]
    By Lemma \ref{lem:addzeroideal}, the index $n+1$ can be dropped from
    $B_{J \cup \{n+1\}}$ with no change because $I_{n+1} = 0$:
    \[
        \widetilde Q_p
        = \clapsum A {p+1} \widetilde B_J
        + \clapsum A{p} \widetilde B_J.
    \]
    By definition of $\mathscr A(p)$ as a collection of all index sets
    up to cardinality $p$, we have
    $\mathscr A(p) \subseteq \mathscr A(p+1)$. The sum simplifies to
    \[
        \widetilde Q_p = \clapsum A {p+1} \widetilde B_J.
    \]
    None of the sets in $\mathscr A(p+1)$ contain the index $n+1$.
    This allows us to rewrite $\widetilde B_J$ as $SB_J$
    via a natural isomorphism according to Proposition \ref{pro:addzeroideal}.
    Taking sums commutes with suspensions:
    \[
        \widetilde Q_p
        = \clapsum A {p+1} \widetilde B_J
        \cong \clapsum A {p+1} SB_J
        = SQ_p.
    \]
    This shows the main result.
    The extra result $\widetilde Q_{n+1} = \widetilde Q_n$ follows from
    $\widetilde Q_n \cong SQ_n = SQ_{n+1} \cong \widetilde Q_{n+1}$.
\end{proof}

\begin{rem}
\label{rem:qTildeIsNatural}
    The constructed isomorphism is natural: It is a composition of equalities
    and the natural isomorphism from Proposition \ref{pro:addzeroideal}.
\end{rem}

\subsection{Compatibility of suspensions}
\label{sec:compatibilitySuspensions}

\begin{lem}
\label{lem:suspendedQuotient}
    Let $I$ be a C*-ideal in $A$. Then
    $S(A/I) \cong SA/SI$.
    Explicitly, there is an isomorphism $\Phi\colon SA/SI \to S(A/I)$ with
    \[
        \Phi(f + SI)(t) = f(t) + I \qquad
        \textrm{for } f\colon [0,1] \to A
        \textrm{ with } f(0) = f(1) = 0.
    \]
\end{lem}

Later, we show naturality of $\Phi$ in a separate lemma; first,
we construct this isomorphism.

\begin{proof}[Proof of Lemma \ref{lem:suspendedQuotient}]
    \newcommand{\cc}{\otimes \conto\mathnormal{]}0,1\mathnormal{[}}
    The function $\Phi$
    is well-defined, additive, and multiplicative because $I$ is an ideal;
    it preserves the involution because $I$ is a C*-ideal.
    If $I$ contains the range of $\Phi(f)$, then $f$ was already in $SI$,
    therefore $\Phi$ is injective.

    For surjectivity, we will use $A'\cc \cong SA'$ for arbitrary C*-algebras
    $A'$ and the nuclearity of
    $\conto\mathnormal{]}0,1[$ as proven, e.g., in \cite{weggeolsen-93}.
    As a result, the top two rows of this commutative diagram become exact:
    \[
        \begin{tikzpicture}
        \mama {
            0 & I\cc & A\cc & A/I \cc & 0
            \\
            0 & SI & SA & S(A/I) & 0
            \\
            0 & SI & SA & SA/SI & 0.
            \\
        };
        \path[->, font=\scriptsize]
            (m-1-1) edge (m-1-2)
            (m-1-2) edge (m-1-3)
            (m-1-3) edge (m-1-4)
            (m-1-4) edge (m-1-5)
            (m-2-1) edge (m-2-2)
            (m-2-2) edge (m-2-3)
            (m-2-3) edge (m-2-4)
            (m-2-4) edge (m-2-5)
            (m-3-1) edge (m-3-2)
            (m-3-2) edge (m-3-3)
            (m-3-3) edge (m-3-4)
            (m-3-4) edge (m-3-5)
            (m-1-2) edge node[auto]{$\cong$} (m-2-2)
            (m-1-3) edge node[auto]{$\cong$} (m-2-3)
            (m-1-4) edge node[auto]{$\cong$} (m-2-4)
            (m-3-4) edge node[right]{$\Phi$} (m-2-4)
        ;
        \path[-]
            (m-2-2) \maeq (m-3-2)
            (m-2-3) \maeq (m-3-3)
        ;
        \end{tikzpicture}
    \]
    The square in the bottom right commutes by our explicit
    construction of $\Phi$. Since the bottom row is the standard quotient exact
    sequence for the inclusion $SI \to SA$,
    the \starhom{} $\Phi$ is an isomorphism by the five lemma.
\end{proof}

\begin{lem}
    The isomorphism $\Phi\colon SA/SI \to S(A/I)$ constructed in Lemma
    \ref{lem:suspendedQuotient} is natural with respect to \starhoms{}
    $h\colon A \to A'$ that map $I$ into $I'$, where $I' \subseteq A'$ is
    a given C*-ideal.
\end{lem}

\begin{proof}
    Construct $\Phi'\colon SA'/SI' \to S(A'/I')$ for the C*-ideal $I' \subseteq
    A'$ according to Lemma \ref{lem:suspendedQuotient}.
    Then $h\colon A \to A'$ gives rise to a diagram:
    \[
        \begin{tikzpicture}
        \mama {
            SA/SI & S(A/I)
            \\
            SA'/SI' & S(A'/I').
            \\
        };
        \path[->, font=\scriptsize]
            (m-1-1) edge node[auto]{$\Phi$} (m-1-2)
            (m-1-1) edge node[left]{$Sh / S(h \restr I)$} (m-2-1)
            (m-2-1) edge node[below]{$\Phi'$} (m-2-2)
            (m-1-2) edge node[right]{$S\big(h/(h \restr I)\big)$} (m-2-2)
        ;
        \end{tikzpicture}
    \]
    This diagram commutes: Consider $(f + SI) \in SA/SI$ for a given
    $f\colon [0,1] \to A$ with $f(0) = f(1) = 0$.
    The upper right path through the diagram maps this to
    $t \mapsto f(t) + I$ and then to the class containing
    $t \mapsto (h \circ f)(t) + h(I)$, which is
    $t \mapsto (h \circ f)(t) + I'$ in $S(A'/I')$.

    The lower left path maps $f + SI$ first to $h \circ f + h(SI) + SI'$,
    which is $h \circ f + SI'$ since $h(I) \subseteq I'$, and then onwards
    also to $t \mapsto (h \circ f)(t) + I'$ according to the
    construction of $\Phi'$.
\end{proof}

\begin{pro}
\label{pro:compatibilitySuspensionBoundary}
    Let $0 \to A \to B \to C \to 0$
    be a short exact sequence of C*-algebras.
    Then for each $s \in \zet$, the following diagram is commutative:
    \[
        \begin{tikzpicture}
            \mamax{2cm} {
                K_{s+1}(C)
                & K_s(A)
                \\
                K_s(SC)
                & K_{s-1}(SA).
                \\
            };
            \path[->, font=\scriptsize]
                (m-1-1) edge node[above]{$\del_{s+1}(C,A)$} (m-1-2)
                (m-1-1) edge node[left]{$\sigma_{s+1}(C)$}
                             node[right]{$\cong$} (m-2-1)
                (m-2-1) edge node[below]{$\del_s(SC,SA)$} (m-2-2)
                (m-1-2) edge node[right]{$\sigma_s(A)$}
                             node[left]{$\cong$} (m-2-2)
            ;
        \end{tikzpicture}
    \]
    Here $\sigma$ denotes suspension isomorphisms and $\del$ denotes the
    boundary maps in the long exact K-theory sequences that arise from the
    original short exact sequence and from its suspension
    $0 \to SA \to SB \to SC \to 0$.
\end{pro}

Even though $\partial$ is natural with respect to \starhoms{},
the claim does not follow immediately from functoriality because
the suspension isomorphisms $\sigma$
arise only in K-theory, not on the level of C*-algebras.

\begin{proof}[Proof of Proposition \ref{pro:compatibilitySuspensionBoundary}]
    The \emph{Bott isomorphism}
    \[
        \begin{tikzpicture}
            \mamax{1.5cm} {
                \beta(C) \colon K_0(C) & K_1(SC) &K_2(C), \\
            };
            \path[->, font=\scriptsize]
                (m-1-1) edge node[above] {$\cong$} (m-1-2)
                (m-1-2) edge node[above] {$\sigma_2(C)^{-1}$} (m-1-3)
            ;
        \end{tikzpicture}
    \]
    which is a composition of two isomorphisms,
    and the \emph{exponential map}
    \[
        \begin{tikzpicture}
            \mamax{1.5cm} {
                \delta_0\colon K_0(C) & K_2(C) &K_1(A) \\
            };
            \path[->, font=\scriptsize]
                (m-1-1) edge node[above] {$\beta(C)$} (m-1-2)
                (m-1-2) edge node[above] {$\del_2(C,A)$} (m-1-3)
            ;
        \end{tikzpicture}
    \]
    are constructed in \cite[Chapters 11--12]{rordam} explicitly to make
    the two outermost paths $\sigma_1(A) \circ \delta_0$ and
    $\del_1(SC,SA) \circ \sigma_2(C) \circ \beta(C)$
    in the following diagram commute:
    \[
        \begin{tikzpicture}
            \mamax{2cm} {
                K_{0}(C)
                & K_2(C)
                & K_1(A)
                \\
                & K_1(SC)
                & K_0(SA).
                \\
            };
            \path[->, font=\scriptsize]
                (m-1-1) edge node[above]{$\beta(C)$}
                             node[below]{$\cong$} (m-1-2)
                (m-1-1) edge[bend left=40] node[above]{$\delta_0$} (m-1-3)
                (m-1-2) edge node[above]{$\del_2(C,A)$} (m-1-3)
                (m-1-2) edge node[left]{$\sigma_2(C)$}
                             node[right]{$\cong$} (m-2-2)
                (m-2-2) edge node[below]{$\del_1(SC,SA)$} (m-2-3)
                (m-1-3) edge node[right]{$\sigma_1(A)$}
                             node[left]{$\cong$} (m-2-3)
            ;
        \end{tikzpicture}
    \]
    Because $\beta(C)$ is an isomorphism, commutativity of the square
    follows from commutativity of the two outermost paths.
    This commuting square proves the claim for $n = 1$,
    the lowest $n$ of interest.

    For higher $s$, the claim follows from this base case by composing
    the entire diagram with an $(s-1)$-fold suspension isomorphism.
    The higher boundary maps in K-theory are defined precisely to agree
    with such a suspension. The claim for lower $s$ follows from composing with
    Bott isomorphisms.
\end{proof}

\subsection{Linking spectral sequences}

Having linked the chain of ideal inclusions $Q_p \to Q_{p+1}$ with $\widetilde
Q_p \to \widetilde Q_{p+1}$, we can now link the spectral sequences that arise
from the first $n$ ideals along increasing cardinalities $n \in \nat$.

\begin{pro}
\label{pro:eToETildeIso}
    Let $A = I_0 + I_1 + \dotsb + I_{n-1}$ be a sum of C*-ideals.
    Construct the spectral sequence $\{E_{*,*}^r, d^r\}_r$
    for this $n$-fold sum as in
    Theorem \ref{thm:sseqforsums}.

    Alternatively, add an extra ideal $I_n = 0$ and construct a second spectral
    sequence $\{\widetilde E_{*,*}^r, \widetilde d^r\}_r$
    for the $(n+1)$-fold sum $I_0 + I_1 + \dotsb + I_{n-1} + 0$.

    Then there are isomorphisms
    $E_{p,q}^r \cong \widetilde E_{p,{q-1}}^r$ for all $r \geq 1$ and
    $p$, $q \in \zet$
    with the following properties:
    \begin{itemize}
        \item
            They are natural with respect to \starhoms{} that
            preserve the ideal decompositions: \starhoms{}
            $h\colon A \to A'$
            for a C*-algebra $A' = I'_0 + I'_1 + \dotsb + I'_{n-1} + 0$
            such that, for all $j < n$, we have $h(I_j) \subseteq I'_j$.
        \item
            They commute with the differentials $d^r$ and
            $\widetilde d^r$.
    \end{itemize}
\end{pro}

\begin{proof}
    Recall the first page of the spectral sequence
    from the main statement of Theorem \ref{thm:sseqforsums}, adapted
    to $n$ instead of $n+1$ ideals:
    \[
        E^1_{p,q} \cong
            \begin{dcases*}
                \bigoplus_{\card{J} = p+1} K_q \Big(
                    \bigcap_{j \in J} I_j \Big)
                & for $0 \leq p < n$,\\
                0 & for $p < 0$ or $p \geq n$.
            \end{dcases*}
    \]
    Suspend $Q_p/Q_{p-1}$ and compensate for this suspension by a degree shift
    in K-theory to maintain isomorphy.
    Apply our isomorphism results from above:
    \begin{align*}
        E^1_{p,q}
        &\cong K_q(Q_p / Q_{p-1}) \\
        &\cong K_{q-1}\big(S(Q_p/Q_{p-1})\big) \\
        &\cong K_{q-1}(SQ_p/SQ_{p-1}) \\
        &\cong K_{q-1}(\widetilde Q_p / \widetilde Q_{p-1}) \\
        &\cong \widetilde E^1_{p,{q-1}}.
    \end{align*}
    All isomorphisms come from our earlier propositions and are therefore
    natural. We do not have to show anything new here for naturality with
    respect to $h\colon A \to A'$ of the isomorphism between the two spectral
    sequences.

    Commutation of isomorphisms and differentials
    follows from the definition of the
    differentials according to Theorem \ref{thm:hsystemsseq} and Definition
    \ref{dfn:howhbecomesk2}:
    The differentials are compositions of maps induced by ideal inclusions,
    maps induced by natural quotient projections, and boundary maps from
    long exact sequences in K-theory that arise from ideal inclusions.
    All these maps commute individually with all isomorphisms applied
    above; in particular,
    Proposition \ref{pro:compatibilitySuspensionBoundary} guarantees that
    the differentials behave well with suspensions.
\end{proof}

\begin{dfn}[Link between spectral sequences for increasingly-many ideals]
\label{dfn:linkBetweenSseqs}
    Let $A(n+1) = I_0 + I_1 + \dotsb + I_n$ be a sum of $n+1$
    C*-ideals and denote by $A(n) \subseteq A(n+1)$
    the sum of the first $n$ ideals, $I_0 + I_1 + \dotsb + I_{n-1}$.

    For all pages $r \geq 1$ and $p$, $q \in \zet$, let $\{g^r_{p,q}\}\colon
    \{E(n)^r_{p,q}, d(n)^r\} \to \{\widetilde E^r_{p,q-1},
        \widetilde d^r\}
    $ be the isomorphism constructed in Proposition \ref{pro:eToETildeIso}
    from the spectral sequence for $A(n)$ into the
    spectral sequence for $I_0 + I_1 + \dotsb + I_{n-1} + 0$.
    The ideal inclusions
    \[
        I_0 = I_0, \qquad
        I_1 = I_1, \qquad
        \ldots, \qquad
        I_{n-1} = I_{n-1}, \qquad
        0 \to I_n
    \]
    induce a morphism $\{i^r_{p,q}\}$ that preserves all degrees,
    \[
        \{i^r_{p,q}\} \colon
        \{\widetilde E^r_{p,q}\} \to \{E(n+1)^r_{p,q}\},
    \]
    of spectral
    sequences between the intermediate
    $\{\widetilde E^r_{p,q}, \widetilde d^r\}_{r,p,q}$
    and the desired spectral sequence
    $\{E(n+1)^r_{p,q}, d(n+1)\}_{r,p,q}$ for $A(n+1)$.

    The \emph{link} between the spectral sequences for $A(n)$ and $A(n+1)$ is
    \[
        \ell(n) = \{\ell(n)^r_{p,q}\} = \{i^r_{p,q-1} \circ g^r_{p,q}\}
        \colon \{E(n)^r_{p,q}\} \to \{E(n+1)^r_{p,q-1}\}.
    \]
\end{dfn}

\begin{lem}[Faithfulness of the link $\ell(n)$]
\label{lem:faithfulLinks}
    Given $n \in \nat$ and $p, q \in \zet$, consider a set of indices
    $J \subs {\{0, 1, \ldots, n-1\}}$ with $\card J = p+1$
    and the direct summand
    $V = K_{q}\big(\bigcap_{j \in J} I_j\big)$ for $J$
    in the group $E(n)^1_{p,q}$.

    Then $\ell(n)^1_{p,q}$ maps $V$ isomorphically onto
    $V' = K_{q-1}\big(\bigcap_{j \in J} I_j\big)$ in $E(n+1)^1_{p,q-1}$
    No other summand of $E(n)^1_{p,q}$ besides $V$ maps onto $V'$ nontrivially.
    No other summand in $E(n+1)^1_{p,q-1}$ besides $V'$ is a nontrivial
    image of $V$.
\end{lem}

\begin{proof}
    We examine the components of $\ell(n)^1_{p,q}$ in the notation
    of Definition \ref{dfn:linkBetweenSseqs}.

    The first component of $\ell(n)^1_{p,q}$ is
    $g^1_{p,q}\colon E(n)^1_{p,q} \to \widetilde E^1_{p,q-1}$. As a suspension
    isomorphism, it agrees with quotients, direct sums,
    and other suspensions. Because it is natural with respect to
    \starhoms{} that preserve ideal decompositions, $g^1_{p,q}$ cannot
    permute $V$ with other direct summands for different
    choices of ideals than $J$ among the first $n$ ideals. Again by naturality,
    it cannot map to
    $K_{q-1}\big(I_n \cap \bigcap_{j \in J} I_j\big)$ within
    $\widetilde E^1_{p,q-1}$ either: This K-theory group must always vanish
    regardless of $A$ because, by definition,
    $\{\widetilde E^1_{p,q}\}_{r,p,q}$ is the spectral sequence for
    $I_n = 0$.

    On all ideals $I_j$ with $j \neq n$, the second component
    $i^1_{p,q-1}\colon \widetilde E^1_{p,q-1} \to E(n+1)^1_{p,q-1}$
    is induced by the identity. It maps
    $g^1_{p,q}(V)$ necessarily to $V'$
    because $n \notin J$ and hits no other summands besides $V'$.

    Thus $\ell(n)^1_{p,q} = i^1_{p,q-1} \circ g^1_{p,q}$ has the claimed
    isomorphy property and hits no other summands besides $V'$.
    It follows that no other summand in $E(n+1)^1_{p,q-1}$
    besides $V'$ may be a nontrivial image of $V$.
\end{proof}

\begin{pro}
    The link $\ell(n)$
    between the spectral sequences $\{E(n)^r_{p,q}, d(n)^r\}_{r,p,q}$
    and $\{E(n+1)^r_{p,q}, d(n+1)^r\}_{r,p,q}$ is natural
    with respect to \starhoms{} $h\colon A \to A'$ for another
    C*-algebra
    $A' = I'_0 + I'_1 + \dotsb + I'_{n'}$ that is a sum of
    $n'+1 \geq n+1$ C*-ideals,
    as long as $h(I_j) \subseteq I'_j$ for all $j \leq n$.
\end{pro}

\begin{proof}
    We have constructed $\ell(n)$ as a composition of two maps that already
    satisfy this desired naturality with respect to $h$ as shown in the
    various earlier lemmas.
\end{proof}

    \newpage
    
\subsection{Main theorem for countably many ideals}
\label{sec:unsereSseqUnendlich}

\begin{thm}[Spectral sequence for countable sums]
\label{thm:unsereSseqUnendlich}
    Let $A$ be the direct
    limit C*-algebra $\overline{I_0 + I_1 + \dotsb + I_j + \dotsb}$ of sums of
    countably many C*-ideals $I_j \subseteq A$.
    There is a spectral sequence
    $\{E^r_{p,q}, d^r\}_{r,p,q}$ with
    \begin{equation}
    \label{eqn:unsereSseqUnendlich}
        \sseqUnendlich
    \end{equation}
    where $J$ ranges over all nonempty finite subsets of indices.
    In general, this is a half-page spectral sequence, any term $E^1_{p,q}$
    with $p \geq 0$ may be nonzero.

    This spectral sequence converges strongly to $K_* A$.
    It is functorial with respect to \starhoms{} that preserve
    countable ideal decompositions.
\end{thm}

We prove Theorem \ref{thm:unsereSseqUnendlich} in two steps:
First, in Proposition \ref{pro:sseqUnendlichExistsNoLinkDef},
we prove existence,
that the spectral sequence is well-defined, that it is functorial,
and that its $E^1_{*,*}$-term matches the description
\ref{eqn:unsereSseqUnendlich}.
Later, in Proposition \ref{pro:unsereSseqUnendlichConvergence},
we prove the strong convergence.

Our strategy is to take the direct limit along
a directed system of links $\ell(n)$ for
$n \to \infty$, but these links $\ell(n)$ do not connect
the spectral sequences $E(n)$ and
$E(n+1)$ perfectly.
There is an index shift from $E(n)^r_{p,q}$ to
$E(n+1)^r_{p,q-1}$. As shown before, the shifted index affects
the degree of the K-theory;
it has no other effect on equation \ref{eqn:unsereSseqUnendlich}.

Complex K-theory admits Bott isomorphisms $\beta$,
\[
    \dotsb
    \stackrel{\beta}{\cong} K_{s-2}\Big( \bigcap_{j \in J} I_j \Big)
    \stackrel{\beta}{\cong} K_{s}\Big( \bigcap_{j \in J} I_j \Big)
    \stackrel{\beta}{\cong} K_{s+2}\Big( \bigcap_{j \in J} I_j \Big)
    \stackrel{\beta}{\cong} \dotsb;
\]
their naturality allows us to work around the index shift.

\begin{rem}
    The Bott isomorphisms $\beta$ are natural with respect to
    $\ast$-homo\-mor\-phisms
    and commute with the suspension isomorphisms $\sigma$
    for any C*-algebra $A$:
    \[
        \begin{tikzpicture}
            \mama {
                K_s (A) &
                K_{s+2} (A)
                \\
                K_{s-1} (SA) &
                K_{s+1} (SA).
                \\
            };
            \path[->, font=\scriptsize]
                (m-1-1) edge node[above]{$\beta$} (m-1-2)
                (m-2-1) edge node[below]{$\beta$} (m-2-2)
                (m-1-1) edge node[left]{$\sigma_s$} (m-2-1)
                (m-1-2) edge node[right]{$\sigma_{s+2}$} (m-2-2)
            ;
        \end{tikzpicture}
    \]
\end{rem}

\begin{dfn}[Degree-amending link $\lambda(n)$]
\label{dfn:degreeAmendingLink}
    Compose two links with the Bott isomorphism to create a morphism
    $\lambda(n)$ of bidegree $(0, 0)$, called the
    \emph{degree-amending link}, between two spectral sequences:
    \[
        \lambda(n) = \beta \circ \ell(n+1) \circ \ell(n)
            \colon \{E(n)^r_{p,q}\}_{r,p,q} \to \{E(n+2)^r_{p,q}\}_{r,p,q}.
    \]
\end{dfn}

\begin{pro}
\label{pro:sseqUnendlichExistsNoLinkDef}
    Let $A = \overline{I_0 + I_1 + \dotsb + I_j + \dotsb}$
    be a sum of C*-ideals.
    The spectral sequence postulated in Theorem \ref{thm:unsereSseqUnendlich}
    with
    \[
        \sseqUnendlich
    \]
    exists and is functorial with respect to \starhoms{} that
    preserve countable ideal decompositions.
\end{pro}

\begin{proof}
    Take the direct limit along $\lambda(2n)$ for $n \to \infty$.
    This yields again a spectral sequence.
    Functoriality of this spectral sequence with respect to ideal inclusions
    follows from all earlier constructions that, as we have remarked
    repeatedly, are natural with respect to \starhoms{} that preserve
    countable ideal decompositions.

    It remains to show that the choice of offset, i.e., $2n$ or $2n+1$,
    and the position of $\beta$ among the links $\ell$
    have no effect on the direct limit of degree-amending links;
    i.e., the following system $\lambda'(2n)$
    of morphisms produces the same direct limit:
    \[
        \lambda'(n) = \ell(n+1) \circ \beta \circ \ell(n)
            \colon \{E(n)^r_{p,q}\} \to \{E(n+2)^r_{p,q}\}.
    \]
    The position of $\beta$ is irrelevant because the Bott isomorphism is
    natural with respect to both \starhom{}s and suspension isomorphisms. In
    Definition
    \ref{dfn:linkBetweenSseqs}, we have defined $\ell(n)$ as a composition
    of several suspension isomorphisms and several
    direct morphisms between C*-algebras. Thus
    for $n \to \infty$, the systems $\lambda(2n)$ and $\lambda'(2n)$ produce
    naturally isomorphic direct limits, without even requiring index shifts.

    To compare the systems $\lambda(2n)$ and $\lambda(2n+1)$ for
    $n \to \infty$, unroll the compositions $\lambda$ into their definitions,
    and take the direct limit -- necessarily yielding the same limit as before
    -- across the unrolled system:
    \begin{align*}
        \colim_{n \to \infty} \lambda(2n+1)
        &= \colim \big(
        \dotsb \to \bullet \loto{\ell(n+1)} \bullet \loto{\ell(n+2)}
            \bullet \loto{\beta} \bullet \loto{\ell(n+3)} \bullet \to \dotsb
            \big)
        \\
        &= \colim \big(
        \dotsb \to \bullet \loto{\ell(n+1)} \bullet
            \loto{\beta} \bullet \loto{\ell(n+2)}
            \bullet \loto{\ell(n+3)} \bullet \to \dotsb
            \big)
        \\
        &= \colim \big(
        \dotsb \to \bullet \loto{\ell(n+2)}
            \bullet \loto{\ell(n+3)} \bullet \loto{\beta} \bullet \to \dotsb
            \big)
        \\
        &= \colim_{n \to \infty} \lambda(2n),
    \end{align*}
    because the morphisms $\ell$ commute with $\beta$ and because removing
    the first element of a directed system will not change the limit.

    Thus our limit spectral
    sequence is well-defined as $\colim_n \lambda(2n)$.
    Finally, because
    \[
        E(n)^r_{p,q}
        \cong
            \begin{dcases*}
                \bigoplus_{\substack{\card J = p + 1 \\ \max{J} < n}}
                K_q \Big( \bigcap_{j \in J} I_j \Big)
                & for $0 \leq p < n$,
                \\
                0 & for $p < 0$ or $p \geq n$,
            \end{dcases*}
    \]
    and because the links $\ell(n)$
    preserve all structure due to their faithfulness
    (Lemma \ref{lem:faithfulLinks}),
    the direct limit $E^r_{p,q}$
    along the K-theory morphisms $\lambda(2n)$
    for $n \to \infty$ is the desired direct sum of K-theory groups
    over all nonempty subsets $J \subseteq \nat$ with
    $\card J = p + 1$ without restrictions about any maximum of $J$.
\end{proof}

\subsection{Convergence}
\label{sec:sseqUnendlichConvergence}

To prove the convergence of the spectral sequence for countable
sums of ideals, we will define a filtration of $K_* A$ via direct limits of
existing filtrations.

Even though the following lemma about direct limits is known theory,
we reprove it with attention to detail, tracking the naturality of all
constructions.

\begin{lem}[Continuity of inclusion chains of abelian group quotients]
\label{lem:continuityOfQuotients}
    Let $(G_n)_{n \in \nat}$ be a directed system of abelian
    groups along inclusions $i_n\colon G_n \subseteq G_{n+1}$.
    For all $n \in \nat$,
    let $H_n \subseteq G_n$ be a subgroup such that
    $H_n \subseteq H_{n+1}$. Write
    \[
        G = \colim_{n \to \infty} (G_n, i_n), \qquad
        H = \colim_{n \to \infty} (H_n, i_n \restr H_n).
    \]
    Then there is an isomorphism
    \[
        \Phi\colon G/H \cong \colim_{n \to \infty} (G_n / H_n)
    \]
    that is natural with respect to morphisms
    between systems $(G_n)_{n \in \nat}$
    and $(\widetilde G_n)_{n \in \nat}$ that preserve their inclusions and
    their respective systems of subgroups
    $(H_n \subseteq G_n)_{n \in \nat}$ and
    $(\widetilde H_n \subseteq \widetilde G_n)_{n \in \nat}$.
\end{lem}

\begin{proof}
    For each $n \in \nat$, consider the following diagram $D_n$.
    All of the horizontal arrows in $D_n$
    are natural projections and the square
    commutes due to naturality of the projections:
    \[
        \begin{tikzpicture}
            \mamax{2cm} {
                G_n & G_n/H_n & \\
                G & G/H_n & G/H. \\
            };
            \path[->, font=\scriptsize]
                (m-1-1) edge node[above]{$\alpha_n$} (m-1-2)
                (m-1-2) edge node[right]{$\gamma_n/H_n$} (m-2-2)
                (m-1-1) edge node[left]{$\gamma_n$}(m-2-1)
                (m-2-1) edge node[below]{$\delta_n$} (m-2-2)
                (m-2-2) edge node[below]{$\pi_n$} (m-2-3)
            ;
        \end{tikzpicture}
    \]
    To construct a direct limit of the sequence of entire diagrams
    $(D_n)_{n \in \nat}$, link two diagrams $D_n$ and $D_{n+1}$ by
    the following system of morphisms:
    \begin{itemize}
        \item $i_n\colon G_n \to G_{n+1}$;
        \item the composition
            $(\tau_{n+1} \circ i_n/H_n)\colon G_n/H_n \to G_{n+1}/H_n \to
            G_{n+1}/H_{n+1}$ where the natural projection
            $\tau_{n+1}\colon G_{n+1}/H_n \to G_{n+1}/H_{n+1}$
            is well-defined because $H_n \subseteq H_{n+1}$;
        \item the natural projection $G/H_n \to G/H_{n+1}$
            from dividing by $H_{n+1}/H_n$; and
        \item the identities on $G$ and $G/H$, respectively.
    \end{itemize}
    All squares that arise from linking two diagrams $D_n$ and $D_{n+1}$
    via these morphisms commute due to naturality of projections; e.g.,
    $\alpha_{n+1} \circ i_n = \tau_{n+1} \circ i_n/H_n \circ \alpha_n$.
    Take the direct limit along $(D_n)_{n \in \nat}$; the limit object is again
    a commutative diagram:
    \begin{equation}
    \label{eqn:limitDiagramAlpha}
        \begin{tikzpicture}
            \mamax{2cm} {
                G & \colim_n (G_n/H_n) & \\
                G & \colim_n (G/H_n) & G/H. \\
            };
            \path[-]
                (m-1-1) \maeq (m-2-1)
            ;
            \path[->, font=\scriptsize]
                (m-1-1) edge node[above]{$\colim_n \alpha_n$} (m-1-2)
                (m-1-2) edge node[right]{$\colim_n(\gamma_n/H_n)$} (m-2-2)
                (m-2-1) edge node[below]{$\colim_n \delta_n$} (m-2-2)
                (m-2-2) edge node[below]{$\colim_n \pi_n$} (m-2-3)
            ;
        \end{tikzpicture}
    \end{equation}
    Furthermore, for each $n \in \nat$, there is a short exact sequence
    \[
        0 \to H_n \to G_n \loto{\alpha_n} G_n/H_n \to 0.
    \]
    Taking direct limits of abelian groups is an exact functor, resulting
    in a short exact sequence of direct limits which fits into the following
    diagram as the top row:
    \[
        \begin{tikzpicture}
            \mamax{2cm} {
                0 & H & G & \colim_n (G_n / H_n) & 0 \\
                0 & H & G & G/H & 0. \\
            };
            \path[-]
                (m-1-2) \maeq (m-2-2)
                (m-1-3) \maeq (m-2-3)
            ;
            \path[->, font=\scriptsize]
                (m-1-1) edge (m-1-2)
                (m-1-2) edge (m-1-3)
                (m-1-3) edge node[above] {$\colim_n \alpha_n$} (m-1-4)
                (m-1-4) edge (m-1-5)
                (m-2-1) edge (m-2-2)
                (m-2-2) edge (m-2-3)
                (m-2-3) edge node[below] {
                    $\colim_n (\pi_n \circ \delta_n)$} (m-2-4)
                (m-2-4) edge (m-2-5)
                (m-1-4) edge node[right] {
                    $\colim_n (\pi_n \circ \gamma_n/H_n)$} (m-2-4)
            ;
        \end{tikzpicture}
    \]
    The right square commutes because its morphisms already commuted
    in the earlier diagram \ref{eqn:limitDiagramAlpha} of direct limits.

    Finally, because both rows are short exact sequences,
    the five lemma guarantees that
    $\colim_n (\pi_n \circ \gamma_n/H_n)$ is the desired
    isomorphism $\Phi$.
    Naturality of $\Phi$ follows from the constructions in this proof:
    Both taking direct limits and taking quotients is natural.
\end{proof}

\begin{ntt}
\label{ntt:anForLimit}
    For a sum $A = \overline{I_0 + I_1 + \dotsb + I_j + \dotsb}$
    of C*-ideals, write
    \[
        A(n) = I_0 + I_1 + \dotsb + I_{n-1}
    \]
    for the sum of the first $n$ ideals and let, for $s \in \zet$,
    \[
        a(n)_s \colon K_s A(n) \to K_s A(n+1)
    \]
    be the map induced in K-theory by the inclusion of C*-algebras
    $A(n) \subseteq A(n+1)$.

    Given $A(n)$, define the cake algebras
    $B(n)_J = B(I_0, I_1, \ldots, I_{n-1})_J$ for index sets
    $J \subseteq \{0, 1, \ldots, n-1\}$
    as in Definition \ref{dfn:cakealgebra}
    and the sums of cake algebras as in
    Definition \ref{dfn:cakeSums},
    \[
        Q(n)_p = \sum_{\card J \leq p + 1} B(n)_J
    \]
    for $p \in \zet$ and $J \subseteq \{0, 1, \ldots, n-1\}$.
\end{ntt}

\begin{rem}
    By Lemma \ref{lem:cakeSumQuotients}, we have a chain of ideals
    $Q(n)_p \subseteq Q(n)_{p+1}$ across all $p \in \zet$.
    The chain stabilizes with $Q(n)_p = 0$ for $p < 0$
    and $Q(n)_p = B(n)_{\{0, 1, \ldots, n-1\}}$ for $p \geq n - 1$.

    Furthermore, $\sum_{p=0}^{n-1} Q(n)_p = B(n)_{\{0, 1, \ldots, n-1\}}
        \cong S^{n-1} A(n)$
    by Theorem \ref{thm:inclusioninducesiso}.
\end{rem}

\begin{ntt}[$i(n, p)_s$]
\label{ntt:inps}
    We recall the filtrations on the spectral sequences for finitely many
    ideals: The convergence target of $\{E(n)^r_{p,q}, d(n)^r\}$ is
    $K_* A(n)$, filtered by
    \begin{align}
    \label{eqn:fpSusp}
        F^p K_s A(n) &\cong
            \im i\colon K_{s-n+1} Q(n)_p \to K_{s-n+1} S^{n-1} A(n)
            \\
        \label{eqn:fpShift}
            &\cong \im i(n, p)_s\colon K_{s-n+1} Q(n)_p \to K_s A(n)
    \end{align}
    as in Definition \ref{dfn:filtrationOfKA} for $p$, $s \in \zet$.
    In the construction of the spectral sequence in Section
    \ref{sec:schochetSseq}, the symbol $i$ may also stand for various
    other maps in K-theory.

    Here in Section \ref{sec:sseqUnendlichConvergence}, we will write
    $i(n, p)_s$ for the maps in \ref{eqn:fpShift}
    that define a filtration of $K_s A(n)$, not of $K_{s-n+1} S^{n-1} A(n)$,
    for a given $p$. Normally, we would index maps in K-theory with the
    degree of the domain, but for $i(n, p)_s$, it will be easiest to track
    the K-theory degree $s$ of the desired convergence target $K_s A(n)$.

    We reserve $i$
    (without annotation in parentheses) to discuss
    internals of Section \ref{sec:schochetSseq}.
\end{ntt}

\begin{rem}
    Because $Q(n)_p$ is an ideal in $S^{n-1} A(n)$, not in $A(n)$,
    we have shifted the
    K-theoretic degree from $K_s Q(n)_p$ to
    $K_{s-n+1} Q(n)_p$ to compensate.
    This shift is similar to the shift in the proof of Theorem
    \ref{thm:sseqforsums} about the spectral sequence for sums of finitely many
    C*-ideals.
\end{rem}

\begin{dfn}[Link between filtrations]
\label{dfn:linkBetweenFiltrations}
    Fix $n \in \nat$ and $p \in \zet$.
    Besides $Q(n)_p$, recall the C*-algebra
    $\widetilde Q(n)_p$ for the $n+1$ ideals $I_0$, $I_1$, $\ldots$,
    $I_{n-1}$, $0$
    as in Proposition \ref{pro:conncetQpWithQpPlusOne}.
    For all $s \in \zet$, define the
    \emph{link between the filtrations of $K_* A(n)$ and
    $K_* A(n+1)$},
    \[
        \psi(n,p)_s \colon K_{s-n+1} Q(n)_p \to K_{s-n} Q(n+1)_p,
    \]
    as the composition
    \[
        \begin{tikzpicture}
            \mamaxy{1.2cm}{0.3cm} {
                \phantom{K_{s-n+1} Q(n)_p}
                &&& \phantom{K_{s-n} \widetilde Q(n)_p}
                \\
                K_{s-n+1} Q(n)_p
                & K_{s-n} SQ(n)_p
                & K_{s-n} \widetilde Q(n)_p,
                & K_{s-n} Q(n+1)_p;
                \\
            };
            \path[-, font=\scriptsize]
                (m-2-1.north) edge (m-1-1.center)
                (m-1-1.center) edge node[above]{$\psi(n,p)_s$} (m-1-4.center)
            ;
            \path[->, font=\scriptsize]
                (m-1-4.center) edge (m-2-4.north)
                (m-2-1) edge node[below] {$\cong$} (m-2-2)
                (m-2-2) edge node[below] {$\cong$} (m-2-3)
                (m-2-3) edge node[below] {$a(n) + 0$} (m-2-4)
            ;
        \end{tikzpicture}
    \]
    here the left arrow is the suspension isomorphism,
    the middle arrow is the natural isomorphism constructed in
    Proposition \ref{pro:conncetQpWithQpPlusOne}, and the right arrow
    is induced by the inclusion
    of the $(n+1)$-fold ideal decomposition $A(n) + 0$ into $A(n+1)$.
\end{dfn}

\begin{lem}
    For all $n \in \nat$ and $p$, $s \in \zet$, the following diagram
    commutes:
    \begin{equation}
    \label{eqn:filtrationLinksCommute}
        \begin{tikzpicture}
            \mamax{2cm} {
                K_{s-n+1} Q(n)_p
                & K_{s-n} Q(n+1)_p
                \\
                K_s A(n) & K_s A(n+1).
                \\
            };
            \path[->, font=\scriptsize]
                (m-1-1) edge node[above] {$\psi(n,p)_s$} (m-1-2)
                (m-1-1) edge node[left] {$i(n,p)_s$} (m-2-1)
                (m-1-2) edge node[right] {$i(n+1,p)_s$} (m-2-2)
                (m-2-1) edge node[below] {$a(n)_s$} (m-2-2)
            ;
        \end{tikzpicture}
    \end{equation}
\end{lem}

\begin{proof}
    Replace $\psi(n,p)_s$ by its definition as a composition of three maps
    to get the top row of the following diagram where $i'$
    denotes the map induced by the ideal inclusion
    $Q(n)_p \subseteq S^{n-1} A(n)$ with appropriate K-theoretic degree shifts:
    \[
        \begin{tikzpicture}
            \mamax{0.8cm} {
                K_{s-n+1} Q(n)_p
                & K_{s-n} SQ(n)_p
                & K_{s-n} \widetilde Q(n)_p,
                && K_{s-n} Q(n+1)_p
                \\
                K_s A(n) & K_s A(n) &&& K_s A(n+1).
                \\
            };
            \path[->, font=\scriptsize]
                (m-1-1) edge node[above] {$\cong$} (m-1-2)
                (m-1-2) edge node[above] {$\cong$} (m-1-3)
                (m-1-3) edge node[above] {$a(n)_s + 0$} (m-1-5)
                (m-1-1) edge node[left] {$i(n,p)_s$} (m-2-1)
                (m-1-2) edge node[left] {$i'$} (m-2-2)
                (m-1-5) edge node[right] {$i(n+1,p)_s$} (m-2-5)
                (m-2-2) edge node[below] {$a(n)_s$} (m-2-5)
            ;
            \path[-]
                (m-2-1) \maeq (m-2-2)
            ;
        \end{tikzpicture}
    \]
    The left square commutes by definition of $i'$
    because the top arrow is the suspension isomorphism.

    The right side commutes because, by Remark
    \ref{rem:qTildeIsNatural},
    the isomorphism at the top
    is natural with respect to \starhoms{} that preserve $(n+1)$-fold
    ideal decompositions: The two \starhoms{} here are the ideal inclusion
    that induces $a(n)$ and
    the ideal inclusion $Q(n+1)_p \to S^{n} A(n+1)$
    that induces $i'$ and $i(n+1,p)_s$ and restricts to $\widetilde Q(n)_p$;
    these two ideal inclusions commute with each other already on the level
    of C*-algebras.
\end{proof}

\begin{dfn}[Filtration of $K_* A$]
\label{dfn:filtrationOfLimit}
    Fix $p$, $s \in \zet$ and compose diagram
    \ref{eqn:filtrationLinksCommute} with itself across all $n \geq 1$:
    \[
        \begin{tikzpicture}
            \mamax{1.2cm} {
                K_{s} Q(1)_p
                & K_{s-1} Q(2)_p
                & K_{s-2} Q(3)_p
                & \dotsb
                \\
                K_s A(1) & K_s A(2) & K_s A(3)
                & \dotsb.
                \\
            };
            \path[->, font=\scriptsize]
                (m-1-1) edge node[left] {$i(1,p)_s$} (m-2-1)
                (m-1-2) edge node[left] {$i(2,p)_s$} (m-2-2)
                (m-1-3) edge node[left] {$i(3,p)_s$} (m-2-3)
                (m-2-1) edge node[below] {$a(1)_s$} (m-2-2)
                (m-2-2) edge node[below] {$a(2)_s$} (m-2-3)
                (m-2-3) edge node[below] {$a(3)_s$} (m-2-4)
                (m-1-1) edge node[above] {$\psi(1, p)_s$} (m-1-2)
                (m-1-2) edge node[above] {$\psi(2, p)_s$} (m-1-3)
                (m-1-3) edge node[above] {$\psi(3, p)_s$} (m-1-4)
            ;
        \end{tikzpicture}
    \]
    The direct limit of $K_s A(n)$ for $n \to \infty$ along $a(n)$ is
    $K_s A$ by continuity of K-theory.
    Take the direct limit of $K_{s-n+1} Q(n)_p$ for $n \to \infty$ along
    $\psi(n,p)_s$ and consider, by functoriality of the direct limit,
    the direct limit of the vertical arrows $i(n,p)_s$,
    \[
        \colim_{n \to \infty} i(n,p)_s\colon
            \Big( \colim_{n \to \infty} K_{s-n+1} Q(n)_p \Big) \to K_s A.
    \]
    With this map, define the \emph{filtration of $K_* A$},
    $\{F^p K_* A\}_{p \in \zet}$, by
    \[
        F^p K_s A = \im \Big( \colim_{n \to \infty} i(n,p)_s \Big)
            \subseteq K_s A.
    \]
\end{dfn}

\begin{lem}
\label{lem:limitFiltrationHausdorff}
    The filtration $\{F^p K_s A\}_{p \in \zet}$ of $K_s A$ from Definition
    \ref{dfn:filtrationOfLimit} is an increasing filtration.
\end{lem}

\begin{proof}
    For all $p \leq p'$, we have $Q(n)_p \subseteq Q(n)_{p'}$.
    Using $i$ as in the notation of Section
    \ref{sec:schochetSseq},
    $i\big(K_{s-n+1} Q(n)_p\big) \subseteq K_{s-n+1} Q(n)_{p'}$.
    Because the $i(n,p)_s$ arise from the directed system of morphisms $i$
    from Remark \ref{rem:iDiagramCommutes},
    \[
        K_s Q(n)_p \loto{i} K_s Q(n)_{p+1} \loto{i}
            \dotsb \loto{i} K_s Q(n)_n = K_s Q(n)_{n+1} = \dotsb,
    \]
    merely via isomorphisms that implement a K-theoretic
    degree shift, we may pull back our direct limit construction
    for $\colim_n i(n,p)_s$ via these isomorphisms to the system of $i$.
    Passing to the direct limit morphism along $n \to \infty$ gives
    \[
        \Big(\colim_{n \to \infty} i \Big)
            \Big(\colim_{n \to \infty} K_{s-n+1} Q(n)_p \Big)
            \subseteq \colim_{n \to \infty} K_{s-n+1} Q(n)_{p'}.
    \]
    Thus the filtration is increasing:
    \begin{align*}
        F^p K_s A &=
        \Big(\colim_{n \to \infty} i(n,p)_s\Big)
            \Big( \colim_{n \to \infty} K_{s-n+1} Q(n)_p \Big)
        \\
        &\subseteq
        \Big(\colim_{n \to \infty} i(n,p')_s\Big)
            \Big( \colim_{n \to \infty} K_{s-n+1} Q(n)_{p'} \Big)
        \\
        &= F^{p'} K_s A. \qedhere
    \end{align*}
\end{proof}

\begin{lem}
    The $p$-indexed filtration
    of $K_* A$ from Definition \ref{dfn:filtrationOfLimit} is
    Hausdorff, exhaustive, and complete according to Definition
    \ref{dfn:Hausdorff}.
\end{lem}

\begin{proof}
    For $p < 0$, all terms $K_{s-n+1} Q(n)_p$
    vanish regardless of $s$ and $n$, thus their direct limit
    also vanishes. This renders the filtration Hausdorff and complete.

    For any class $[x] \in K_s A$, there is $n \in \nat$ such that $K_s A(n)$
    contains the preimage of $[x]$. We can choose $p = n$ to
    make $i(n,n)_s \colon K_{s-n+1} Q(n)_n \to K_s A(n)$ surjective, thereby
    including that preimage of $[x]$ in the range of $i(n,n)_s$. Commutativity
    of the diagram in Definition
    \ref{dfn:filtrationOfLimit} shows that $[x]$ is in
    the range of $\colim_n i(n,p)_s$. Thus the filtration is exhaustive.
\end{proof}

\begin{rem}
    The biggest problem in passing to direct limits for $n \to \infty$
    were the iterated suspensions $S^{n-1}$, but these have already been handled
    by the degree-amending links $\lambda(2n)$ from
    Definition \ref{dfn:degreeAmendingLink}
    via degree shifts and Bott isomorphisms.
    Thus throughout Section \ref{sec:sseqUnendlichConvergence},
    we may rest assured
    that any direct limits of modules or differentials on
    $E^r_{p,q}$ for $r \geq 1$
    or $r = \infty$ remain well-defined for the spectral sequence
    $\{E^r_{p,q}, d^r\}_{r,p,q}$ for countably many C*-ideals.
\end{rem}

\begin{pro}[Convergence of the limit spectral sequence]
\label{pro:unsereSseqUnendlichConvergence}
    Let $A$ be a C*-algebra with
    $A = \overline{I_0 + I_1 + \dotsb + I_j + \dotsb}$,
    a sum of C*-ideals.
    The limit spectral sequence constructed in
    Theorem \ref{thm:unsereSseqUnendlich}, defined
    as the direct limit along the system $\lambda(2n)$ with
    \[
        E^r_{p,q}
            = \colim_{n \to \infty}
                \Big(
                    \beta \circ \ell(2n+1) \circ \ell(2n)
                    \colon \{E(2n)^r_{p,q}\} \to \{E(2n+2)^r_{p,q}\}
                \Big),
    \]
    converges strongly to $K_*A$.
\end{pro}

\begin{proof}
    For all $n \in \nat$,
    the spectral sequence
    $\{E(n)^r_{p,q}, d(n)^r\}$ converges strongly:
    \[
        E(n)^\infty_{p,q} \cong F^p K_{p+q} A(n) / F^{p-1} K_{p+q} A(n).
    \]
    The terms $E^\infty_{p,q}$ are the direct limits of these
    $E(n)^\infty_{p,q}$ along the morphisms induced on $E(2n)^\infty_{p,q}$
    by the $\lambda(2n)$ because these $\lambda(2n)$ had all desirable
    properties -- naturality with respect to \starhoms{}
    that preserve countable ideal decompositions and commutativity with the
    differentials. Furthermore, K-theory is continuous. Along
    the system of $\lambda(2n)$,
    \[
        \colim_{n \to \infty}
        \big( F^p K_{p+q} A(2n) / F^{p-1} K_{p+q} A(2n) \big)
        \cong
        \colim_{n \to \infty}
        \big( E(2n)^\infty_{p,q}, \lambda(2n) \big)
        \cong
        E^\infty_{p,q}.
    \]
    The filtration $\{F^p K_* A\}_{p \in \zet}$ of $K_* A$
    is Hausdorff, exhaustive, and complete by Lemma
    \ref{lem:limitFiltrationHausdorff} and
    $E^\infty_{p,q}$ is isomorphic to
    $F^p K_{p+q} A / F^{p-1} K_{p+q} A$ by
    Lemma \ref{lem:continuityOfQuotients}.
    Therefore the limit spectral sequence
    $\{E^r_{p,q}, d^r\}_{r,p,q}$
    converges strongly to $K_* A$.
\end{proof}

This concludes the proof of Theorem
\ref{thm:unsereSseqUnendlich} about the existence, well-definedness,
functoriality, and strong convergence of the limit spectral sequence
for countably many C*-ideals.

\subsection{Uncountable sums of ideals}
\label{sec:uncountableSums}

In most geometrical applications, if a C*-algebra may be written as a sum
of easily computable ideals, this sum will be a countable sum.
We have described a spectral sequence for this case.
Still, it seems reasonable to generalize the cardinality of the algebra
decomposition.

\newcommand{\fins}{\mathop{\mathrm{Fin}}(\alpha)}

\restateAbstractSseqUncountable

This generalizes Theorem \ref{thm:unsereSseqUnendlich} about
countable $\alpha$.
To prove Theorem \ref{thm:uncountableAbstractSseq} for uncountable $\alpha$,
we will adapt our construction for countable index sets to suitable
direct limits that capture all of $\alpha$.

\begin{dfn}[Directed system of finite sets]
    For a set $\alpha$, we define its \emph{directed system of finite sets},
    \[
        \fins = \set{J \subs \alpha}{\card J \textrm{ is finite}};
    \]
    this system is partially ordered by the subset relation $\subs$.
\end{dfn}

\begin{rem}
    This is indeed a directed system: Given arbitrary $J$, $J' \in \fins$,
    both are equal to or smaller than their union $J \cup J'$, still a
    finite set.

    The direct limit of $\fins$ in the category of sets is $\alpha$.
    All chains, i.e., linear subsystems, in $\fins$
    are either finite or have the order type of ${(\nat, \leq)}$.

    We may consider the partially ordered set $\fins$ itself a thin category:
    Its elements become objects. Comparable sets
    $J \subs J'$ are linked with a unique morphism $J \to J'$.
\end{rem}

\begin{ntt}[Algebras for subsets $J \subs \alpha$]
    For the C*-algebra $A = \overline{\sum_{\beta \in \alpha} I_\beta}$
    as in the statement of Theorem \ref{thm:uncountableAbstractSseq}
    and $J \in \fins$, define a subalgebra $A(J)$ of $A$ by
    \[
        A(J) = \sum_{j \in J} I_j.
    \]
    Let $J' \in \fins$ be another subset with $J \subs J'$.
    For $s \in \zet$, let
    \[
        a(J, J')_s\colon K_s A(J) \to K_s A(J')
    \]
    be the map induced in K-theory by the inclusion of C*-algebras
    $A(J) \subs A(J')$.
\end{ntt}

\begin{rem}[Directed system in K-theory]
    For each $s \in \zet$, consider the functor from $\fins$
    to abelian groups that maps $J$ to $K_s A(J)$ and a comparable pair of
    sets $J \subs J'$ to $a(J, J')_s\colon K_s A(J) \to K_s A(J')$.
    This turns $\set{K_s A(J)}{J \in \fins}$ with the system
    of morphisms $\set{a(J, J')_s}{J \subs J' \in \fins}$
    into a directed system.
    Because K-theory is continuous,
    \[
        \colim_{J \in \fins} K_s A(J) = K_s A.
    \]
\end{rem}

\begin{ntt}[Spectral sequence for $J$]
    For $J \in \fins$, the finite sum of ideals $A(J)$ has a spectral
    sequence $\{E(J)^r_{p,q}, d(J)^r\}_{r,p,q}$
    according to Theorem \ref{thm:sseqforsums} that converges strongly to
    $A(J)$ and has first-page terms of the form
    \[
        E(J)^1_{p,q} \cong
            \begin{dcases*}
                \bigoplus_{\substack{\card L = p+1 \\ L \subs J}}
                    K_q \Big( \bigcap_{j \in L} I_j \Big)
                & for $0 \leq p < \card J$, \\
                0
                & for $p < 0$ or $p \geq \card J$,
            \end{dcases*}
    \]
    where $L$ ranges over all nonempty subsets of $J$.
\end{ntt}

\begin{rem}[Directed system of spectral sequences]
\label{rem:directedSystemOfJSseqs}
    Let $J \subs J' \in \fins$ be two sets such that
    $\card J$ and $\card{J'}$
    differ by an even number. Then the spectral sequences
    $\{E(J)^r_{p,q}, d(J)^r\}_{r,p,q}$ and $\{E(J')^r_{p,q}, d(J')^r\}_{r,p,q}$
    fit into a directed system of spectral sequences connected
    by degree-amending links shaped like
    $\lambda(2n)$ from Definition \ref{dfn:degreeAmendingLink}.
    These morphisms have bidegree $(0, 0)$.

    Let $F \subs \fins$ be the subsystem of $\fins$ of all sets
    $J \in \fins$ with even cardinality. Then
    $F$ and $\fins$ have the same direct limit $\alpha$
    in the category of sets.

    Consider the category of spectral sequences of the form
    $\{E(J)^r_{p,q}, d(J)^r\}_{r,p,q}$ for $J \in F$ with morphisms
    shaped like $\lambda(2n)$:
    Passing from $J \in \fins$ to the spectral sequence
    $\{E(J)^r_{p,q}, d(J)^r\}_{r,p,q}$ becomes a functor between directed
    systems.
\end{rem}

\begin{rem}
\label{rem:evenOddDoesNotMatter}
    It does not matter whether the sets $J \in F$ have even or odd
    cardinalities. As we have seen in the proof of
    Proposition \ref{pro:sseqUnendlichExistsNoLinkDef}, the limit
    spectral sequence along ${(\nat, \leq)}$
    does not depend on whether we consider the subsystem linked
    by $\lambda(2n)$ or that linked by $\lambda(2n+1)$.
\end{rem}

\begin{lem}
\label{lem:structureOfUncountableSseq}
    Let $F \subs \fins$ be the subsystem of $\fins$ of all sets
    $J \in \fins$ with even cardinality.
    Fix $p \in \zet$ with $p \geq 0$ and fix $q \in \zet$.

    Consider all groups $E(J)^1_{p,q}$ for
    $J \in F$: This is a directed system of abelian groups; the morphisms
    are restrictions of degree-amending links of the shape of $\lambda(2n)$
    from Definition \ref{dfn:degreeAmendingLink}.
    Then
    \[
        \colim_{J \in F}
            E(J)^1_{p,q}
            =
            \colim_{J \in F}
                \bigoplus_{\substack{\card L = p+1 \\ L \subs J}}
                    K_q \Big( \bigcap_{j \in L} I_j \Big)
            =
                \bigoplus_{\substack{\card L = p+1 \\ L \subs \alpha}}
                    K_q \Big( \bigcap_{j \in L} I_j \Big).
    \]
\end{lem}

\begin{proof}
    For $J \subs J' \in F$,
    the morphisms $E(J)^1_{p,q} \to E(J')^1_{p,q}$ are well-defined
    because degree-amending links
    have bidegree $(0, 0)$.

    These morphisms were defined as compositions of Bott
    isomorphisms and links (Definition \ref{dfn:linkBetweenSseqs})
    between spectral sequences; they preserve all information:
    Given $L \subs J$ of cardinality $p + 1$,
    the direct summand $K_q \big( \bigcap_{j \in L} I_j \big)$
    from $E(J)^1_{p,q}$
    maps isomorphically onto its copy in the direct sum $E(J')^1_{p,q}$.
    Given $L \subs J'$ such that $L$ is not a subset of $J$,
    the direct summand $K_q \big( \bigcap_{j \in L} I_j \big)$
    is not in the range.

    In the category of abelian groups, the direct limit may be constructed
    from a large direct sum of all objects, dividing by relations according
    to the morphisms. Here the links behave like inclusions,
    enforcing trivial relations.

    Finally, $F$ is cofinal in $\fins$: For any given set $J \subs \alpha$,
    the system $F$ contains a set $J'$ with $J \subs J'$.
    Therefore
    the desired direct limit is the direct sum taken over all subsets of
    $\alpha$ that have cardinality $p + 1$.
\end{proof}

With these preparations, we may now prove our main theorem.

\begin{proof}[Proof of Theorem \ref{thm:uncountableAbstractSseq}]
    Let $F \subs \fins$ be the subsystem of $\fins$ of all sets
    $J \in \fins$ with even cardinality.

    Consider the directed system of spectral sequences
    $\{E(J)^r_{p,q}, d(J)^r\}_{r,p,q}$ along all $J \in F$ from
    Remark \ref{rem:directedSystemOfJSseqs}.
    Lemma \ref{lem:structureOfUncountableSseq} guarantees that the page
    $E^1_{*,*}$ of the limit spectral sequence has the desired structure.

    All constructions in earlier sections behaved well with the differentials
    $d(J)^r$. Here in Section \ref{sec:uncountableSums}, we have applied
    various functorial direct limit constructions. Therefore the
    limit spectral sequence has the desired differentials.

    Likewise, functoriality of the spectral sequence with respect to
    ideal decompositions follows from all earlier sections and
    the functoriality of direct limits.

    Finally, we must prove the strong convergence.
    All direct limit results from
    Section \ref{sec:sseqUnendlichConvergence} about
    strong convergence along the directed system $(\nat, \leq)$ continue
    to hold for our directed system $F$ because all infinite chains
    in $F$ have the order type ${(\nat, \leq)}$:
    Whenever the symbol $n \in \nat$
    determines a K-theoretic degree shift in
    Section \ref{sec:sseqUnendlichConvergence}, this may be replaced
    with $n = \card J$ for $J \in F$.
    Even though the diagram in Definition \ref{dfn:filtrationOfLimit}
    of the morphism $\colim_n i(n,p)_s$ relies on the linearity
    of ${(\nat, \leq)}$, the construction itself is worded purely
    with direct limits: Objects in this category are morphisms of the form
    $i(n,p)_s$ and morphisms in this category are commutative diagrams.
    This construction does not require linearity of the
    underlying system, yet
    provides the desired filtration on $E^\infty_{*,*}$.
\end{proof}

\section{Infinite coarse excision}\thispagestyle{plain}

\newcommand{\letXbeCoarseCountable}{Let
    $(X, d)$ be a coarse space and
    let $\{X_\beta\}_{\beta \in \alpha}$
    be a countable coarsely excisive cover of $(X, d)$.
}

\subsection{Direct limits of coarse algebras}

The spectral sequence for infinite sums of C*-ideals allows us to strengthen
Theorem \ref{thm:sseqForFiniteCoarselyExcisiveCovers}, the spectral sequence
for finite coarsely excisive covers, to infinite coarsely excisive covers.

\begin{pro}
\label{pro:theCoarseDirectLimit}
    Let $(X, d)$ be a coarse space and
    let $\{X_\beta\}_{\beta \in \alpha}$
    be a coarsely excisive cover of $(X, d)$. The index set $\alpha$
    may be finite, countably infinite, or uncountable.
    \letFualalgBeEither
    For each $\beta \in \alpha$, consider the C*-ideal
    $\fualalg X_\beta \cong \fualalg(X_\beta \subs X)$ of $\fualalg X$.

    Then the direct limit of finite sums of ideals,
    \begin{equation}
    \label{eqn:theCoarseDirectLimit}
        \overline{\bigcup_{\substack{
            J \subs \alpha
            \\ \card J \in \nat
        }}
        \Big( \sum_{j \in J} \fualalg(X_j \subs X) \Big)},
    \end{equation}
    is a C*-ideal of $\fualalg X$.
\end{pro}

\begin{proof}
    Let $A$ denote the C*-algebra in expression
    \ref{eqn:theCoarseDirectLimit}.

    For each $\beta \in \alpha$, the algebra
    $\fualalg(X_\beta \subs X)$ is a C*-ideal in $\fualalg X$.
    The inclusion
    $\fualalg(X_\beta \subs X) \to \fualalg X$ factors through any finite
    sum $\sum_{j \in J} \fualalg(X_j \subs X)$ when
    $\beta \in J$, and
    the inclusion of that finite sum in $\fualalg X$ factors again
    through $A$; thus the closed $A$ is a sub-C*-algebra of $\fualalg X$.

    To check that $A$ is a C*-ideal in $\fualalg X$, it remains
    to show that $A$ is an algebraic two-sided ideal.
    Given $a \in A$ and $b \in \fualalg X$,
    find a sequence of finite sets $(J_n)_{n \in \nat}$ with $J_n \subs \alpha$
    and a sequence $(a_n)_{n \in \nat}$ with
    $a_n \in \sum_{j \in J_n} \fualalg (X_j \subs X)$ that
    converges to $a$. Finite sums of C*-ideals $\fualalg(X_j \subs X)$
    are again
    C*-ideals, thus both $(a_n b)_{n \in \nat}$ and $(b a_n)_{n \in \nat}$
    stay within the closed $A$. These sequences converge in $A$ to
    $ab$ and $ba$ respectively because multiplication is continuous.
\end{proof}

\begin{rem}
\label{rem:infiniteDecompsArentFullAlgebra}
    For finite decompositions, the
    direct limit of finite sums from
    expression \ref{eqn:theCoarseDirectLimit}
    equals $\fualalg X$ by Theorem \ref{thm:connectTwoWorlds}.
\end{rem}

\subsection{Corollaries for coarsely excisive covers}

\restatSseqCoarseUncountable

\begin{proof}
    Apply the spectral sequence from Theorem \ref{thm:uncountableAbstractSseq}
    about arbitrary sums of abstract C*-algebras to the algebras from
    Theorem \ref{thm:connectTwoWorlds} for coarse spaces.

    By Theorem \ref{thm:uncountableAbstractSseq}, the spectral sequence
    converges strongly to
    the K-theory of the norm closure of finite sums of the input ideals.
    These ideals are
    $\fualalg X_\beta \cong \fualalg (X_\beta \subs X)$ and we have described
    the norm closure of their finite sums in Proposition
    \ref{pro:theCoarseDirectLimit}.

    The special case for finite coarsely excisive covers
    follows from Remark
    \ref{rem:infiniteDecompsArentFullAlgebra}.
\end{proof}

\begin{rem}[Warning about uncountable decompositions]
    Let $(X, d)$ be a coarse space. The important C*-algebras
    $\cualalg X$, $\dualalg X$, and $\qualalg X$ are defined via very ample
    representations $\varrho\colon \conto X \to BH$ for a separable Hilbert
    space $H$. Separability of the Hilbert space is crucial for several
    isomorphism theorems. Ampleness of the representation guarantees that
    no two functions $f \neq f' \in \conto X$ may be represented on that
    separable Hilbert space by operators
    $\varrho(f)$ and $\varrho(f')$ that differ only by a compact operator.

    If the coarsely excisive cover $\{X_\beta\}_{\beta \in \alpha}$
    of $(X, d)$ has an uncountable index set $\alpha$, the separability
    requirement may force $\fualalg X_\beta = \fualalg X_{\beta'}$ for
    many $\beta \neq \beta'$, or
    may force outright triviality of ideals.
    It appears hard to construct interesting examples for
    uncountable coarse excision.
\end{rem}

\newcommand{\xcap}{{(X_0 \cap X_1)}}

With Theorem \ref{thm:sseqCoarseUncountable},
we can strengthen the following Mayer-Vietoris result.
Let $(X, d)$ be a coarse space and $X_0$, $X_1 \subseteq X$ such that
$\{X_0, X_1\}$ is a coarsely excisive cover.
All rows and columns in the following diagram are long exact sequences;
the index $j$ in $\bigoplus_j$ ranges over $\{0,1\}$:
\[
    \begin{tikzpicture}
    \mamax{0.3cm}{
        \dotsb &
        K_{s+1} \cualalg X &
        K_{s} \cualalg \xcap &
        \bigoplus_j K_s \cualalg X_j &
        K_s \cualalg X &
        \dotsb
    \\
        \dotsb &
        K_{s+1} \dualalg X &
        K_{s} \dualalg \xcap &
        \bigoplus_j K_s \dualalg X_j &
        K_s \dualalg X &
        \dotsb
    \\
        \dotsb &
        K_{s+1} \qualalg X &
        K_{s} \qualalg \xcap &
        \bigoplus_j K_s \qualalg X_j &
        K_s \qualalg X &
        \dotsb
    \\
        \dotsb &
        K_{s} \cualalg X &
        K_{s-1} \cualalg \xcap &
        \bigoplus_j K_{s-1} \cualalg X_j &
        K_{s-1} \cualalg X &
        \dotsb.
    \\};
    \path[->, font=\scriptsize]
        (m-1-1) edge (m-1-2)
        (m-1-2) edge (m-1-3)
        (m-1-3) edge (m-1-4)
        (m-1-4) edge (m-1-5)
        (m-1-5) edge (m-1-6)
        (m-2-1) edge (m-2-2)
        (m-2-2) edge (m-2-3)
        (m-2-3) edge (m-2-4)
        (m-2-4) edge (m-2-5)
        (m-2-5) edge (m-2-6)
        (m-3-1) edge (m-3-2)
        (m-3-2) edge (m-3-3)
        (m-3-3) edge (m-3-4)
        (m-3-4) edge (m-3-5)
        (m-3-5) edge (m-3-6)
        (m-4-1) edge (m-4-2)
        (m-4-2) edge (m-4-3)
        (m-4-3) edge (m-4-4)
        (m-4-4) edge (m-4-5)
        (m-4-5) edge (m-4-6)
        (m-1-2) edge (m-2-2)
        (m-2-2) edge (m-3-2)
        (m-3-2) edge (m-4-2)
        (m-1-3) edge (m-2-3)
        (m-2-3) edge (m-3-3)
        (m-3-3) edge (m-4-3)
        (m-1-4) edge (m-2-4)
        (m-2-4) edge (m-3-4)
        (m-3-4) edge (m-4-4)
        (m-1-5) edge (m-2-5)
        (m-2-5) edge (m-3-5)
        (m-3-5) edge (m-4-5)
    ;
    \end{tikzpicture}
\]
The columns are exact by definition of $\cualalg$, $\dualalg$,
and $\qualalg$.
Commutativity follows from the naturality of the Mayer-Vietoris sequence.

\begin{pro}
    Let $(X, d)$ be a coarse space and
    let $\{X_\beta\}_{\beta \in \alpha}$ be a
    coarsely excisive cover of $(X, d)$.
    Consider the following diagram with exact columns; horizontal
    arrows are induced by (direct sums of) inclusions, and unions
    range over all finite index sets $J \subs \alpha$:
    \[
        \newcommand{\limitAlg}[1]{
            \big( \overline{ \bigcup_{J}
            \sum_{j \in J} #1 (X_j \subs X) } \big)
        }
        \begin{tikzpicture}
        \mama{
            \bigoplus_{\beta \in \alpha} K_s \cualalg X_\beta &
            K_s \limitAlg{\cualalg} &
            K_s \cualalg X
        \\
            \bigoplus_{\beta \in \alpha} K_s \dualalg X_\beta &
            K_s \limitAlg{\dualalg} &
            K_s \dualalg X
        \\
            \bigoplus_{\beta \in \alpha} K_s \qualalg X_\beta &
            K_s \limitAlg{\qualalg} &
            K_s \qualalg X
        \\
            \bigoplus_{\beta \in \alpha} K_{s-1} \cualalg X_\beta &
            K_{s-1} \limitAlg{\cualalg} &
            K_{s-1} \cualalg X.
        \\};
        \path[->, font=\scriptsize]
            (m-1-1) edge (m-1-2)
            (m-2-1) edge (m-2-2)
            (m-3-1) edge (m-3-2)
            (m-4-1) edge (m-4-2)
            (m-1-2) edge (m-1-3)
            (m-2-2) edge (m-2-3)
            (m-3-2) edge (m-3-3)
            (m-4-2) edge (m-4-3)
            (m-1-1) edge (m-2-1)
            (m-2-1) edge (m-3-1)
            (m-3-1) edge (m-4-1)
            (m-1-2) edge (m-2-2)
            (m-2-2) edge (m-3-2)
            (m-3-2) edge (m-4-2)
            (m-1-3) edge (m-2-3)
            (m-2-3) edge (m-3-3)
            (m-3-3) edge (m-4-3)
        ;
        \end{tikzpicture}
    \]
    Then this diagram commutes.
\end{pro}

\begin{proof}
    The columns are exact again by definition of $\cualalg$, $\dualalg$,
    and $\qualalg$.
    Commutativity follows from continuity of K-theory -- the algebras in the
    center column are direct limits -- and from
    functoriality of the spectral sequence in Theorem
    \ref{thm:unsereSseqUnendlich} with respect to morphisms
    between C*-algebras; here, direct sums of inclusion morphisms.
\end{proof}

\subsection{The coarse space
    \texorpdfstring{$\zet^\infty$}{Z-infty}}
\label{subsection:zinfty}

Consider the free $\zet$-module $\zet^\infty = \bigoplus_\nat \zet$
of $\nat$-indexed tuples $(x_0, x_1, \ldots, x_n, \ldots)$
with only finitely many entries
different from 0. This space can be metrized in different ways.

\begin{dfn}[Weight functions, weighted 1-metric]
    Let $w\colon \nat \to \rel_{>0}$ be an arbitrary function, for example
    \[
        w\colon n \mapsto n+1, \qquad
        w\colon n \mapsto 1, \textrm{ or} \qquad
        w\colon n \mapsto \frac1{n+1}.
    \]
    We call $w$ a \emph{weight function}.
    Given a weight function $w$, define a metric $d_w$ on $\zet^\infty$:
    \[
        d_w\big((x_0, x_1, \ldots x_n, 0, 0, \ldots),
                (y_0, y_1, \ldots, y_{n'}, 0, 0, \ldots)\big)
        = \sum_{j=0}^\infty w(j)\left|x_j - y_j\right|.
    \]
\end{dfn}

\begin{exm}
    For the constant weight function $w\colon n \mapsto 1$, the metric
    $d_w$ coincides with the usual 1-metric $d_1$.
\end{exm}

\begin{rem}[Topological properties of $\zet^\infty$]
    For $w\colon n \mapsto n+1$, the metric $d_w$ is proper:
    With minimum distance $k+1$ between points in the
    $k$-th dimension, any ball of finite diameter is
    finite, thus compact. For $w\colon n \mapsto 1$ or
    $w\colon n \mapsto \frac1{n+1}$, closed $d_w$-balls with
    finite radius larger than 1 are not compact anymore.
    Under $w\colon n \mapsto \frac1{n+1}$,
    the space $(\zet^\infty, d_w)$ is not even locally compact.
\end{rem}

More than with the topological properties of these spaces, we are concerned
with their coarse properties.

The identity on any coarse space is a coarse map:
Choose $S = R$ in Definition \ref{dfn:coarseMap}.
But the identity $\zet^\infty \to \zet^\infty$ fails to be a coarse map
when the two spaces are metrized according to two different weight functions
among
$n \mapsto n+1$,
$n \mapsto 1$, and
$n \mapsto \frac1{n+1}$.
The identity ceases to be uniformly expansive:
Points with distance $1$ in dimension $n$ may have distance $n+1$
or even $(n+1)^2$ in the target space. No constant $S > 0$ can serve as an
upper bound across all dimensions $n$.

If the identity fails as a coarse equivalence, can other maps substitute?
The answer is no, for a similar reason:

\begin{pro}[Coarse properties of $\zet^\infty$]
    The weight functions $n\mapsto n+1$ and $n \mapsto 1$
    generate different coarse structures on
    $\zet^\infty$; i.e., there is no coarse equivalence
    according to Definition \ref{dfn:coarseEquivalence}.
\end{pro}

\begin{proof}
    Assume that there is a pair of coarse equivalences
    \begin{align*}
        f\colon& \left(\zet^\infty, d_{n \mapsto n+1}\right)
            \to  {\left(\zet^\infty, d_{n \mapsto 1}\right)},\\
        g\colon& \left(\zet^\infty, d_{n \mapsto 1}\right)
            \to  {\left(\zet^\infty, d_{n \mapsto n+1}\right)}.
    \end{align*}
    Then let $S > 0$ satisfy all of these conditions:
    \begin{itemize}
        \item For all $x \in \zet^\infty$, we have
            $d_{n \mapsto n+1}(x, gfx) \leq S$, this is possible
            because $f$ and $g$ are a pair of coarse equivalences.
        \item Whenever $x$, $x' \in \zet^\infty$ with
            $d_{n \mapsto 1}(x, x') \leq 1$, then
            $d_{n \mapsto n+1}(gx, gx') \leq S$, this is possible because
            $g$ is a coarse map.
        \item For simplicity, $S \in \nat$.
    \end{itemize}
    Let $x$ and $x'$ have $d_{n \mapsto 1}(x, x') \leq 1$.
    Then $d_{n \mapsto n+1}(gx, g0) \leq S$
    and in all dimensions
    $S$, $(S+1)$, $(S+2)$, $\ldots$, the coordinates of $g(x)$ and $g(x')$ must
    be identical.

    Any two points $y$, $y'$ can be linked by a finite sequence of hops
    between two points each with
    $d_{n \mapsto 1}$-distance $\leq 1$. By induction,
    $g(y)$ and $g(y')$ agree in all coordinates from dimension $S$ onwards.

    Let $z \in (\zet^\infty, d_{n \mapsto n+1})$ have different coordinates
    than $g(y)$ in dimension $S$.
    Then $(g \circ f)(z)$ and $z$ have distance at least $S+1$.
    This is not allowed when $f$ and $g$ are coarse equivalences.
\end{proof}

\begin{rem}
    When we replace the weight function $n \mapsto 1$ with
    $n \mapsto \frac{1}{n+1}$, the same argument shows that
    the coarse structure induced by that weight function is not equivalent
    to $(\zet^\infty, d_{n \mapsto n+1})$ either.
\end{rem}


\begin{ntt}
    Let $X \subseteq \zet^\infty$ be a set. For a metric $d_w$ as above
    and $R > 0$, we write
    $N_w(X, R)$ instead of $N_{d_w}(X, R)$ for the $R$-neighborhood
    of $X$ under the metric $d_w$ according to Definition
    \ref{dfn:rneighborhood}.
\end{ntt}

\begin{lem}
\label{lem:zinfty}
    For all three weight functions $w\colon n \mapsto n+1$,
    $w\colon n \mapsto 1$, and $w\colon n \mapsto \frac1{n+1}$,
    the decomposition $\{X_j\}_{j \in \nat}$ with
    \[
        X_j = \zet_{\geq 0}^j \times \zet_{\leq 0}
            \times \zet^\infty
    \]
    defines a coarsely excisive cover of $(\zet^\infty, d_w)$.
\end{lem}

The block decomposition of $\rel^n$ from Definition
\ref{dfn:blockDecomposition} was similar, but covered only a
finite-dimensional space.

\begin{proof}[Proof of Lemma \ref{lem:zinfty}]
    \newcommand{\capj}{\bigcap_{j \in J} X_j}
    Consider a finite nonempty index set $J \subseteq \nat$ and
    the finite subcollection $\{X_j\}_{j \in J}$ of the decomposition
    $\{X_j\}_{j \in \nat}$. We have to show that, for $R > 0$, there
    exists $S > 0$ with
    \[
        \bigcap_{j \in J} N_w(X_j, R)
        \subseteq N_w\Big(\bigcap_{j \in J} X_j, S\Big).
    \]
    Let $n$ denote the highest index in $J$.
    We may write
    \[
        \bigcap_{j \in J} X_j
            = Y_0 \times Y_1 \times \ldots \times Y_{n-1}
            \times \zet_{\leq 0} \times \zet^\infty,
    \]
    where $Y_j = \{0\}$ for $j \in J$, otherwise $Y_j = \zet_{\geq 0}$.
    For easier notation, we will write $Y_n$ for $\zet_{\leq 0}$.

    We are interested in the $R$-neighborhood of each $X_j$
    for $R > 0$ and in the
    $S$-neighborhood of the intersection for a suitable $S$.
    Let $x = (x_0, x_1, x_2, \ldots)$ be a point in $\zet^\infty$.
    The coordinates after the $n$-th coordinate do not matter anymore: For
    the intersection $\capj$, the distance of $x$ to $\capj$
    depends only on the early coordinates up to the $n$-th.

    Since the metric $d_w$ on $\zet^\infty$ is a weighted 1-metric --
    distance is the weighted sum of the dimension-wise distances -- it makes
    sense to decompose $(\zet^\infty, d_w)$ into a product of metric
    spaces, $\zet^{n+1} \times \zet^\infty$, of the first $n+1$
    dimensions and the remainder $\zet^\infty$ that is irrelevant
    for the chosen $J$. Let $p$ and $q$ be the projections for this
    product decomposition,
    \begin{align*}
        p\colon \zet^\infty \to \zet^{n+1},\phantom{\zet^\infty} \,
        &p(x_0, x_1, x_2, \ldots) = (x_0, x_1, \ldots, x_n),\\
        q\colon \zet^\infty \to \zet^\infty,\phantom{\zet^{n+1}} \,
        &q(x_0, x_1, x_2, \ldots) = (x_{n+1}, x_{n+2}, \ldots).
    \end{align*}
    Both projections admit one-sided inverse
    by padding all other coordinates with
    zeros. We pull back the metric $d_w$ to either factor along these inverses.
    For $x \in \zet^\infty$ and a space $Y$ that is either
    $Y = X_j$ for a $j \in J$ or $Y = \capj$, we then have
    \begin{equation}
    \label{eqn:onlyFiniteDimensionsMatter}
        d_w(x, Y) =
        d_w\big(px, p(Y)\big)
        + \underbrace{d_w\big(qx, q(Y)\big)}_{
            \textrm{= 0}
        };
    \end{equation}
    the rightmost summand vanishes because
    both $q(X_j)$ and $q\big(\capj\big)$ are the entire range $q(\zet^\infty)$
    by construction.
    On the finite-dimensional remainder $\zet^{n+1}$, the weight function $w$
    admits an upper and a lower bound: There is a constant
    $M \geq n+1$ such that $\frac1M \leq w(k) \leq M$ for all $k < n+1$.
    For $x$, $y \in \zet^{n+1}$,
    Lemma \ref{lem:block1sup} shows
    \begin{equation}
    \label{eqn:mmm}
        d_w(x, y)
        \leq M d_1(x, y)
        \leq M d_\infty(x, y)
        \leq M^2 d_1(x, y)
        \leq M^3 d_w(x, y).
    \end{equation}
    The finite-dimensional space $\zet^{n+1}$
    embeds isometrically into $\rel^{n+1}$. This embedding is
    not necessarily the
    inclusion $\zet^{n+1} \subseteq \rel^{n+1}$; rather, depending on $w$,
    its image lattice is scaled differently per dimension.
    Still, we can now apply Proposition
    \ref{pro:blockDecompSupIsExcisive} for the $\sup$-metric $d_\infty$
    on $\rel^{n+1}$:
    \[
        \bigcap_{j \in J} N_\infty\big(p(X_j), R\big) =
        N_\infty\Big( \bigcap_{j \in J} p(X_j), R \Big).
    \]
    Together with \ref{eqn:onlyFiniteDimensionsMatter} and \ref{eqn:mmm},
    we conclude that $S = M^3R$ certifies the desired coarse excisiveness
    in the infinite-dimensional space $\zet^\infty$:
    \[
        \bigcap_{j \in J} N_w(X_j, R)
        \subseteq N_w\Big( \bigcap_{j \in J} X_j, M^3R \Big).
        \qedhere
    \]
\end{proof}

\begin{pro}
    Let $\{X_j\}_{j \in \nat}$ be the coarsely excisive cover
    of $\zet^\infty$ as in
    Lemma \ref{lem:zinfty} and let $A$ denote the direct limit C*-algebra
    \[
        A
        =
        \overline{\bigcup_{
            \substack{J \subs \nat \\ \card J \in \nat}}
            \sum_{j \in J}
            \cualalg(X_j \subs \zet^\infty)}
        =
        \overline{\sum_{j \in \nat} \cualalg(X_j \subs \zet^\infty)}.
    \]
    Under any of the three considered weight functions,
    the algebra $A$ then has trivial K-theory:
    $K_* A = 0$.
\end{pro}

\begin{proof}
    We may use our spectral sequence from Theorem
    \ref{thm:unsereSseqUnendlich} because the cover $\{X_j\}_{j \in \nat}$
    is coarsely excisive.

    Each $X_j$ is flasque (Definition \ref{dfn:flasque})
    because it contains the flasque factor
    $\zet_{\leq 0}$.

    For finite $J \subseteq \nat$,
    the intersection $\bigcap_{j \in J} X_j$ is again flasque:
    Let $n$ be the largest index in $J$. We may describe this intersection
    by listing its one-dimensional factors: First, there are $n$ factors that
    we may ignore. Then there is the flasque factor
    $\zet_{\leq 0}$.
    All further factors form a copy of $\zet^\infty$.
    Because of the flasque factor, we conclude that
    $K_* \cualalg \big(\bigcap_{j \in J} X_j\big) = 0$.

    Since this holds for all finite $J \subseteq \nat$, Theorem
    \ref{thm:unsereSseqUnendlich} gives a spectral
    sequence with first page $E^1_{*,*} = 0$, converging to
    $K_* A = 0$.
\end{proof}

The countable intersection $\bigcap_{j \in \nat} X_j$ is a single
point, but infinite intersections do not appear in the spectral sequence.

\subsection{Wedge sum of rays}

Let $\alpha$ be either $\nat$ or a finite cardinality with $\alpha > 0$.
Let $X_\beta = (\rel_{\geq 0}, 0)$ be a ray for all $\beta \in \alpha$, pointed
at the origin. Define
\[
    X = \bigvee_{\beta \in \alpha} X_\beta,
\]
a finite or countable wedge sum of the half-open rays glued together at their
origins. Define a coarse structure on $X$ by the metric
\[
    d(x, y) = \begin{cases}
        |x - y| & \textrm{if } x \textrm{ and } y
                    \textrm{ lie in the same ray,}\\
        |x| + |y| & \textrm{otherwise.}
    \end{cases}
\]
If local compactness were desired, we could remove the common point
$0$. This would not change any large-scale properties of $X$. But in
this example, local compactness is irrelevant.

\begin{lem}
\label{lem:wedgeSumYBeta}
    Let $Y_\beta = X_0 \cup X_\beta$ for all $\beta \in \alpha$.
    Cover $X$ by $\set{Y_\beta}{\beta \in \alpha}$.

    This is a coarsely excisive cover of the wedge sum $X$.
\end{lem}

\begin{proof}
    We have to check: For all finite $\{Y_{\beta(0)}, Y_{\beta(1)}, \ldots,
    Y_{\beta(n-1)}\}$ and all $R > 0$,
    there exists $S > 0$ such that the $n$-fold intersection of the
    $R$-neighborhoods lies in the $S$-neighborhood of the intersection:
    \[
        \bigcap_{j<n} N_d(Y_{\beta(j)}, R) \subseteq
        N_d{\Big(\bigcap_{j<n} Y_{\beta(j)}, S\Big)}.
    \]
    For $n = 1$, this is trivial with $S = R$.
    For $n > 1$, consider $n$ pairwise different $Y_{\beta(j)}$:
    The intersection $\bigcap_{j<n} Y_{\beta(j)}$ is always $Y_0 = X_0$.

    The common point $0$ joins all rays $X_\beta$, its $R$-neighborhood
    $N_d(\{0\}, R)$
    is therefore the union of the intervals $[0, R]$ from all rays $X_\beta$.
    Thus:
    \begin{align*}
        \bigcap_{j<n} N_d(Y_{\beta(j)}, R)
            &= X_0 \cup N_d(\{0\}, R),
        \\
        N_d\Big(\bigcap_{j<n} Y_{\beta(j)}, S\Big)
            = N_d(X_0, S) &= X_0 \cup N_d(\{0\}, S).
    \end{align*}
    Choose $S = R$ to see that
    the cover $\{Y_\beta\}_{\beta \in \alpha}$ is coarsely excisive.
\end{proof}

\newcommand{\wedgeBunchZ}{{
    \bigoplus\limits_{\substack{\beta \in \alpha \\ \beta \neq 0}} \zet
}}

\begin{pro}
    For the wedge sum $X = \bigvee_{\beta \in \alpha} X_\beta$ with
    each $X_\beta = (\rel_{\geq 0}, 0)$ a ray pointed at the origin
    and the coarsely excisive cover $\{Y_\beta\}_{\beta \in \alpha}$
    from Lemma \ref{lem:wedgeSumYBeta},
    let $A$ be the direct limit C*-algebra of sums
    $\sum_{j \in J} \cualalg(Y_j \subs X)$ over finite sets
    $J \subs \alpha$. Then
    \[
        K_s A = \begin{cases}
            0 & \textrm{for } s \textrm{ even,} \\
            \wedgeBunchZ & \textrm{for } s \textrm{ odd.} \\
        \end{cases}
    \]
\end{pro}

\begin{proof}
    The set $Y_0$ remains a flasque ray.
    Each other $Y_\beta$ is coarsely equivalent
    to $\rel$ and therefore has K-theory
    $K_0\cualalg Y_n = 0$ and $K_1 \cualalg Y_n = \zet$.
    This determines the column $E^1_{0,*}$ of the first page.

    Each finite intersection of at least two different $Y_n$ is
    the flasque space $Y_0 = X_0$.
    The K-theory of its Roe algebra vanishes.
    Therefore the $E^1_{*,*}$-term looks as follows:
    
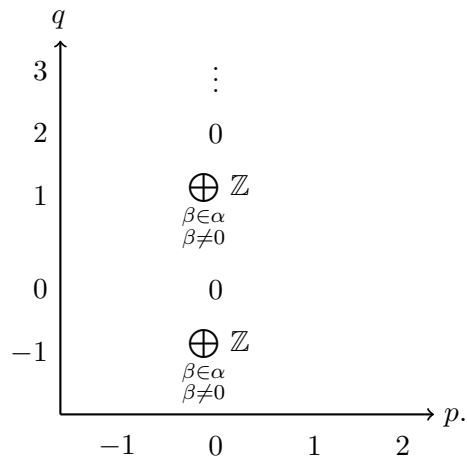
\[
        \begin{tikzpicture}
        \masseq{
            \phantom{-}3 & & \vdots & & \\
            \phantom{-}2 & & 0      & & \\
            \phantom{-}1 & & \wedgeBunchZ & & \\
            \phantom{-}0 & & 0      & & \\
            -1 & & \wedgeBunchZ & & \\
            & -1  &  0  &  1  & 2 \\
            };
        \drawSseqAxes{5}{6}
        \end{tikzpicture}
    \]
    This spectral sequence collapses on the first page.
    We may read the K-theory of $A \subs \cualalg X$ from the only
    nonzero column: If $\alpha$ is countably infinite,
    the dimension of the free $\zet$-module
    in odd degrees is countably infinite;
    if $\alpha$ is finite, the dimension is $\alpha - 1$.
\end{proof}

For finite $\alpha$, an alternative proof to compute $A = \cualalg X$
by induction repeats the Mayer-Vietoris principle $\alpha - 1$
times for two-fold coverings: Glue
single rays, one after another, to the wedge sum that starts with a single
ray.

\section{Generalizations}\thispagestyle{plain}

\subsection{KO-theory}

Instead of K-theory of C*-algebras over $\com$, we may examine
\emph{KO-theory} for a C*-algebra $A$ over $\rel$, denoted $KO_*A$.
All basic definitions carry over without change, turning $KO_*$ into
a $\zet$-graded covariant continuous functor into abelian groups.

The major difference is the degree of the Bott isomorphism:
Instead of $K_s A \cong K_{s+2} A$, real Bott periodicity admits a natural
isomorphism $\beta\colon KO_s A \cong KO_{s+8} A$ for all $s \in \zet$.
As a result, for C*-ideals $I \subseteq A$, the six-term exact sequence
\[
    \dotsb \to K_0 I \to K_0 A \to K_0 (A/I) \loto{\partial_2 \circ \beta}
        K_1 I \to K_1 A \to K_1 (A/I) \loto{\partial_1} K_0 I \to \dotsb
\]
becomes a 24-term exact sequence in the real case:
\[
    \dotsb \to KO_0 A \to KO_0 (A/I) \loto{\partial_8 \circ \beta} KO_7 I \to
        KO_7 A \to KO_7 (A/I) \loto{\partial_7} KO_6 I \to \dotsb.
\]
Looking back to the constructions of the various spectral sequences,
we relied on the Bott isomorphism merely for the construction
of the degree-amending links
$\lambda(n)\colon E(n)^r_{p,q} \to E(n+2)^r_{p,q}$ from Definition \ref{dfn:degreeAmendingLink}. These morphisms
and Proposition \ref{pro:sseqUnendlichExistsNoLinkDef} about the existence
of the limit spectral sequence
can be adapted to work with KO-theory: Define
\[
    \lambda_\rel(n)\colon E(n)^r_{p,q} \to E(n+8)^r_{p,q}
\]
by chaining 8 links $\{\ell(n)^r_{p,q}\}_{r,p,q}$
from Definition \ref{dfn:linkBetweenSseqs}
-- instead of only
2 such links in the case of K-theory -- with the real Bott isomorphism.
The direct limit along these $\lambda_\rel(8n)$ for $n \to \infty$ does
not depend on the position of the Bott isomorphism within the chain
nor on whether we consider
the directed system along $\lambda_\rel(8n)$,
$\lambda_\rel(8n+1)$, $\ldots$, or $\lambda_\rel(8n+7)$.

Convergence is proven as in Section \ref{sec:sseqUnendlichConvergence}
and generalized to uncountable ideal decompositions
as in Section \ref{sec:uncountableSums}.
This yields a well-defined spectral sequence that computes the KO-theory
of C*-ideal sums: The statement from
Theorem \ref{thm:uncountableAbstractSseq} holds when we replace K-theory
with KO-theory.

\subsection{Group actions}

Let $(X, d)$ be a metric space. Let $G$ be a countable discrete group
that acts on $X$ freely and properly by $d$-isometries.
This action extends to $\conto X$ via $(gf)(x) = f(g^{-1}x)$ for
$g \in G$, $f \in \conto X$, and $x \in X$.
In addition to a very ample representation $\varrho\colon \conto X \to BH$
for a separable Hilbert space $H$, let $U\colon G \to H$ be a unitary
representation with $U(g) \varrho(f) = \varrho(gf) U(g)$.

This gives rise to C*-algebras $C^*_G X$ and $D^*_G X$ by
changing the usual definitions of $\cualalg X$ and $\dualalg X$: The
norm closure is taken of only the $G$-invariant
operators in $BH$ that satisfy all other requirements of $\cualalg X$
and $\dualalg X$, respectively.
Furthermore, define
$Q^*_G X = D^*_G X / C^*_G X$.

In \cite[Definition 3.6]{siegel-mv}, for a $G$-invariant subspace $Y \subs X$
that we may assume to be closed, P. Siegel constructs the relative
C*-algebras $C^*_G(Y \subs X)$ and $D^*_G(Y \subs X)$ by imposing
on operators in $C^*_G X$ and $D^*_G X$ the
conditions from Section \ref{sec:relativeRoeAlgebras}
for support near $Y$ and local compactness outside $Y$.
Finally, define $Q^*_G(Y \subs X)$ as the quotient
$D^*_G(Y \subs X) / C^*_G(Y \subs X)$.
There is a long exact sequence for $s \in \zet$,
\[
    \dotsb
    \to K_{s} C^*_G X
    \to K_{s} D^*_G X
    \to K_{s} Q^*_G X
    \to K_{s-1} C^*_G X
    \to \dotsb.
\]
Furthermore, $Q^*_G X \cong \qualalg X_G$. The sequence may thus
be rewritten with the K-homology of $X_G$ instead of $Q^*_G X$
according to Remark \ref{rem:kHomology}.

Let the functor $F^*_G$ stand for either $C^*_G$, $D^*_G$, or $Q^*_G$.
According to \cite[Propositions 3.8, 3.9]{siegel-mv}, for
closed $G$-invariant subspaces $Y \subs X$ and $s \in \zet$, we have
\[
    K_s F^*_G (Y \subs X) \cong K_s F^*_G Y.
\]
For $G$-invariant coarsely excisive covers $\{X_0, X_1\}$ of $X$, we have
\begin{align*}
    F^*_G (X_0 \subs X) + F^*_G (X_1 \subs X) &= F^*_G X,
    \\
    F^*_G (X_0 \subs X) \cap F^*_G (X_1 \subs X)
    &= F^*_G (X_0 \cap X_1 \subs X).
\end{align*}
This leads to a Mayer-Vietoris exact sequence.
We may expect a generalization
of our spectral sequence to arbitrary $G$-invariant
coarsely excisive covers $\{X_\beta\}_{\beta \in \alpha}$ of $X$.
The definitions of $C^*_G(Y \subs X)$ and $D^*_G(Y \subs X)$ treat the
$G$-invariance in the least intrusive way possible. There should be no
difficulty in adapting
the equations to finite selections $J \subs \alpha$ of the coarsely
excisive cover:
\begin{align*}
    K_s \Big( \sum_{j \in J} F^*_G (X_j \subs X) \Big)
    &\cong
    K_s F^*_G \Big( \bigcup_{j \in J} X_j \Big),
    \\
    K_s \Big( \bigcap_{j \in J} F^*_G (X_j \subs X) \Big)
    &\cong
    K_s F^*_G \Big( \bigcap_{j \in J} X_j\Big),
\end{align*}
with $F^*_G$ standing for $C^*_G$, $D^*_G$, or $Q^*_G$.
The constructions will be similar to those leading to our
Theorem \ref{thm:connectTwoWorlds} for these equations where $G$ is trivial.

This provides a $G$-invariant version of our spectral sequence
from Theorem \ref{thm:sseqCoarseUncountable} for coarsely excisive
covers.

    \newpage
    \bibliography{res/z_bibtex}
    \bibliographystyle{alpha}

\if{0}
    \cleardoublepage
    \thispagestyle{empty}
    \pagenumbering{gobble}
    \singlespacing
    \input{./res/09-ende/03-lebenslauf.tex}
\fi
\end{document}